\documentclass[11pt]{compositio}
\usepackage{amsmath}
\usepackage{bm,mathrsfs}
\usepackage{hyperref}
\hypersetup{colorlinks=true}

\newcommand{\map}[1]{\xrightarrow{#1}}
\newcommand{\iso}{\xrightarrow{\simeq}}

\newcommand{\Hom}{\mathrm{Hom}}
\newcommand{\Aut}{\mathrm{Aut}}
\newcommand{\End}{\mathrm{End}}
\newcommand{\Spec}{\mathrm{Spec}}
\newcommand{\Q}{\mathbb Q}
\newcommand{\Z}{\mathbb Z}
\newcommand{\R}{\mathbb R}
\newcommand{\C}{\mathbb C}
\newcommand{\F}{\mathbb F}
\newcommand{\A}{\mathbb A}
\newcommand{\co}{\mathcal O}
\newcommand{\ord}{\mathrm{ord}}

\newcommand{\Lie}{\mathrm{Lie}}

\newcommand{\siegel}{\mathrm{Siegel}}

\newcommand{\action}{\bullet}

\newcommand{\GSpin}{\mathrm{GSpin}}
\newcommand{\SO}{\mathrm{SO}}
\newcommand{\GSp}{\mathrm{GSp}}

\newcommand{\SL}{\mathrm{SL}}
\newcommand{\Fil}{\mathrm{Fil}}

\newcommand{\kk}{{\bm{k}}}

\newcommand{\dR}{\mathrm{dR}}
\newcommand{\crys}{\mathrm{crys}}

\newcommand{\Be}{\mathrm{Be}}

\input xy
\xyoption{all}

\begin{document}

\title{Height pairings on orthogonal Shimura varieties}

\author{Fabrizio Andreatta}
\email{fabrizio.andreatta@unimi.it}
\address{Dipartimento di Matematica ``Federigo Enriques", Universit\`a di Milano, via C.~Saldini 50, Milano, Italia}

\author{Eyal Z. Goren}
\email{eyal.goren@mcgill.ca}
\address{Department of Mathematics and Statistics, McGill University, 805 Sherbrooke St. West, Montreal, QC, Canada}

\author{Benjamin Howard}
\email{howardbe@bc.edu}
\address{Department of Mathematics, Boston College, 140 Commonwealth Ave,
Chestnut Hill, MA, USA}

\author{Keerthi Madapusi Pera}
\email{keerthi@math.uchicago.edu}
\address{Department of Mathematics, University of Chicago, 5734 S University Ave, Chicago, IL, USA}

\shortauthors{F.~Andreatta, E. Z.~Goren, B.~Howard, K.~Madapusi Pera}

\thanks{F.~Andreatta is supported by the Italian grant Prin 2010/2011. E. Z.~Goren is supported by NSERC discovery grant, B.~Howard is supported by NSF grant DMS-1201480. K.~Madapusi Pera is supported by NSF Postdoctoral Research Fellowship DMS-1204165.}


\begin{abstract}
Let $M$ be the Shimura variety associated to the group of spinor similitudes of a  quadratic space over $\Q$ of signature $(n,2)$.
We prove a conjecture of Bruinier and Yang, relating the arithmetic intersection  multiplicities of special divisors
and CM points on $M$  to the central derivatives of  certain  $L$-functions.
Each such $L$-function  is the Rankin-Selberg  convolution associated with a  cusp form of half-integral weight
$n/2 +1 $,  and the weight $n/2$ theta series of a positive definite quadratic space of rank $n$.

When  $n=1$ the Shimura variety  $M$ is a classical quaternionic Shimura curve, and our result  is a variant of the Gross-Zagier theorem on heights of Heegner
points.
\end{abstract}

\maketitle

\setcounter{tocdepth}{1}
\tableofcontents

\theoremstyle{plain}
\newtheorem{theorem}{Theorem}[subsection]
\newtheorem{proposition}[theorem]{Proposition}
\newtheorem{lemma}[theorem]{Lemma}
\newtheorem{corollary}[theorem]{Corollary}
\newtheorem{conjecture}[theorem]{Conjecture}
\newtheorem{bigtheorem}{Theorem}

\theoremstyle{definition}
\newtheorem{definition}[theorem]{Definition}
\newtheorem{hypothesis}[theorem]{Hypothesis}

\theoremstyle{remark}
\newtheorem{remark}[theorem]{Remark}

\numberwithin{equation}{section}
\renewcommand{\thebigtheorem}{\Alph{bigtheorem}}
\renewcommand{\labelenumi}{(\roman{enumi})}

\section{Introduction}

In this paper we prove  a generalization, conjectured by Bruinier-Yang \cite{BY},  of the Gross-Zagier theorem
on heights of Heegner points on modular curves.

Instead of a modular curve we work on the  Shimura variety defined by a reductive group over $\Q$ of type $\GSpin(n,2)$.
 The Heegner points are replaced by  $0$-cycles arising from embeddings of rank two tori into $\GSpin(n,2)$, and by
divisors arising from embeddings of  $\GSpin(n-1,2)$ into $\GSpin(n,2)$.
We prove that the arithmetic intersections of these special cycles are related
to the central derivatives of certain Rankin-Selberg convolution $L$-functions.

\subsection{The Shimura variety}

Fix an integer $n\ge 1$ and  a quadratic space $(V,Q)$ over $\Q$ of signature $(n,2)$.  From $V$
one can construct a Shimura datum $(G,\mathcal{D})$, in which  $G=\GSpin(V)$ is a reductive group over $\Q$ sitting in an exact sequence
\[
1 \to \mathbb{G}_m \to \GSpin(V) \to \SO(V) \to 1,
\]
and the hermitian domain $\mathcal{D}$ is an open subset of the space of isotropic lines in
$\mathbb{P}^1(V_\C)$.

Any choice of lattice $L\subset V$ on which $Q$ is $\Z$-valued determines a compact open subgroup
$K\subset G(\A_f)$.  We fix such a lattice, and assume that $L$ is maximal, in the sense that it admits no superlattice on which $Q$ is
$\Z$-valued.    The data $L\subset V$ now determines a complex orbifold
\[
M(\C) = G(\Q) \backslash \mathcal{D} \times G(\A_f) / K
\]
which, by the theory of canonical models of Shimura varieties, is the space of complex points of an algebraic
stack $M\to \Spec(\Q)$.  (Throughout this article \emph{stack} means \emph{Deligne-Mumford stack}.)

Except for small values of $n$, the Shimura variety $M$ is not of PEL type; that is, it is not naturally a moduli space of abelian varieties with polarization,
endomorphisms, and level structures.  It is, however, of  Hodge type, and so the recent work of Kisin \cite{kisin} (and its extension in \cite{mp:reg} and
\cite{mp:2adic}) provides us with a regular and flat integral model
\[
\mathcal{M} \to \Spec(\Z).
\]

\subsection{Special divisors}

The integral model carries over it  a  canonical family of abelian varieties:  the \emph{Kuga-Satake abelian scheme}
\[
\mathcal{A} \to \mathcal{M}.
\]
The Kuga-Satake abelian scheme is endowed with an action of the Clifford algebra $C(L)$, and with  a $\Z/2\Z$-grading $\mathcal{A}=\mathcal{A}^+\times
\mathcal{A}^-$.

 For any scheme $S\to\mathcal{M}$, the pullback  $\mathcal{A}_S$ comes with  a distinguished $\Z$-module
\[
V( \mathcal{A}_S ) \subset \End (\mathcal{A}_S)
\]
 of  \emph{special endomorphisms}, and  there is a positive definite quadratic form
 \[
 Q: V( \mathcal{A}_S ) \to \Z
 \]
 characterized by $x\circ x = Q(x) \cdot  \mathrm{Id}$.
Slightly more generally, there is a distinguished subset
\[
V_\mu( \mathcal{A}_S ) \subset V( \mathcal{A}_S )  \otimes_\Z\Q
\]
for each coset $\mu \in L^\vee /L$, where $L^\vee$ is the dual lattice of $L$.
Taking $\mu=0$ recovers $V( \mathcal{A}_S )$.

The special endomorphisms allow us  to define a family of \emph{special divisors} on $\mathcal{M}$.
For $m\in \Q_{>0}$ and $\mu \in L^\vee/L$, let $\mathcal{Z}(m,\mu) \to \mathcal{M}$ be the moduli stack
that  assigns to any $\mathcal{M}$-scheme  $S \to \mathcal{M}$ the set
\[
\mathcal{Z}(m,\mu)(S) = \{ x\in V_\mu( \mathcal{A}_S)   : Q(x)=m \} .
\]
The morphism $\mathcal{Z}(m,\mu) \to \mathcal{M}$ is relatively representable, finite, and unramified,
and allows us to view $\mathcal{Z}(m,\mu)$ as a Cartier divisor on $\mathcal{M}$.

In the generic fiber of $\mathcal{M}$ the special divisors  agree with the
divisors  appearing   in \cite{Borcherds},  \cite{Bruinier}, \cite{BY},  and
\cite{Ku}.  In special cases where~$V$ has signature $(1,2)$, $(2,2)$,  or $(3,2)$,  the Shimura variety $\mathcal{M}$
is  (a quaternionic  version of) a modular curve, Hilbert modular surface,  or Siegel threefold, respectively.
In these cases the special divisors  are traditionally known as  \emph{complex multiplication points},  \emph{Hirzebruch-Zagier divisors},
and \emph{Humbert surfaces}.    See, for example,   \cite{VDG}, \cite{KR1},  \cite{KR2}, and \cite{KRY3}.

\subsection{CM points}

Let $V_0\subset V$ be a negative definite plane in $V$,  and set $L_0 = V_0 \cap L$.
 The Clifford algebra $C(L_0)$ is an order in a quaternion algebra over $\Q$, and its even
part $C^+(L_0) \subset C(L_0)$ is an order in a quadratic
imaginary field ~$\kk$. Fix an embedding $\kk \to \C$, and  assume
from now on that $C^+(L_0)=\co_\kk$ is the maximal order in $\kk$.

The group $\GSpin(V_0) \cong  \mathrm{Res}_{\kk/\Q}\mathbb{G}_m$
is a rank two torus, and has associated to it a $0$-dimensional
Shimura variety  $Y \to \Spec(\kk)$, which can be reinterpreted as
the moduli stack of elliptic curves with CM by $\co_\kk$. This
allows us to construct a smooth integral model $\mathcal{Y}\to
\Spec(\co_\kk)$. If $\kk$ has odd discriminant the inclusion $V_0
\to V$ induces an embedding $\GSpin(V_0) \to G$ and a relatively
representable, finite, and unramified morphism
\[
 i\colon \mathcal{Y} \to \mathcal{M}_{\co_\kk}.
\]
  The stack $\mathcal{Y}$ has its own Kuga-Satake abelian scheme $\mathcal{A}_0 \to \mathcal{Y}$, endowed with
  a $\Z/2\Z$-grading $\mathcal{A}_0 \cong \mathcal{A}_0^+ \times \mathcal{A}_0^-$, and an action
of $C(L_0)$. Here $\mathcal{A}_0^+$ is the universal elliptic curve with CM by $\co_\kk$ and $\mathcal{A}_0\cong \mathcal{A}_0^+\otimes_{\co_\kk} C(L_0)$.  It is related to the Kuga-Satake abelian scheme $\mathcal{A} \to \mathcal{M}$ by a $C(L)$-linear isomorphism
\[
\mathcal{A}\vert_{ \mathcal{Y} } \cong \mathcal{A}_0 \vert_ {
\mathcal{Y}}  \otimes_{C(L_0)} C(L).
\]

Let  $\widehat{\mathrm{Pic}}( \mathcal{M} )$ denote the group  of
metrized line bundles on $\mathcal{M}$, and similarly with
$\mathcal{M}$ replaced by $\mathcal{M}_{\co_\kk  }$ or
$\mathcal{Y}$.  The composition
\[
\widehat{\mathrm{Pic}}( \mathcal{M} ) \to \widehat{\mathrm{Pic}}(
\mathcal{M}_{\co_\kk  })  \map{i^*} \widehat{\mathrm{Pic}}(
\mathcal{Y}) \map{ \widehat{\deg} } \R
\]
is denoted  $\widehat{\mathcal{Z}} \mapsto [ \widehat{\mathcal{Z}} : \mathcal{Y}]$.
We call it  \emph{arithmetic degree along $\mathcal{Y}$}.

\subsection{Harmonic modular forms}

Denote by $\widetilde{\SL}_2(\Z)$ the metaplectic double cover of $\SL_2(\Z)$, and by
\[
\omega_L\colon \widetilde{\SL}_2(\Z) \to \Aut_\C(\mathfrak{S}_L)
\]
the Weil representation on the space $\mathfrak{S}_L$ of $\C$-valued functions on $L^\vee/L$.
There is a complex conjugate representation $\overline{\omega}_L$ on the
same space, and a contragredient action $\omega^\vee_L$ on the dual space $\mathfrak{S}_L^\vee$.

Bruinier and Funke \cite{BruinierFunke} have defined a surjective conjugate-linear differential operator
\[
\xi\colon H_{ 1- \frac{n}{2}} (\omega_L) \to S_{1+ \frac{n}{2}} (\overline{\omega}_L).
\]
Here  $S_{1+ \frac{n}{2}} (\overline{\omega}_L)$ is the space of weight $1+ \frac{n}{2}$ cusp forms
valued in $\mathfrak{S}_L$, and  transforming according to the representation $\overline{\omega}_L$.
The space $H_{ 1- \frac{n}{2}} (\omega_L)$ is defined similarly, but the forms
are \emph{harmonic weak Mass forms} in the sense of \cite[\S 3.1]{BY}.

To any $f\in H_{ 1- \frac{n}{2}} (\omega_L)$ there is associated a formal $q$-expansion
\[
f^+(\tau) = \sum_{ \substack{m\in \Q \\ m\gg -\infty} }     c_f^+(m ) q^m \in \mathfrak{S}_L [[q]],
\]
and each coefficient is uniquely a linear combination $\sum_\mu c_f^+(m,\mu) \varphi_\mu$
of the characteristic functions $\varphi_\mu$ of cosets $\mu\in L^\vee / L$.  Assuming that $f^+$
has \emph{integral principal part}, in the sense that $c_f^+(m,\mu) \in \Z $ whenever $m<0$, we may form
the Cartier divisor
\[
\mathcal{Z}(f) = \sum_{ m > 0} \sum_{\mu \in L^\vee / L } c_f^+( - m , \mu ) \cdot  \mathcal{Z}(m,\mu)
\]
(the sum is finite)  on the integral model $\mathcal{M}$,
following  the work of Borcherds \cite{Borcherds}, Bruinier \cite{Bruinier}, and Bruinier-Yang \cite{BY}.

Using the formalism of regularized theta lifts developed in  \cite{Borcherds},  Bruinier \cite{Bruinier} defined a Green function $\Phi(f)$ for
$\mathcal{Z}(f)$.  This Green function determines a metric on the corresponding line bundle, yielding a metrized line bundle
\[
\widehat{\mathcal{Z}} (f) \in \widehat{\mathrm{Pic}} ( \mathcal{M}) .
\]

\subsection{The main result}

The central problem is to compute the  arithmetic degree $[
\widehat{\mathcal{Z}} (f) : \mathcal{Y}]$. Assuming that the
divisor $\mathcal{Z}(f)_{\co_\kk}$ intersects the cycle
$\mathcal{Y}$ properly, the arithmetic degree  decomposes as a sum
of local contributions.  The calculation of the archimedean
contribution, which is essentially the sum of the values of
$\Phi(f)$ at all $y\in\mathcal{Y}(\C)$, is the main result of
\cite{BY}.

Based on their calculation of  the archimedean  contribution, Bruinier and Yang formulated a conjecture \cite[Conjecture 1.1]{BY}  relating  $[
\widehat{\mathcal{Z}} (f)  :  \mathcal{Y}]$ to the derivative of an $L$-function. Our main result is a proof of this conjecture.

The statement of the result requires two more ingredients.  The first is the \emph{metrized cotautological bundle}
\[
\widehat{ \bm{T} } \in \widehat{\mathrm{Pic} }( \mathcal{M}).
\]
By definition, the hermitian domain $\mathcal{D}$ is an open subset of the space of isotropic lines in
$\mathbb{P}^1(V_\C)$.  Restricting the tautological bundle
on $\mathbb{P}^1(V_\C)$ yields a line bundle $\omega_\mathcal{D}$ on $\mathcal{D}$, which  descends first to the orbifold $M(\C)$,
and then to the canonical model $M$.  The resulting  line bundle, $\omega$,  is the \emph{tautological bundle} on $M$,
also called the \emph{line bundle of weight one modular forms}.    There is a natural metric (\ref{taut metric}) on $\omega$,
 and our $\widehat{ \bm{T} }$ is an integral model of  $\widehat{\omega}^{-1}$.

The second ingredient is a classical vector valued theta series.  From the $\Z$-quadratic space
\[
\Lambda = \{ x\in L : x \perp L_0 \}
\]
of signature $(n,0)$, one can construct  a half-integral weight theta series
\[
\Theta_\Lambda (\tau)  \in M_{\frac{n}{2}}(\omega_\Lambda^\vee).
\]
Here  $\omega_\Lambda$ is the Weil representation of $\widetilde{\SL}_2(\Z)$ on the space $\mathfrak{S}_\Lambda$ of
functions on $\Lambda^\vee / \Lambda$, and $\omega_\Lambda^\vee$ is its contragredient.   See \S \ref{ss:analytic} for the
precise definition.

The following appears in the text as Theorem \ref{thm:mainthm}.

\begin{bigtheorem}\label{bigthm:A}
Assume that the discriminant  of $\kk$ is odd, and let $h_\kk$ and $w_\kk$ be the class number
and number of roots of unity in $\kk$, respectively.
Every  weak harmonic Maass form $f\in H_{1-n/2}(\omega_L)$ with integral principal part  satisfies
\[
[ \widehat{\mathcal{Z}} (f) : \mathcal{Y} ]  + c_f^+(0,0)  \cdot [ \widehat{\bm{T}}  : \mathcal{Y} ]
= - \frac{h_\kk}{w_\kk}  \cdot L'(  \xi(f) ,\Theta_\Lambda,0) ,
\]
where  $c_f^+(0,0)$ is the value of   $c_f^+(0) \in \mathfrak{S}_L$ at the trivial coset in $L^\vee/L$,
and $L(  \xi(f) ,\Theta_\Lambda,s)$ is the Rankin-Selberg convolution $L$-function of  \S \ref{ss:analytic}.
\end{bigtheorem}

\begin{remark}
The analogous theorem for Shimura varieties associated with unitary similitude groups was proved in
\cite{BHY}.
\end{remark}

\begin{remark} \label{rem:eisenstein intro}
The convolution $L$-function appearing in the theorem can be realized by integrating $\xi(f)$ and $\Theta_\Lambda$ against
a certain Eisenstein series $E_{L_0} (\tau ,s)$.   This Eisenstein series vanishes at $s=0$, and hence the same is true of
$L(  \xi(f) ,\Theta_\Lambda,s)$.
\end{remark}

\begin{remark}
As noted earlier, our cycles $\mathcal{Y} \to \mathcal{M}_{\co_\kk}$ arise from embeddings of rank two tori into $\GSpin(n,2)$.  One could instead look at the cycles defined by embeddings of maximal tori; these are the \emph{big CM cycles} of \cite{BKY}. In \cite{AGHMP} we prove a result analogous to Theorem \ref{bigthm:A} in this setting. We show that it implies an averaged form of Colmez's conjecture on the Faltings heights of CM abelian varieties (via the method of Yang \cite{Yang-colmez}). This has provided a key ingredient in J. Tsimerman's recent proof of the Andr\'e-Oort conjecture \cite{Tsimerman} , which characterizes Shimura subvarieties of Shimura varieties of Hodge type by means of special points.
\end{remark}

\begin{remark}
If one is willing to have the equality of Theorem \ref{bigthm:A} only up to rational multiples of $\log(2)$, in principle it should be
possible to drop the assumption that $d_\kk$ is odd or  that
$C^+(L_0)$ is maximal and replace it by the weaker assumption that
$C^+(L_0)$ is maximal away from $2$. Such improvements come at a
cost: under these weaker hypotheses one does not know anything
like the explicit formulas appearing in Proposition
\ref{prop:eisenstein coeff} and Theorem \ref{thm:X degree}.

Instead, one should be able (we have not checked the details)  to prove the main result by using the Siegel-Weil arguments of \cite{BKY} and \cite{KY}.  In essence,
the Siegel-Weil formula allows one to compare the point counting part of Theorem \ref{thm:X degree} with the values of certain Whittaker functions appearing in the
Eisenstein series coefficients of Proposition \ref{prop:eisenstein coeff}, without deriving explicit formulas for either one.
\end{remark}

A rough outline of the contents is as follows:

In \S \ref{s:shimura varieties} we establish the
basic properties of the Shimura variety $\mathcal{M}$ on which we will be performing intersection theory,
and of the special cycles $\mathcal{Z}(m,\mu)$.    We prove that the $\mathcal{Z}(m,\mu)$ define Cartier
divisors on $\mathcal{M}$, and explore  the functoriality of the formation of $\mathcal{M}$ and its divisors
with respect to the quadratic space $L$.
More precisely, given an isometric embedding $L_0 \hookrightarrow L$ of maximal $\Z$-quadratic spaces
of signatures $(n_0,2)$ and $(n,2)$, with $1\le n_0 \le n$, we prove the existence of a canonical morphism
$\mathcal{M}_0 \to \mathcal{M}$ between the associated integral models, and show that the pullback of a special divisor on $\mathcal{M}$
is a prescribed linear combination of special divisors on $\mathcal{M}_0$.

In \S \ref{s:analysis} we remind the reader of some of the analytic theory  used in \cite{BY}.  In
particular, we recall the essential properties of harmonic weak Maass forms,  the divisors $Z(f)$ on $M$ associated to such forms,  and the Green functions
$\Phi(f)$ for these divisors, which are defined as regularized theta lifts.   We also recall the theta series and convolution $L$-functions that appear in the main
theorem, and  Schofer's \cite{Sch} calculation of the Fourier coefficients of the Eisenstein series of Remark \ref{rem:eisenstein intro}.  The only thing new here
is that, thanks to the constructions of \S  \ref{s:shimura varieties}, we are able to define an extension  of the divisor $Z(f)$ on $M$  to a divisor
$\mathcal{Z}(f)$ on the integral model $\mathcal{M}$.

In \S \ref{s:CM section} we fix a quadratic space $L_0$ over $\Z$ of signature $(0,2)$, whose even Clifford algebra is
isomorphic to the maximal order in a quadratic imaginary field $\kk$.  The Clifford algebra of $L_0$  then has the form
$C(L_0) \iso \co_\kk \oplus L_0$.  The Shimura variety
associated to the  rank two torus  $\GSpin(L_{0\Q})$ has dimension~$0$, and can be realized as the moduli stack of
elliptic curves with complex multiplication by $\co_\kk$.  Using this interpretation we define an integral model
$\mathcal{Y} \to \Spec(\co_\kk)$, and define the Kuga-Satake abelian scheme as $\mathcal{A}_0 = \mathcal{E} \otimes_{\co_\kk} C(L_0)$,
where $\mathcal{E} \to \mathcal{Y}$ is the universal CM  elliptic curve.   All of the results of
\S \ref{s:shimura varieties} are then extended to  this new setting.

In particular, we define  a family of divisors
$\mathcal{Z}_0(m,\mu)$ on  the arithmetic curve $\mathcal{Y}$, each of which can be characterized as the locus of points where $\mathcal{A}_0$
has an extra quasi-endomorphism with prescribed properties.
This can also be expressed in terms of endomorphisms of the universal CM elliptic curve:
along  each special divisor the action of $\co_\kk$ on $\mathcal{E}$ extends to an action of an order in a   definite quaternion algebra.
Theorem \ref{thm:X degree} gives explicit formulas for the degrees of the divisors $\mathcal{Z}_0(m,\mu)$, and these
degrees match up with Schofer's formulas for the coefficients of  Eisenstein series.     Many cases of Theorem \ref{thm:X degree}
already appear in work of Kudla-Rapoport-Yang \cite{KRY1}  and Kudla-Yang \cite{KY}, and our proof, like theirs,
makes essential use of Gross's calculation \cite{Gr} of the endomorphism rings of canonical lifts of CM elliptic curves.

In \S \ref{s:main section} we prove Theorem \ref{bigthm:A}.   The
greatest difficulty comes from the fact that the $1$-cycle
$\mathcal{Y}$ on $\mathcal{M}_{\co_\kk}$ may have irreducible
components contained entirely within
$\mathcal{Z}(m,\mu)_{\co_\kk}$.  Thus we must compute improper
intersections.  Our techniques for doing this are a blend of the
methods used in \cite{BHY} on unitary Shimura varieties, and the
methods of Hu's thesis \cite{Hu}, which reconstructs the
arithmetic intersection theory of Gillet-Soul\'e \cite{GS} using
Fulton's method of deformation to the normal cone.

The cycle $\mathcal{Y}$ has a normal bundle $N_\mathcal{Y}\mathcal{M} \to \mathcal{Y}$,
and the arithmetic divisor   $\widehat{\mathcal{Z}}(f)$  on $\mathcal{M}$ has a \emph{specialization to the normal bundle}
$\sigma (\widehat{\mathcal{Z}}(f) )$, which is an arithmetic divisor  on $N_\mathcal{Y}\mathcal{M}$.    The specialization is defined in such a way that the arithmetic degree  of
$\widehat{\mathcal{Z}}(f)$ along $\mathcal{Y}\to \mathcal{M}$ is equal to the arithmetic degree  of $\sigma (\widehat{\mathcal{Z}}(f) )$
along the zero section $\mathcal{Y} \to N_\mathcal{Y}\mathcal{M}$.  What we are able to show, essentially,
is that every component of the arithmetic divisor $\sigma (\widehat{\mathcal{Z}}(f) )$ that meets  the zero section improperly
has $\sigma(\widehat{\bm{T}})$, the specialization of the cotautological bundle, as its associated line bundle.  This implies that
every  component  of $\widehat{\mathcal{Z} } (f)$ that meets $\mathcal{Y}$ improperly contributes the same quantity,
$[ \widehat{\bm{T}} : \mathcal{Y}]$, to  the degree $[ \widehat{\mathcal{Z}}(f) : \mathcal{Y} ]$.
The  contribution to the intersection  of the remaining components can be read off from the
degrees of the divisors $\widehat{\mathcal{Z}}_0(m,\mu)$ computed in \S \ref{s:CM section},
and the calculation of $[ \widehat{\bm{T}} : \mathcal{Y}]$   quickly reduces to the Chowla-Selberg formula.

Finally, we remark that the proof of Theorem \ref{bigthm:A} makes essential use of the Bruinier-Yang calculation
(following Schofer \cite{Sch})
of the values of $\Phi(f)$ at the points of $\mathcal{Y}(\C)$.   Because of the particular way in which the
divergent integral defining $\Phi(f)$ is regularized, the Green function $\Phi(f)$ has a well-defined value even at points of
divisor $\mathcal{Z}(f)$ along which $\Phi(f)$ has a logarithmic singularity.  In other words, the Green function
 comes, by construction, with a discontinuous extension to all points of $\mathcal{M}(\C)$.
 Hu's thesis sheds some light  on this phenomenon, and the \emph{over-regularized} values  of $\Phi(f)$ at the points of
 $\mathcal{Z}(f)$ will play an essential role in our calculation of  improper intersections.

\subsection{Acknowledgements}

The authors thank  Jan Bruinier, Jos\'e Burgos Gil, Steve Kudla, and Tonghai Yang for helpful conversations and insights. They also thank the referees for useful comments.


\section{The GSpin Shimura variety and its special divisors}
\label{s:shimura varieties}



\subsection{Preliminaries}
\label{s:preliminaries}


Let $R$ be a commutative ring.  A \emph{quadratic space} over $R$ is  a projective $R$-module $V$ of finite rank,
equipped with a homogenous function $Q\colon V\to R$ of degree $2$ for which the  symmetric pairing
\[
[x,y] = Q(x+y) -Q(x) -Q(y)
\]
is $R$-bilinear.  We say that $V$ is \emph{self-dual} if this pairing induces an isomorphism
$V \iso \Hom_R(V,R)$.

The \emph{Clifford algebra} $C(V)$ is the quotient of the tensor algebra $\bigotimes V$
by the ideal generated by elements of
the form $x \otimes x -Q(x)$.  The $R$-algebra $C(V)$ is generated by the image
of the natural injection $V\to C(V)$, and the  grading on $\bigotimes V$ induces a $\Z/2\Z$ grading
\[
C(V)=C^+(V) \oplus C^-(V).
\]

The Clifford algebra $C(V)$ has the following universal property:  given an associative algebra $B$ over $R$ and a map of $R$-modules $f:V\to B$ satisfying $f(x)^2 = Q(x)$  for all $x\in V$, there is a unique map of $R$-algebras $\tilde{f}:C(V)\to B$ satisfying $\tilde{f}\vert_V = f$. In particular, applying this property with $B=C(V)^{\rm{op}}$ and $f$ the natural inclusion of $V$ in $C(V)^{\rm{op}}$, we obtain a canonical anti-involution $*$ on $C(V)$ satisfying
\[
 (x_1x_2\cdots x_{r-1} x_r)^* = x_rx_{r-1}\cdots x_2x_1
\]
for any $x_1,x_2,\ldots,x_{r-1},x_r\in V$.

To a self-dual quadratic space $V$, we can attach the reductive group scheme $\GSpin(V)$ over $R$ with functor of points
\[
\GSpin(V) (S)= \{ g\in  C^+(V)_S ^\times : g  V_S  g^{-1} = V_S\}
\]
for any $R$-algebra $S$.  Denote by $g\action v = gvg^{-1}$ the natural action of $\GSpin(V)$ on $V$.
For any $R$-algebra $S$, the map $g\mapsto g^*g$ on $\GSpin(V)(S)$ takes values in $S^\times$~\cite[3.2.1]{bass:quadratic}, and so defines a canonical homomorphism of $R$-group schemes
\[
\nu:\GSpin(V)\to\mathbb{G}_{m,R},
\]
which we call the \emph{spinor norm}.

Choose an element $\delta\in C^+(V)^\times$ satisfying $\delta^* = -\delta$. We can use this element to define a symplectic pairing on $C(V)$ as follows (see \cite[(1.6)]{mp:reg}).  The self-duality hypothesis on $V$ implies that $C(V)$ is a graded Azumaya algebra over $R$ in the sense of \cite[\S 2.3, Theorem]{bass:quadratic}. In particular, there is a canonical reduced trace form $\mathrm{Trd}:C(V) \to R$, and the pairing $(z_1,z_2)\mapsto\mathrm{Trd}(z_1z_2)$ is a perfect symmetric bilinear form on $C(V)$. The desired  symplectic pairing on $C(V)$ is $\psi_{\delta}(z_1,z_2) = \mathrm{Trd}(z_1\delta z_2^*)$.

It is easily checked that
\[
 \psi_{\delta}(gz_1,gz_2) = \nu(g)\psi_{\delta}(z_1,z_2).
\]
for any $R$-algebra $S$, $z_1,z_2\in C(V)_S$, and $g\in\GSpin(V)(S)$.
In other words, the action of $\GSpin(V)$ on $C(V)$ by left multiplication
realizes  $\GSpin(V)$ as a subgroup of the $R$-group scheme $\GSp_{\delta}$ of symplectic similitudes of $(C(V),\psi_{\delta})$, and
 the symplectic similitude character  restricts to the spinor norm on $\GSpin(V)$.


\subsection{The complex Shimura variety}
\label{s:the complex Shimura variety}


Let $n\ge 1$ be an integer.
Fix a quadratic space $(L,Q)$ over $\Z$ of signature $(n,2)$, and set $V=L_\Q$.
We  assume throughout that $L$ is \emph{maximal} in the sense that there is no lattice $L\subsetneq L' \subset  V$
satisfying $Q(L')\subset \Z$.  Set $G=\GSpin(V)$; this is a reductive group over $\Q$.

Let $\mathcal{D}$ be the space of oriented negative $2$-planes $\bm{h}\subset V_\R$.
If we extend the $\Q$-bilinear form on $V$ to a $\C$-bilinear form on $V_\C$, the real manifold $\mathcal{D}$
is isomorphic to the Grassmannian
\begin{equation}\label{grassman}
\mathcal{D} \iso  \left\{ z \in V_\C \smallsetminus \{ 0\} :
 [z,z] =0 \mbox{ and }  [z,\overline{z}] <0  \right\} / \C^\times ,
\end{equation}
and in particular is naturally a complex manifold.
The isomorphism is as follows: for each $\bm{h}$ pick an oriented basis $\{x,y\}$ so that $Q(x)=Q(y)$ and
$[x,y]=0$, and let  $z=x+iy$. The inverse construction is obvious.

There are canonical inclusions of $\R$-algebras $\C \hookrightarrow  C(\bm{h})  \hookrightarrow  C(V_\R)$, where the first is
determined by
\[
i\mapsto \frac{ xy}{\sqrt{Q(x) Q(y)} }.
\]
The induced map $\C^\times \to C(V_\R)^\times$ takes values in $G(\R)$, and arises from a morphism of real algebraic groups
\begin{equation}\label{Eqn_alpha_h}
 \alpha_{\bm{h}}:\mathrm{Res}_{\C/\R}\mathbb{G}_m  \to G_\R.
\end{equation}
In this way we identify $\mathcal{D}$ with a $G(\R)$-conjugacy class in
 $\Hom(\mathrm{Res}_{\C/\R}\mathbb{G}_m, G_\R)$.

Set $\widehat{L} = L_{\widehat{\Z}}$, and define a compact open subgroup
\[
K = G(\A_f) \cap C\big(\widehat{L}\big)^\times
\]
of $G(\A_f)$.  Note that  $g\action \widehat{L} =\widehat{L}$ for all $g\in K$.
 Denote by
\[
L^\vee= \{ x\in V : [x, L]\subset \Z \}
\]
the dual lattice of $L$.   As in the proof of \cite[(2.6)]{mp:reg}, the action  of $K$ on $\widehat{L}$
defines a homomorphism $K\to \SO(\widehat{L})$ whose image is precisely the subgroup of elements
acting trivially on the discriminant group $L^\vee/L \iso \widehat{L}^\vee/\widehat{L}$.

The pair $(G,\mathcal{D})$ is a Shimura datum, and
\begin{equation}\label{uniformization}
M(\C) = G(\Q) \backslash \mathcal{D} \times G(\A_f) / K
\iso  \bigsqcup_{ g\in G(\Q)\backslash G(\A_f) /K } \Gamma_g \backslash \mathcal{D},
\end{equation}
 is the complex  orbifold of $\C$-points of the Shimura variety attached to this datum and level subgroup $K$.
 Here we have set $\Gamma_g =  G(\Q) \cap gKg^{-1}.$
We will refer to $M(\C)$ as the \emph{GSpin Shimura variety} associated with $L$.

Given an algebraic representation $G \to \underline{\mathrm{Aut}}(N)$ on a finite dimensional $\Q$-vector space, and  a $K$-stable lattice $N_{\widehat{\Z}}\subset N_{\A_f}$,
we obtain a  $\Z$-local system $\bm{N}_{\rm Be}$ on $M(\C)$ whose fiber at a point $[(\bm{h},g)]\in M(\C)$ is identified with
$N\cap g N_{\widehat{\Z}}$. The corresponding vector bundle \[\bm{N}_{\dR,M(\C)}=\co_{M(\C)}\otimes\bm{N}_{\rm Be}\]
is equipped with a filtration $\Fil^\bullet\bm{N}_{\dR,M(\C)}$, which at any point $[(\bm{h},g)]$ equips the fiber of $\bm{N}_{\rm Be}$
with the Hodge structure determined by \eqref{Eqn_alpha_h}. This gives us a functorial assignment from pairs $(N,N_{\widehat{\Z}})$ as above to variations of $\Z$-Hodge structures over $M(\C)$.

Applying this to $V$ and the lattice $\widehat{L}\subset V_{\A_f}$, we obtain a canonical variation of polarized $\Z$-Hodge structures $(\bm{V}_{\rm Be},\Fil^\bullet\bm{V}_{\dR,M(\C)})$.   For each point $z$ of (\ref{grassman}) the induced Hodge decomposition of $V_\C$
has
\[
V_\C^{(1,-1)} = \C z ,
\qquad  V_\C^{(-1,1)}=\C\overline{z},
\qquad V_\C^{(0,0)} = (\C z + \C \overline{z} )^\perp.
\]
It follows that $\Fil^1\bm{V}_{\dR,M(\C)}$ is an isotropic line and $\Fil^0\bm{V}_{\dR, M(\C)}$ is its annihilator with respect
to the pairing on $\bm{V}_{\dR, M(\C)}$ induced from that on $L$.

Similarly, viewing $H=C(V)$ as a representation of $G$ by left multiplication, with its lattice $H_{\widehat{\Z}}=C(L)_{\widehat{\Z}}$, we obtain a variation of $\Z$-Hodge structures $(\bm{H}_{\rm Be},\Fil^\bullet\bm{H}_{\dR,M(\C)})$. This variation has type $(-1,0),(0,-1)$ and is therefore the homology of a family of complex tori over $M(\C)$.

This variation of $\Z$-Hodge structures is polarizable, and so the family of complex tori in fact arises from an abelian scheme over $M(\C)$. To see this, first consider the representation of $G$ on $\Q$ afforded by the spinor norm $\nu$, along with the obvious $K$-stable lattice $\widehat{\Z}\subset\A_f$. For any $\bm{h}\in\mathcal{D}$, the composition $\nu\circ\alpha_{\bm{h}}$ is the homomorphism
\[
\mathrm{Res}_{\C/\R}\mathbb{G}_m \map{ z\mapsto z\overline{z} } \mathbb{G}_{m,\R}.
\]
From this, we see that the associated variation of $\Z$-Hodge structures is the Tate twist $\underline{\Z}(1)$, whose underlying $\Z$-local system is just the constant sheaf $\underline{\Z}$, but whose (constant) Hodge filtration is concentrated in degree $-1$.

Now, choose $\delta\in C^+(L)\cap C^+(V)^\times$ satisfying $\delta^*=-\delta$. As explained in~\eqref{s:preliminaries}, this defines a $G$-equivariant symplectic pairing
\[
\psi_{\delta}:C(V)\times C(V) \to \Q(\nu)
\]
by $\psi_\delta(x,y) = \mathrm{Trd}(x\delta y^*)$.
Since $\delta$ lies in $C^+(L)$, the restriction of this form to $C(L)$ is $\Z$-valued, and in particular induces a $K$-invariant alternating form on $C(L)_{\widehat{\Z}}$ with values in $\widehat{\Z}(\nu)$.

We can arrange our choice of $\delta$ to have the following positivity property: For every $\bm{h}\in\mathcal{D}$, the Hermitian form $\psi_\delta(\alpha_{\bm{h}}(i)x,\overline{y})$ on $C(V)_{\C}$ is (positive or negative) definite. To produce a concrete such choice, choose a negative definite plane $\bm{h}_0\subset V$ defined over $\Q$, and an oriented basis $x,y\in\bm{h}_0\cap L$ for $\bm{h}_0$ satisfying $[x,y]=0$. Then it can be checked that $\delta = xy\in C(L)$ has the desired positivity property.
Any such choice of $\delta$ produces an alternating morphism
\begin{align*}
\bm{H}_{\rm Be} \otimes \bm{H}_{\rm Be} &\to \underline{\Z}(1)
\end{align*}
of variations of Hodge structures, and the positivity property implies that this is  a polarization.

 \begin{remark}
By our construction, any choice of $\delta\in C^+(L)\cap C(V)^\times$ with the required positivity property produces a polarization on $\mathcal{A}_{M(\C)}$. While the polarization on $\mathcal{A}_{M(\C)}$ depends on the choice of $\delta$, the abelian scheme itself does not.
\end{remark}

In sum, we have produced an abelian scheme $\mathcal{A}_{M(\C)}\to M(\C)$, whose homology is canonically identified with the polarizable variation of Hodge structures $(\bm{H}_{\rm Be},\bm{H}_{\dR,M(\C)})$. We  call this the \emph{Kuga-Satake abelian scheme}. It has relative dimension $2^{n+1}$.

The action of $G$ on $H$ commutes with the right multiplication action of $C(V)$, and also respects the $\Z/2\Z$-grading $H=H^+\times H^-$ defined by $H^\pm=C^\pm(V)$.  Thus the Kuga-Satake abelian scheme inherits  a right action of    $C(L)$,  and a $\Z/2\Z$-grading  
\[
\mathcal{A}_{M(\C)}= \mathcal{A}_{M(\C)}^+ \times \mathcal{A}_{M(\C)}^-.
\] 
Elements of $C^+(L)$ and $C^-(L)$ act as homogeneous endomorphisms of graded degrees  zero and one, respectively.
The left multiplication action of $V\subset C(V)$ on $H$ induces an embedding
\[
 (\bm{V}_{\rm Be},\Fil^\bullet\bm{V}_{\dR, M(\C)})\subset \big(\underline{\End}(\bm{H}_{\rm Be}),\Fil^\bullet \underline{\End}(\bm{H}_{\dR,M(\C)}) \big)
\]
of variations of Hodge structures.

For $x\in V$ of positive length, define a divisor on  $\mathcal{D}$ by
\[
\mathcal{D}(x) = \{ z \in \mathcal{D} : z \perp x \}.
\]
As in the work of  Borcherds  \cite{Borcherds}, Bruinier \cite{Bruinier}, and Kudla \cite{Ku}, for every   $m\in \Q_{>0}$ and $\mu\in L^\vee/L$ we define a complex orbifold
\begin{equation}\label{eqn:Zmmu}
Z(m, \mu ) (\C)= \bigsqcup_{ g\in G(\Q)\backslash G(\A_f) /K }  \Gamma_g \backslash \Big(
\bigsqcup_{  \substack{  x\in  \mu_g+ L_g \\ Q(x)=m  }    }   \mathcal{D}(x)
\Big).
\end{equation}
Here  $L_g \subset V$ is the $\Z$-lattice determined by $\widehat{L}_g = g\action \widehat{L}$,  and
$
\mu_g=g\action \mu\in L_g^\vee/L_g.
$
Comparing with (\ref{uniformization}), the   natural finite and unramified morphism
\[
Z(m, \mu ) (\C) \to M(\C)
\]
allows us to view (\ref{eqn:Zmmu}) as an effective Cartier divisor on $M(\C)$. We will see below that (\ref{eqn:Zmmu})
admits a moduli interpretation in terms of the Kuga-Satake abelian scheme.


\subsection{Canonical models over $\Q$}
\label{sec:canonical}


The reflex field of the Shimura datum $(G,\mathcal{D})$ is $\Q$, and hence $M(\C)$ is the complex fiber of an algebraic stack  $M$ over $\Q$.
By \cite[\S 3]{mp:reg} the Kuga-Satake abelian scheme, along with all its additional structures, has a canonical descent $\pi:\mathcal{A}_M\to M$.

We denote by\footnote{In \cite{mp:reg}, one works consistently with cohomology rather than homology,
and so the conventions here  will be dual to the conventions there.
}
\[
\bm{H}_{\dR,M} = \underline{\Hom}\big( R^1\pi_{\mathrm{Zar}, *}
\Omega^\bullet_{\mathcal{A}_M/M} , \co_{M} \big)
\]
the first de Rham homology of $\mathcal{A}_M$ relative to $M$; that is, the dual of the first relative hypercohomology of the de Rham complex.  It is a locally free
$\co_M$-module of rank $2^{n+2}$, equipped with its Hodge filtration
$\Fil^0\bm{H}_{\dR,M} \subset \bm{H}_{\dR,M}$  and Gauss--Manin connection. It also has a $\Z/2\Z$-grading $\bm{H}_{\dR,M}=\bm{H}_{\dR,M}^+\times \bm{H}_{\dR,M}^-$ arising from the grading $\mathcal{A}_M=\mathcal{A}_M^+\times \mathcal{A}_M^-$ on the Kuga-Satake abelian scheme, and similarly a right action of $C(L)$ induced from that on~$\mathcal{A}_M$. Over $M(\C)$, it is canonically identified with $\bm{H}_{\dR,M(\C)}$ with its filtration and integrable connection.

From~\cite[(3.12)]{mp:reg}, we find that there is a canonical descent $\bm{V}_{\dR, M}$ of $\bm{V}_{\dR, M(\C)}$, along with its filtration and integrable connection, so that we have an embedding
\[
 \bm{V}_{\dR,M}\subset\underline{\End}_{C(L)} (\bm{H}_{\dR,M})
\]
of filtered vector bundles over $M$ with integrable connection.
This embedding is a descent of the one already available over $M(\C)$. There is a canonical quadratic form $\bm{Q} : \bm{V}_{\dR,M} \to \co_M$
given on sections by $ x\circ x=\bm{Q}(x) \cdot \mathrm{Id}$, where the composition takes place in $\underline{\End}_{C(L)}(\bm{H}_{\dR,M})$.

In the same way we define
\[
\bm{H}_{\ell , \rm et} = \underline{\Hom}\big( R^1\pi_{\mathrm{et} , *} \underline{\Z}_\ell , \underline{\Z}_\ell \big)
\]
 to be the first  $\ell$-adic \'etale homology of $\mathcal{A}_M$ relative to $M$. Over $M(\C)$, it is canonically identified with $\Z_\ell\otimes\bm{H}_{\rm Be}$. The local system $\Z_\ell \otimes\bm{V}_{\rm Be}$ descends to an $\ell$-adic sheaf $\bm{V}_{\ell, \rm et}$ over $M$, and there is a canonical embedding as a local direct summand
\[
 \bm{V}_{\ell , \rm et}\subset\underline{\End}_{C(L)}(\bm{H}_{\ell , \rm et}).
\]
Again, composition on the right hand side induces a  quadratic form on $\bm{V}_{\ell ,\rm et}$ with values in the constant sheaf $\underline{\Z}_\ell$.


\subsection{Integral models}
\label{sec:integral}


The stack $M$ admits a regular integral model $\mathcal{M}$ over
$\Z$, which we construct in this section following~\cite[\S 7]{mp:reg}. In
\emph{loc.~cit.}~such a model is constructed over $\Z[1/2]$ as one uses the
main result of \cite{kisin}, where $2$ is assumed to be
invertible. As we shall see, most of the results extend also to the
case of the prime $p=2$ replacing  \cite{kisin} with
\cite{mp:2adic}. The construction is involved, and the final definition of~$\mathcal{M}$ is given in Proposition~\ref{prop:M_model over Z}.

It follows from   \cite[(6.8)]{mp:reg} that we may fix an isometric embedding $L\subset L^\diamond$,  where $(L^\diamond, Q^\diamond)$ is
a maximal quadratic space over $\Z$, self-dual at $p$ and of signature $( n^\diamond,2)$. 
Set
\[
\Lambda=L^{\perp}=\{x\in L^\diamond:x\perp L\}.
\]
The construction of the integral model over $\Z_{(p)}$ of $M$ will proceed by first constructing a smooth integral model for the larger Shimura variety $ M^\diamond$ associated with $ L^\diamond$.  There is a finite and unramified morphism $M\to M^\diamond$ over $\Q$ which we will extend to integral models, and  the integral model for $M$ will  inherit extra structure (for example, the vector bundles of Proposition \ref{L_tilde_independence} below)  realized by pulling back structures from the larger integral model.

Over $ M^\diamond$  we have the vector bundle $\bm{V}^\diamond_{\dR, M^\diamond}$ associated with $ L^\diamond$. Let
$\bm{V}^\diamond_{\dR,M}$ be its  restriction  to $M$. By construction, we have an isometric embedding
\[
\bm{V}_{\dR,M}\hookrightarrow\bm{V}^\diamond_{\dR,M}
\]
of filtered vector bundles with integrable connections over $M$.
Setting $\bm{\Lambda}_{\dR,M}=\Lambda\otimes_{\Z}\co_M$, we obtain an injection of vector bundles
\begin{equation}\label{eqn:LambdatildeVdR}
\bm{\Lambda}_{\dR,M}\hookrightarrow\bm{V}^\diamond_{\dR,M}
\end{equation}
identifying $\bm{\Lambda}_{\dR,M}$ with a local direct summand of $\bm{V}^\diamond_{\dR,M}$.

The Shimura variety $ M^\diamond$  admits a smooth canonical model $\mathcal{M}^\diamond_{\Z_{(p)}}$ over $\Z_{(p)}$. This is the main
theorem of~\cite{kisin} for $p\neq 2$ and it is proven in general in  \cite{mp:2adic}. We now recall the construction.
Choose $\delta\in C^+( L^\diamond)\cap C( V^\diamond)^\times$ as in \S \ref{s:the complex Shimura variety}, so that the associated symplectic pairing
$\psi_{\delta}$ on $C( L^\diamond)$ produces a polarization on the Kuga-Satake abelian scheme
\[
 A^\diamond_{ M^\diamond(\C)}\to M^\diamond(\C).
\]
This polarization has degree $d^\diamond$, where $d^\diamond$ is the discriminant of the restriction of $\psi_{\delta}$ to $C( L^\diamond)$. The polarized Kuga-Satake
abelian scheme defines a morphism of complex algebraic stacks
\[
i_{\delta}: M^\diamond_{\C}\to\mathcal{X}^\siegel_{2^{ n^\diamond+1},d^\diamond,\C},
\]
where $\mathcal{X}^\siegel_{2^{ n^\diamond+1},d^\diamond}$ is the Siegel moduli stack over $\Z$ parameterizing polarized abelian schemes of dimension
$2^{n^\diamond+1}$ of degree $d^\diamond$. It can be checked using complex uniformization that this map is finite and unramified. The theory of canonical models of
Shimura varieties  implies that $i_{\delta}$ descends to a finite and unramified map of $\Q$-stacks
\[
i_\delta  :   M^\diamond\to\mathcal{X}^\siegel_{2^{ n^\diamond+1},d^\diamond,\Q}.
\]
The integral model  $\mathcal{M}^\diamond_{\Z_{(p)}}$ is defined as the normalization of $\mathcal{X}^\siegel_{2^{ n^\diamond+1},d^\diamond,\Z_{(p)}}$ in $M^\diamond$, in the following sense:

\begin{definition}
Given an algebraic stack $\mathcal{X}$ over $\Z_{(p)}$, and a normal algebraic stack $Y$ over $\Q$ equipped with a finite map
$j_{\Q}:Y\to\mathcal{X}_{\Q}$, the \emph{normalization} of $\mathcal{X}$ in $Y$ is the finite $\mathcal{X}$-stack $j:\mathcal{Y}\to\mathcal{X}$, characterized by
the property that $j_*\co_{\mathcal{Y}}$ is the integral closure of $\co_{\mathcal{X}}$ in $(j_{\Q})_*\co_Y$. 

The normalization is uniquely determined by the following universal property: given a finite morphism $\mathcal{Z}\to\mathcal{X}$ with $\mathcal{Z}$ a normal algebraic stack over $\Z_{(p)}$, any map of $\mathcal{X}_{\Q}$-stacks
$\mathcal{Z}_{\Q}\to Y$ extends uniquely to a map of $\mathcal{X}$-stacks $\mathcal{Z}\to\mathcal{Y}$.
\end{definition}

\begin{proposition}
The stack $\mathcal{M}^\diamond_{\Z_{(p)}}$ is smooth over $\Z_{(p)}$. It does not depend on the choice of $\delta\in C^+( L^\diamond)\cap C( V^\diamond)^\times$.
\end{proposition}
\begin{proof}
This essentially follows from~\cite[(2.3.8)]{kisin} if $p\neq 2$ and from~\cite[Theorem 3.10]{mp:2adic} in general. We provide some of the details.

Abbreviate
\[
\mathcal{X}^\siegel=\mathcal{X}^\siegel_{2^{ n^\diamond+1},d^\diamond,\Z_{(p)}},
\]
and let $(\mathcal{A}^\siegel, \lambda^\siegel)$ be the universal polarized abelian scheme over $\mathcal{X}^\siegel$. For any $m\in\Z_{>0}$ coprime to
$pd^\diamond$, let $\mathcal{X}^\siegel_{m}\to\mathcal{X}^\siegel$ be the finite \'etale cover parameterizing full level-$m$ structures on $\mathcal{A}^\siegel$;
that is,  parametrizing isomorphisms of group schemes
\[
 C( L^\diamond)\otimes\underline{\Z/m\Z} \xrightarrow{\simeq} \mathcal{A}^\siegel[m]
\]
that carry the symplectic form $\psi_{\delta}\otimes 1$ on the left hand side to a $(\Z/m\Z)^\times$-multiple of the Weil pairing on $\mathcal{A}^\siegel[m]$
induced by $\lambda^\siegel$. For $m\geq 3$ the stack  $\mathcal{X}^\siegel_m$ is a scheme over $\Z_{(p)}$.

Write $ G^\diamond=\GSpin( V^\diamond)$, and set
\[
K^\diamond= G^\diamond(\A_f)\cap C( L^\diamond)_{\widehat{\Z}}^\times.
\]
Let $K^\diamond(m)\subset K^\diamond$ be the largest subgroup acting trivially on $C( L^\diamond)\otimes(\Z/m\Z)$. We obtain a finite \'etale cover $
M^\diamond_m\to M^\diamond$ of Shimura varieties with Galois group $K^\diamond/K^\diamond(m)$, where
\[
  M^\diamond_m(\C) =  G^\diamond(\Q) \backslash {\mathcal{D}^\diamond} \times  G^\diamond(\A_f) / K^\diamond(m)
\]
and  ${\mathcal{D}^\diamond}$ is the space of oriented negative definite planes in $ V^\diamond_{\R}$.

The map $i_\delta :  M^\diamond\to\mathcal{X}^\siegel_{\Q}$ lifts to a map \[i_{\delta,m}: M^\diamond_m\to\mathcal{X}^\siegel_{m,\Q}.\] From~\cite[(2.1.2)]{kisin} we
find that  $i_{\delta,m}$ is a closed immersion of $\Q$-schemes for sufficiently large $m$. Fix such an $m$, and let $\mathcal{M}^\diamond_{m,\Z_{(p)}}$ be the
normalization of $\mathcal{X}^\siegel_m$ in $ M^\diamond_m$. Since $i_{\delta,m}$ is a closed immersion, this is simply the normalization of the Zariski closure of
$ M^\diamond_m$ in $\mathcal{X}^\siegel_m$. It now follows from~\cite[(2.3.8)]{kisin} and \cite[Theorem 3.10]{mp:2adic} that $\mathcal{M}^\diamond_{m,\Z_{(p)}}$ is
smooth over $\Z_{(p)}$, and moreover, that it is, up to unique isomorphism, \emph{independent} of the choice of $\delta$.

The stack $\mathcal{M}^\diamond_{\Z_{(p)}}$ is the \'etale quotient of $\mathcal{M}^\diamond_{m,\Z_{(p)}}$ by the action of $K^\diamond/K^\diamond(m)$, and so the
proposition follows.
\end{proof}

By the construction of the integral model $\mathcal{M}^\diamond_{\Z_{(p)}}$, the Kuga-Satake abelian scheme $ A^\diamond \to  M^\diamond$ extends to
${\mathcal{A}^\diamond}\to\mathcal{M}^\diamond_{\Z_{(p)}}$. A  non-trivial consequence of the construction is that the vector bundle $\bm{V}^\diamond_{\dR,
M^\diamond}$, with its integrable connection and filtration, extends to a filtered vector bundle with integrable connection
$(\bm{V}^\diamond_{\dR,\mathcal{M}^\diamond_{\Z_{(p)}}},\Fil^\bullet\bm{V}^\diamond_{\dR,\mathcal{M}^\diamond_{\Z_{(p)}}})$ over $\mathcal{M}^\diamond_{\Z_{(p)}}$
(see~\cite[\S 4]{mp:reg} for details).

Let $\bm{H}^\diamond_{\dR,\mathcal{M}^\diamond_{\Z_{(p)}}}$ be the first de Rham homology of $\mathcal{A}^\diamond$, so that we have an embedding of filtered vector bundles
\begin{equation}\label{eqn:tildeVdRinclusion}
 \bm{V}^\diamond_{\dR,\mathcal{M}^\diamond_{\Z_{(p)}}}\subset\underline{\End}_{C(L)}\bigl(\bm{H}^\diamond_{\dR,\mathcal{M}^\diamond_{\Z_{(p)}}}\bigr)
\end{equation}
extending that in the generic fiber and exhibiting its source as a local direct summand of the target.
Composition in the endomorphism ring gives a non-degenerate quadratic form on
\[
{\bm{Q}}^\diamond\colon\bm{V}^\diamond_{\dR,\mathcal{M}^\diamond_{\Z_{(p)}}}\to\co_{\mathcal{M}^\diamond_{\Z_{(p)}}}
\]
again extending that on its generic fiber.

Similarly, let
\[
\bm{H}^\diamond_{\crys} = \underline{\Hom}\big( R^1\pi_{\crys,*} \co_{\mathcal{A}^\diamond_{\F_p}/\Z_p}^\crys ,
\co_{\mathcal{M}^\diamond_{\F_p}/\Z_p}^\crys \big)
\]
 be the first  crystalline homology of $\mathcal{A}^\diamond_{\F_p}$ relative to $\mathcal{M}^\diamond_{\F_p}$.  As $\mathcal{M}^\diamond_{\Z_{(p)}}$ is smooth over $\Z_{(p)}$, the vector bundle with integrable connection \eqref{eqn:tildeVdRinclusion} determines a crystal $\bm{V}^\diamond_{\crys}$ over $\mathcal{M}^\diamond_{\F_p}$ equipped with an embedding
  \[
 \bm{V}^\diamond_{\crys}\subset\underline{\End}_{C(L)}\bigl(\bm{H}^\diamond_{\crys}\bigr)
 \]
as a local direct summand; see~\cite[(4.14)]{mp:reg} for details.

We are now ready to give the construction of the desired integral model $\mathcal{M}_{\Z_{(p)}}$ for $M$. The isometric embedding $V\to V^\diamond$ induces a homomorphism of Clifford algebras, which restricts to a closed immersion $G\to G^\diamond$.  This induces a morphism of Shimura varieties $M\to M^\diamond$.  The direct summands $\Lambda_{\Q}\subset V^\diamond$ and $\Lambda_{\widehat{\Z}}\subset L^\diamond_{\widehat{\Z}}$ are pointwise stabilized by $G(\Q)$ and $K$, respectively. From this, arguing as in \cite[(6.5)]{mp:reg}, we find that there is a canonical embedding
\begin{equation}\label{Eqn_Lambda_emb}
 \Lambda\subset\End_{C( L^\diamond)}({\mathcal{A}^\diamond}\vert_M)
\end{equation}
whose de Rham realization is~\eqref{eqn:LambdatildeVdR}.
Let $\check{\mathcal{M}}_{\Z_{(p)}}$ be the normalization of $\mathcal{M}^\diamond_{\Z_{(p)}}$ in $M$. Then, by~\cite[Prop. 2.7]{FaltingsChai}, the
embedding~\eqref{Eqn_Lambda_emb} extends to
\begin{equation}\label{Eqn_Lambda_emb_integral}
 \Lambda\subset\End_{C( L^\diamond)}\bigl({\mathcal{A}^\diamond}\vert_{\check{\mathcal{M}}_{\Z_{(p)}}}\bigr).
\end{equation}

Since $\bm{V}^\diamond_{\dR,\mathcal{M}^\diamond_{\Z_{(p)}}}$ is a direct summand of
$\underline{\End}_{C(L)}\bigl(\bm{H}^\diamond_{\dR,\mathcal{M}^\diamond_{\Z_{(p)}}}\bigr)$, the de Rham realization of~\eqref{Eqn_Lambda_emb_integral} provides us
with an extension
\begin{equation}\label{Eqn_Lambda_dR}
\bm{\Lambda}_{\dR}\to\bm{V}^\diamond_{\dR,\check{\mathcal{M}}_{\Z_{(p)}}}
\end{equation}
of \eqref{eqn:LambdatildeVdR},
where \[\bm{\Lambda}_{\dR}=\Lambda\otimes_{\Z}\co_{\check{\mathcal{M}}_{\Z_{(p)}}}.\]
Let $\check{\mathcal{M}}^{\rm pr}_{\Z_{(p)}}\subset\check{\mathcal{M}}_{\Z_{(p)}}$ be the open locus where the cokernel of~\eqref{Eqn_Lambda_dR} is a vector bundle
of rank $n$ (see~\cite[6.16(i)]{mp:reg}). This is also the locus over which the image of (\ref{Eqn_Lambda_dR})  is a local direct summand. In particular, it contains the generic fiber $M$ of
$\check{\mathcal{M}}_{\Z_{(p)}}$.

Let $\mathrm{M}^{\rm loc}(L)$ be the projective scheme over $\Z$ parameterizing isotropic lines $L^1\subset L$. The following can be deduced
from~\cite[(6.16),(6.20),(2.16)]{mp:reg}.

\begin{proposition}\label{prop:Mcheck}
\mbox{}
\begin{enumerate}
\item Around any point of $\check{\mathcal{M}}^{\rm pr}_{\Z_{(p)}}$, we can find an \'etale neighborhood $T$ such that there exists an isometry
\[
 \xi:\co_T\otimes L^\diamond\xrightarrow{\simeq}\bm{ V}^\diamond_{\dR,T}
\]
satisfying $\xi(\co_T\otimes\Lambda)=\bm{\Lambda}_{\dR,T}$, and such that the morphism $T\to\mathrm{M}^{\rm loc}(L)$ determined by the  isotropic line 
\[ 
\xi^{-1}(\Fil^1\bm{ V}^\diamond_{\dR,T})\subset\co_T\otimes L^\diamond 
\]  
is \'etale. \item If $L_{(p)}$ is self-dual,  or if $p>2$  and $p^2\nmid \mathrm{disc}(L)$, then $\check{\mathcal{M}}^{\rm pr}_{\Z_{(p)}}=\check{\mathcal{M}}_{\Z_{(p)}}$.
 \item If $p>2$ and $p^2 \mid \mathrm{disc}(L)$, then $\check{\mathcal{M}}^{\rm pr}_{\Z_{(p)}}$ is regular outside of a $0$ dimensional closed substack.
\end{enumerate}
\end{proposition}

We now define:
\begin{itemize}
\item $\mathcal{M}_{\Z_{(p)}}=\check{\mathcal{M}}_{\Z_{(p)}}$ if $L_{(p)}$ is self-dual, or if $p>2$  and $p^2\nmid\mathrm{disc}(L)$; 
\item
$\mathcal{M}_{\Z_{(p)}}=\text{the regular locus of }\check{\mathcal{M}}^{\rm pr}_{\Z_{(p)}}$ if $p>2$ and $p^2\mid \mathrm{disc}(L)$; 
\item
$\mathcal{M}_{\Z_{(p)}}=\text{the smooth locus of }\check{\mathcal{M}}^{\rm pr}_{\Z_{(p)}}$ if $p=2$.
\end{itemize}
Then $\mathcal{M}_{\Z_{(p)}}$ is the desired regular integral model for $M$ over $\Z_{(p)}$.

\begin{remark}
If  $p=2$ and $L_{(2)}$ is not self-dual,  or if $p>2$ and $p^2 \mid \mathrm{disc}(L)$, this is
certainly not the optimal definition of the regular model $\mathcal{M}_{\Z_{(p)}}$. For a discussion and examples of this
when $p$ is odd, see~\cite[(6.27),(6.28)]{mp:reg}.
\end{remark}

Equip $\bm{\Lambda}_{\dR}=\Lambda\otimes_{\Z}\co_{\mathcal{M}_{\Z_{(p)}}}$ with the integrable connection $1\otimes\rm d$, and the trivial filtration concentrated in degree $0$.  By construction, over $\mathcal{M}_{\Z_{(p)}}$ we have a canonical isometric embedding
\begin{equation*}
\bm{\Lambda}_{\dR}\subset\bm{V}^\diamond_{\dR,\mathcal{M}_{\Z_{(p)}}}
\end{equation*}
of filtered vector bundles with integrable connections. Therefore, since the quadratic form on $\bm{V}^\diamond_{\dR,\mathcal{M}_{\Z_{(p)}}}$ is self-dual, the
orthogonal complement
\[
\bm{V}_{\dR,\mathcal{M}_{\Z_{(p)}}}=\bm{\Lambda}^{\perp}_{\dR}\subset\bm{V}^\diamond_{\dR,\mathcal{M}_{\Z_{(p)}}}
\]
is again a vector sub-bundle with integrable connection. Indeed, it is precisely the kernel of the surjection of vector bundles
\[
 \bm{V}^\diamond_{\dR,\mathcal{M}_{\Z_{(p)}}}\xrightarrow{\simeq}\bm{V}^{\diamond,\vee}_{\dR,\mathcal{M}_{\Z_{(p)}}}\to  \bm{V}^{\diamond,\vee}_{\dR,\mathcal{M}_{\Z_{(p)}}}/\mathrm{ann}(\bm{\Lambda}_{\dR}).
\]
Moreover, the isotropic line $\Fil^1\bm{V}^\diamond_{\dR,\mathcal{M}_{\Z_{(p)}}}$ is actually contained in $\bm{V}_{\dR,\mathcal{M}_{\Z_{(p)}}}$. For convenience,
we abbreviate
\[
\bm{V}^1_{\dR,\mathcal{M}_{\Z_{(p)}}}= \Fil^1\bm{V}^\diamond_{\dR,\mathcal{M}_{\Z_{(p)}}}.
\]

\begin{proposition}\label{L_tilde_independence}
The integral models $\mathcal{M}_{\Z_{(p)}}$ and $\check{\mathcal{M}}_{\Z_{(p)}}$ are, up to unique isomorphism, independent of the auxiliary choice of $
L^\diamond$. Furthermore, the vector bundle $\bm{V}_{\dR,\mathcal{M}_{\Z_{(p)}}}$ and the isotropic line
$
 \bm{V}^1_{\dR,\mathcal{M}_{\Z_{(p)}}}\subset\bm{V}_{\dR,\mathcal{M}_{\Z_{(p)}}}
$
are also independent of this auxiliary choice.
\end{proposition}
\begin{proof}
Suppose that we have a different embedding $L\hookrightarrow L^\diamond_1$ with $ L^\diamond_1$ maximal, self-dual at $p$, and of signature $( n^\diamond_1,2)$.

First, assume that the embedding $L\hookrightarrow L^\diamond_1$ factors as $L\hookrightarrow L^\diamond\hookrightarrow L^\diamond_1$. The construction of Kisin
in~\cite{kisin} and its generalization in \cite{mp:2adic} show that the map of integral models $\mathcal{M}^\diamond_{\Z_{(p)}}\to\mathcal{M}^\diamond_{1,\Z_{(p)}}$
associated with $ L^\diamond\hookrightarrow L^\diamond_1$ is finite and unramified. Moreover, Proposition \ref{prop:Mcheck} shows that, if $\Lambda_1=
(L^\diamond)^{\perp}\subset L^\diamond_1$, then the corresponding map
\[
\bm{\Lambda}_{1,\dR,\mathcal{M}^\diamond_{\Z_{(p)}}}\to\bm{ V}^\diamond_{1,\dR,\mathcal{M}_{\Z_{(p)}}}
\]
is an embedding of vector bundles, and that its orthogonal complement is exactly $\bm{V}^\diamond_{\dR,\mathcal{M}_{\Z_{(p)}}}$. From this, we easily deduce that
the constructions of both $\check{\mathcal{M}}_{\Z_{(p)}}$ and $\check{\mathcal{M}}^{\rm pr}_{\Z_{(p)}}$ do not depend on the choice of isometric embedding. The
proof also shows that the vector bundle $\bm{V}_{\dR,\mathcal{M}_{\Z_{(p)}}}$, along with its isotropic line, is also independent of this choice.

For the general case, consider the embedding
$
 \Delta: L  \hookrightarrow  L^\diamond\oplus L^\diamond_1
$
defined by $v\mapsto (v,-v)$.
The quadratic form on $ L^\diamond\oplus L^\diamond_1$ given by the orthogonal sum of those on its individual summands induces a quadratic form on the quotient
\[
 L^\diamond_2=( L^\diamond\oplus L^\diamond_1)/\Delta(L)
\]
of signature $(n^\diamond+n^\diamond_1-n,2)$. Moreover, the natural maps $ L^\diamond\to L^\diamond_2$ and $ L^\diamond_1\to L^\diamond_2$ are isometric inclusions
of direct summands, whose restrictions to $L$ coincide.

Now, $ L^\diamond_2$ need not be self-dual at $p$, but, via the argument in~\cite[(6.8)]{mp:reg}, we can embed it isometrically as a direct summand in a maximal
quadratic space $ L^\diamond_3$ of signature $( n^\diamond_3,2)$, which is self-dual at $p$. In particular, replacing $ L^\diamond_1$ with $ L^\diamond_3$, we are
reduced to the already considered case where $ L^\diamond$ is an isometric direct summand of $ L^\diamond_1$.
\end{proof}

We can now give the construction of the stack $\mathcal{M}$ over $\Z$. It is simply obtained by patching together the spaces $\mathcal{M}_{\Z_{(p)}}$. To do this
rigorously, we choose a finite collection of maximal quadratic spaces $ L^\diamond_1, L^\diamond_2,\ldots, L^\diamond_r$ with the following properties:
\begin{itemize}
\item For each $i=1,2,\ldots,r$, $ L^\diamond_i$ has signature $( n^\diamond_i,2)$, for $ n^\diamond_i\in\Z_{>0}$; \item For each $i$, there is an isometric
embedding $L\hookrightarrow L^\diamond_i$; \item If, for each $i$, we denote by $S_i$ the set of primes dividing  $\mathrm{disc}(L^\diamond_i)$, then
$\bigcap_{i=1}^rS_i=\emptyset$.
\end{itemize}
It is always possible to find such a collection. For instance, first choose any maximal quadratic lattice of signature $( n^\diamond_1,2)$, admitting $L$ as an
isometric direct summand. Let $\{p_2,\ldots,p_r\}$ be the set of primes dividing the discriminant of $ L^\diamond_1$. Now, for each $i=2,3,\ldots,r$, let $
L^\diamond_i$ be a maximal quadratic lattice, self-dual at $p_i$, of signature $( n^\diamond_i,2)$, admitting $L$ as an isometric direct summand. Then the
collection $\{ L^\diamond_1,\ldots, L^\diamond_r\}$ does the job.

For $i=1,2,\ldots,r$, let $ M^\diamond_i$ be the GSpin Shimura variety over $\Q$ attached to $ L^\diamond_i$. Given an appropriate choice of $\delta_i\in C^+(
L^\diamond_i)\cap C( L^\diamond_i)_{\Q}^\times$, with associated alternating form $\psi_{\delta_i}$ on $C( L^\diamond_i)$, we obtain a sequence of finite and
unramified map of $\Q$-stacks
\[
M\to M^\diamond_i \xrightarrow{i_{\delta_i}} \mathcal{X}^\siegel_{2^{ n^\diamond_i+1},d^\diamond_i,\Q}.
\]
Here, $d^\diamond_i$ is the discriminant of $\psi_{\delta_i}$.

Let $\Z[(S_i)^{-1}]\subset\Q$ be the localization of $\Z$ obtained by inverting  the primes in $S_i$. Let $\mathcal{M}_{\Z[(S_i)^{-1}]}$ be the stack over
$\Z[(S_i)^{-1}]$ obtained by normalizing $\mathcal{X}^\siegel_{2^{ n^\diamond_i+1},d^\diamond_i,\Z[(S_i)^{-1}]}$ in $M$.

\begin{proposition}
\label{prop:M_model over Z}
There is a regular, flat, algebraic $\Z$-stack $\mathcal{M}$ such that for each $i$ the restriction of $\mathcal{M}$ over $\Z[(S_i)^{-1}]$ is isomorphic to
$\mathcal{M}_{\Z[(S_i)^{-1}]}$. Moreover, the vector bundle $\bm{V}_{\dR,M}$ on $M$, along with its integrable connection and the isotropic line $\bm{V}^1_{\dR,M}$, has a
canonical extension $(\bm{V}_{\dR},\bm{V}^1_{\dR})$ over $\mathcal{M}$.
\end{proposition}
\begin{proof}
This is immediate from Proposition~\ref{L_tilde_independence}.
\end{proof}

\begin{definition}\label{def:cotautological}
The line bundle  $\bm{V}_{\dR}^1\subset \bm{V}_{\dR}$ is  the \emph{tautological bundle} on $\mathcal{M}$.
Its dual is the \emph{cotautological bundle}.
\end{definition}

\begin{remark}
Under the uniformization
$
 \mathcal{D}\to \Gamma_g\backslash\mathcal{D} \subset \mathcal{M}(\C),
$
of~\eqref{uniformization}, and using the notation of~\eqref{grassman}, the tautological bundle and cotautological bundle  pull-back to the line bundles $z\mapsto\C z$ and $z\mapsto V_\C/(\C z)^{\perp}$, respectively.
\end{remark}

By (7.12) of \cite{mp:reg}, the Kuga-Satake abelian scheme over $M$ extends (necessarily uniquely) to an abelian scheme
\[
\pi \colon \mathcal{A} \to \mathcal{M}
\]
endowed with a right $C(L)$-action and $\Z/2\Z$-grading.  The de Rham realization of~$\mathcal{A}$ provides us with an extension $\bm{H}_{\dR}$ over $\mathcal{M}$ of the
filtered vector bundle with connection $\bm{H}_{\dR,M}$, equipped with its $\Z/2\Z$-grading and right $C(L)$-action. Let
\[
\bm{H}_{\crys} = \underline{\Hom}\big( R^1\pi_{\crys,*} \co_{\mathcal{A}_{\F_p}/\Z_p}^\crys ,
\co_{\mathcal{M}_{\F_p}/\Z_p}^\crys \big)
\]
 be the first  crystalline homology of $\mathcal{A}_{\F_p}$ relative to $\mathcal{M}_{\F_p}$.

Denote by $\bm{1}$ the structure sheaf $\co_{\mathcal{M}}$ with its standard connection and
trivial filtration in the de Rham case; the constant sheaf $\underline{\Z}_p$ on $\mathcal{M}[1/p]$
in the \'etale case; the crystal $\co_{\mathcal{M}_{\F_p},{\crys}}$ in the crystalline case; and the trivial variation of $\Z$-Hodge structures over $M_{\C}^{\rm an}$ in the Betti case.   If $\bm{H}$ is any one of
  $\bm{H}_{\dR}$, $\bm{H}_{\ell , \rm et}$, $\bm{H}_{\crys}$, or $\bm{H}_{\Be}$, then $\bm{H}$ is a $\bm{1}$-module in the appropriate category.

It follows from
\cite[(3.12),(7.13)]{mp:reg} that there is a  canonical local direct summand
\begin{equation}\label{eq:bmV}
\bm{V} \subset \underline{\mathrm{End}}_{C(L)}(\bm{H})
\end{equation}
of grade shifting endomorphisms.  In fact, as we will see below in Proposition~\ref{prop:functorial}, the homological realizations of the canonical isometric
embedding $\Lambda\hookrightarrow\End({\mathcal{A}^\diamond}\vert_{\mathcal{M}})$ exhibit $\Lambda\otimes\bm{1}$ as an isometric local direct summand of
$\bm{V}^\diamond\vert_{\mathcal{M}}$, and 
\[
\bm{V}=(\Lambda\otimes\bm{1})^{\perp}\subset \bm{V}^\diamond\vert_{\mathcal{M}}.
\]

In the de Rham or Betti case, $\bm{V}$ is identified with the realizations already constructed above. Moreover, for any prime $p$, the \'etale realization $\bm{V}_{p,\rm et}\subset\underline{\mathrm{End}}_{C(L)}(\bm{H}_{p,\rm et})$ is the unique extension of that over $\mathcal{M}[1/p]$. Here, we are using the following fact: Given a normal flat Noetherian algebraic $\Z[1/p]$-stack $S$, and a lisse $p$-adic sheaf $F$ over $S$, any lisse subsheaf of $F\vert_{S_{\Q}}$ extends uniquely to a lisse subsheaf of $F$. To see this, we can reduce to the case where $S$ is a normal, connected, flat Noetherian $\Z[1/p]$-scheme, where the assertion amounts to saying that the map of \'etale fundamental groups $\pi_1(S_{\Q})\to\pi_1(S)$ is \emph{surjective}. This is shown in~\cite[Exp. V, Prop. 8.2]{sga1}.

For any section $x$ of $\bm{V}$,  the element $x\circ x$ is a scalar endomorphism of $\bm{H}$.   Define  a quadratic form
\[
\bm{Q}:\bm{V}\to\bm{1}
\]
by $x\circ x = \bm{Q}(x)\cdot \rm Id$.

\begin{remark}
Although we will not require this in what follows, we note that there is a certain compatibility between the crystalline realization $\bm{V}_{\crys}$ over $\mathcal{M}_{\F_p}$ and the $p$-adic realization $\bm{V}_{p,\rm et}$ over $\mathcal{M}[1/p]$.
Suppose that we are given a finite extension $F/\Q_p$, contained in an algebraic closure $\overline{\Q}_p$ of $\Q_p$, with residue field $k$. Write $F_0\subset F$ for the maximal unramified subextension of $F$.
Suppose that we are given a point $s\in\mathcal{M}(F)$ specializing to a point $s_0\in\mathcal{M}(k)$. Write $\overline{s}$ for the corresponding $\overline{\Q}_p$-point of $\mathcal{M}$.
Then the comparison isomorphism between the crystalline cohomology of $\mathcal{A}_{s_0}$ with respect to $W(k)=\co_{F_0}$ and the $p$-adic \'etale cohomology of $\mathcal{A}_{\overline{s}}$ induces, by ~\cite[(7.11)]{mp:reg}, an isometry
\[
 \bm{V}_{\crys,s_0}\otimes_{F_0} B_{\crys}\xrightarrow{\simeq}\bm{V}_{p,\rm et,\overline{s}}\otimes_{\Q_p} B_{\crys}.
\]
\end{remark}


\subsection{Functoriality of integral models}


We consider the question of  functoriality of the previous constructions with respect to the maximal quadratic space $L$.
Let $L_0\subset L$ be an isometric embedding of maximal quadratic spaces over $\Z$, of signatures $(n_0,2)$
and $(n,2)$, respectively. This implies that  $L_0$ is a $\Z$-module direct summand of $L$. We will maintain our assumption that $n_0\ge 1$.

The quadratic spaces $L_0$ and $L$  have  associated Shimura varieties $M_0$ and $M$,
with level structure defined by the compact open subgroups
\begin{align*}
K_0&= \GSpin(L_0\otimes \A_f)\cap C(\widehat{L}_0)^\times \\
K &=
\GSpin(L\otimes\A_f)\cap C(\widehat{L})^\times.
\end{align*}
The inclusion $L_0\hookrightarrow L$ induces a homomorphism of Clifford algebras, which then restricts to an
inclusion  of algebraic groups $\GSpin(V_0)\rightarrow\GSpin(V)$ satisfying
$K_0=K\cap \GSpin(V_0)(\A_f)$.
By the theory of canonical models of Shimura varieties, there
is an induced finite and unramified morphism of $\Q$-stacks  $M_0\to M.$

Let $\mathcal{M}_0$ and $\mathcal{M}$ be the integral models over $\Z$, and let $\mathcal{A}_0 \to \mathcal{M}_0$ and $\mathcal{A} \to \mathcal{M}$ be their
Kuga-Satake abelian schemes.  They are equipped with actions of $C(L_0)$ and $C(L)$, respectively.    Let $\bm{H}_0$ and  $\bm{H}$ stand for the various homological
realizations of $\mathcal{A}_0$ and $\mathcal{A}$, respectively. We have the canonical sub-objects  $\bm{V}_0 \subset \underline{\mathrm{End}}_{C(L_0)}(\bm{H}_0)$
and $\bm{V}\subset \underline{\mathrm{End}}_{C(L)}(\bm{H})$ introduced  in  \S\ref{sec:integral}.

The Kuga-Satake abelian schemes over $M_0$ and  $M$ are determined by the lattices
$C(L_0)\subset C(V_0)$ and $C(L)\subset C(V)$ (in the sense of \cite[\S 4.12]{DeligneShimura}),
and the argument in  \cite[(6.4)]{mp:reg} shows that there is a  $C(L)$-linear isomorphism of abelian schemes
\begin{equation}\label{eqn:rat tensor}
\big( \mathcal{A}_0 \otimes_{C(L_0)} C(L) \big)\vert_{M_0} \iso  \mathcal{A}\vert_{M_0}
\end{equation}
over the generic fiber $M_0$ of $\mathcal{M}_0$.
As $L_0$ is a direct summand of $L$, one easily
checks that $C(L)$ is a free $C(L_0)$-module, and so the  tensor construction
$\mathcal{A}_0\otimes_{C(L_0)} C(L)$ in (\ref{eqn:rat tensor}) is defined.

\begin{proposition}\label{prop:functorial}
The morphism $M_0 \to M$ extends uniquely to an unramified morphism
$\mathcal{M}_0\to \mathcal{M}$, and (\ref{eqn:rat tensor})
extends to a $C(L)$-linear isomorphism of abelian schemes
\begin{equation}\label{abscheme_isom}
\mathcal{A}_0\otimes_{C(L_0)} C(L) \iso \mathcal{A}\vert_{\mathcal{M}_0}
\end{equation}
over $\mathcal{M}_0$.  In particular,
\begin{equation}\label{realization tensor}
 \bm{H}_0\otimes_{C(L_0)}C(L)\iso \bm{H}\vert_{\mathcal{M}_0}.
\end{equation}

Furthermore, if we set $\Lambda =\{ \lambda \in L : \lambda \perp L_0\}$ and define a sheaf $\bm{\Lambda} =\Lambda\otimes \bm{1}$ on $\mathcal{M}_0$, then there is a canonical embedding $\Lambda\subset\End\bigl( \mathcal{A}\vert_{\mathcal{M}_0}\bigr)$ with the following properties:

\begin{enumerate}
\item Its homological realization $\bm{\Lambda} \hookrightarrow \underline{\End}( \bm{H}\vert_{\mathcal{M}_0})$
 arises from an isometric map $\bm{\Lambda}\hookrightarrow \bm{V}\vert_{\mathcal{M}_0}$,
 and exhibits $\bm{\Lambda}$ as a local direct summand of $\bm{V}\vert_{\mathcal{M}_0 }$.

\item
The  injection
\[
\underline{\mathrm{End}}_{C(L_0)} (\bm{H}_0) \to \underline{\mathrm{End}}_{C(L)}( \bm{H}\vert_{\mathcal{M}_0})
\]
induced by (\ref{realization tensor}) identifies $\bm{V}_0$ with the   submodule of all elements of
$\bm{V}\vert_{\mathcal{M}_0   }$ anticommuting with all elements of  $\bm{\Lambda}$.
Furthermore,   $\bm{V}_0 \subset \bm{V}\vert_{\mathcal{M}_0   }$
is locally a direct summand.

\item  In the de Rham case, the inclusion $\bm{V}_{\dR}\subset\bm{V}_{\dR}\vert_{\mathcal{M}_0}$ identifies
\[
\bm{V}_{0,\dR}^1\iso \bm{V}_{\dR}^1\vert_{\mathcal{M}_0}.
\]
\end{enumerate}
\end{proposition}

\begin{proof}
It suffices to construct the morphism $\mathcal{M}_0 \to \mathcal{M}$   after localizing at a prime $p$. So we can assume that we have a sequence of isometric embeddings $L_0\hookrightarrow L
\hookrightarrow L^\diamond$, where $ L^\diamond$ is maximal and self-dual at $p$.

By taking the normalizations of $\mathcal{M}^\diamond_{\Z_{(p)}}$ in $M_0$ and $M$, respectively, we see that the finite morphism $M_0\to M$ extends to a map
\begin{equation}\label{eqn:McheckFunct}
\check{\mathcal{M}}_{0, \Z_{(p)}}\to\check{\mathcal{M}}_{\Z_{(p)}}.
\end{equation}

Set $\Lambda_1=L^{\perp}\subset L^\diamond$ and $\widetilde{\Lambda}=L_0^{\perp}\subset  L^\diamond$, so that  $\Lambda_1$ is a direct summand in
$\widetilde{\Lambda}$ with orthogonal complement $\Lambda$. From this, it is clear that~\eqref{eqn:McheckFunct} restricts to a map $\check{\mathcal{M}}^{\rm
pr}_{0,\Z_{(p)}}\to\check{\mathcal{M}}^{\rm pr}_{\Z_{(p)}}$. We can deduce from Proposition \ref{prop:Mcheck} that, \'etale locally on the source, this map is
isomorphic to an \'etale neighborhood of the map of quadrics $\mathrm{M}^{\rm loc}(L_0)\to\mathrm{M}^{\rm loc}(L)$. In particular, it is unramified.

To check that this carries $\mathcal{M}_{0,\Z_{(p)}}$ to $\mathcal{M}_{\Z_{(p)}}$ it suffices now to show that the map 
\[
\mathrm{M}^{\rm loc}(L_0)_{\Z_{(p)}}\to\mathrm{M}^{\rm loc}(L)_{\Z_{(p)}}
\]
 of quadrics over $\Z_{(p)}$ carries the smooth locus to the smooth locus and, if $p\neq 2$, the regular
locus of the source into that of the target. 

For the first assertion, observe that an $\overline{\F}_p$-valued point of the space is in the smooth locus if and only if it corresponds to an isotropic line that is not contained in the radical of $L_{0,\overline{\F}_p}$. But then it is \emph{a fortiori} not in the radical of $L_{\overline{\F}_p}$, and hence corresponds to a smooth point of $\mathrm{M}^{\mathrm{loc}}(L)$ as well. 

We are left to prove the second claim, so that we can assume that $p$ is odd.

If the discriminant of $L_{\Z_{(p)}}$ is not divisible by $p^2$, the entire target is regular (see~\cite[(2.16)]{mp:reg}), and there is nothing to check. Otherwise,
the locus where $\mathrm{M}^{\rm loc}(L)_{\Z_{(p)}}$ is not regular consists exactly of the two $\F_{p^2}$-valued points where the corresponding isotropic line $L^1
\subset L_{\F_{p^2}}$ lies in the radical $\mathrm{rad}(L_{\F_{p^2}})$ and is also isotropic for the non-degenerate quadratic form on this radical
(see~\emph{loc.~cit.}).

Suppose therefore that we have an isotropic line $L^1_{0,\overline{\F}_p}\subset L_{0,\overline{\F}_p}$. We need to show: If, as an isotropic line in
$L_{\overline{\F}_p}$, it corresponds to one of the two singular points above, then the discriminant of $L_0$ is also divisible by $p^2$, and
$L^1_{0,\overline{\F}_p}$ corresponds to a singular point of $\mathrm{M}^{\rm loc}(L_0)$.

As any \emph{smooth} point of $\mathrm{M}^{\rm loc}(L_0)$ (that is, a point at which the space is smooth over $\Z$) maps to a smooth point of $\mathrm{M}^{\rm
loc}(L)$, the only remaining possibility is that the discriminant of $L_0$ is divisible exactly by $p$, and that $L^1_{0,\overline{\F}_p}$ is just the radical of
$L_{0, \overline{\F}_p}$, corresponding to the unique non-smooth point of $\mathrm{M}^{\rm loc}(L_0)_{\Z_{(p)}}$. However, this point is defined over $\F_p$,
whereas both the singular points of $\mathrm{M}^{\rm loc}(L)_{\Z_{(p)}}$ are only defined over $\F_{p^2}$. So this possibility can also be excluded.

Now, the fact that the isomorphism of abelian schemes also extends over $\mathcal{M}_0$ is a consequence of \cite[Proposition I.2.7]{FaltingsChai}.

The canonical embedding $\Lambda\subset \End\bigl(\mathcal{A}\vert_{\mathcal{M}_{0,\Z_{(p)}}}\bigr)$ is constructed as follows: First, we will define an action of
$\Lambda$ on
\[
\mathcal{A}_0\otimes_{\Z}C(L)= \bigl(\mathcal{A}_0^+\otimes_{\Z}C(L) \bigr)\times \bigl(\mathcal{A}_0^-\otimes_{\Z}C(L) \bigr).
\]
Given a section of the form $(a^+\otimes c,a^-\otimes c')$ of this product and an element $\lambda\in\Lambda$, we set
\[
 \lambda\cdot(a^+\otimes c,a^-\otimes c')=(a^+\otimes\lambda c,-a^-\otimes\lambda c').
\]
As $L_0$ anti-commutes with $\Lambda$ within $C(L)$,  it is easy to check that this action of $\Lambda$ descends to one on
\[
 \mathcal{A}_0\otimes_{C(L_0)}C(L)\xrightarrow[\simeq]{\eqref{abscheme_isom}}\mathcal{A}\vert_{\mathcal{M}_{0,\Z_{(p)}}}.
\]

Claims (i), (ii) and (iii) are now shown in \cite[(7.13)]{mp:reg} for the crystalline realizations over $\mathcal{M}_{0,\F_p}$, but in the case where $L$ is
self-dual at $p$. From this, we can easily deduce the general case, just as above, by embedding everything in a quadratic space that is self-dual at $p$.
\end{proof}


\subsection{Special endomorphisms}
\label{ss:specialend}


For any scheme $S \to \mathcal{M}$  the pull-back $\mathcal{A}_S$ of the Kuga-Satake abelian scheme
has a distinguished submodule
\[
V(\mathcal{A}_S) \subset \End_{C(L)} (\mathcal{A}_S)
\]
of \emph{special endomorphisms}, defined in \cite[(5.4)]{mp:reg}.
If $S$ is connected and $s\to S$ is any geometric point, there is a cartesian diagram
\[
\xymatrix{
{ V(\mathcal{A}_S) } \ar[r]  \ar[d] & {  \End_{C(L)}(\mathcal{A}_S ) } \ar[d] \\
{ V(\mathcal{A}_s) } \ar[r]  & {  \End_{C(L)}(\mathcal{A}_s ) . }
}
\]
In other words, if $S$ is connected, an endomorphism is special if and only if it is special at one (equivalently, all) geometric points of $S$.

At a geometric point $s$,  the property of being special  can be characterized homologically, using the subspace
\begin{equation}\label{eqn:hom special}
\bm{V}_{s}\subset\End_{C(L)} (\bm{H}_s).
\end{equation}
If $s=\Spec(\C)$   then  $x\in \End_{C(L)} (\mathcal{A}_s)$ is special if and only
if its  Betti realization lies in the subspace
(\ref{eqn:hom special}). This is equivalent to the $\ell$-adic \'etale realization of $x$
lying in (\ref{eqn:hom special}) for one (equivalently, all)  primes $\ell$. It is also equivalent to requiring that the de Rham realization $t_{\dR}(x)\in \End_{C(L)}(\bm{H}_{\dR,s})$ lie in the subspace $\bm{V}_{\dR,s}$.

If $s$ is a geometric point of characteristic $p$, then $x\in
\End_{C(L)} (\mathcal{A}_s)$ is special if and only if its
crystalline realization  lies in (\ref{eqn:hom special}).  This
implies that the $\ell$-adic \'etale realization lies in
(\ref{eqn:hom special}) for all $\ell$ different from the
characteristic of $s$; cf.~\cite[(5.22)]{mp:reg}.

Fix an element $\mu\in L^\vee/L$ and a prime $p>0$, and suppose that we are given an embedding $L\hookrightarrow L^\diamond$ of maximal quadratic lattices with $L^\diamond$ self-dual at $p$. Let $\Lambda = L^\perp\subset L^\diamond$ be the subspace of elements orthogonal to $L$.

There are canonical isomorphisms
\begin{equation}\label{eqn:disc triv}
\Z_p\otimes(L^\vee/L) \xleftarrow{\simeq} \Z_p\otimes(L^\diamond/(L\oplus\Lambda)) \xrightarrow{\simeq}\Z_p\otimes(\Lambda^{\vee}/\Lambda),
\end{equation}
which allow us to view $\mu_p$ as an element of $\Z_p\otimes(\Lambda^\vee/\Lambda)$. Fix a lift $\tilde{\mu}_p\in \Lambda^\vee$ of $\mu_p$.

By construction, the map $\Lambda\hookrightarrow\End( \mathcal{A}^\diamond_S)$ from Proposition~\ref{prop:functorial} factors through an isometric embedding $\Lambda\hookrightarrow V( \mathcal{A}^\diamond_S )$.  Via this embedding, we can view $\tilde{\mu}_p$ as an element of $V(\mathcal{A}_S^\diamond)_\Q$.  

\begin{lemma}
\label{lem:mu def self-dual ind}
The subset
\[
V_\mu(\mathcal{A}_S)_{(p)} = \{x\in V(\mathcal{A}_S)_\Q:\; x + \tilde{\mu}_p\in V(\mathcal{A}^\diamond_S)\otimes\Z_{(p)}\}\subset V(\mathcal{A}_S)_\Q
\]
is functorial in $S$, and is independent of the choice of embedding $L\hookrightarrow L^\diamond$ of maximal lattices with $L^\diamond$ self-dual at $p$.
\end{lemma}
\begin{proof}
Note that the definition is clearly independent of the choice of lift $\tilde{\mu}_p$.

The functoriality of the subset is clear, so it only remains to prove the independence from the choice of embedding into a self-dual at $p$ maximal lattice. For this, as in the proof of Proposition~\ref{L_tilde_independence} we can reduce to the case where both $L$ and $L^\diamond$ are self-dual at $p$, where we have to show that the definition using the identity $L \to L$ is equivalent to the one given using the embedding $L\hookrightarrow L^\diamond$. 

In this case, $\tilde{\mu}_p$ can be chosen to be $0$, and we are reduced to checking that an element $x\in V(\mathcal{A}_S)_\Q$ maps into $V(\mathcal{A}_S)\otimes{\Z_{(p)}}$ if and only if it maps into $V(\mathcal{A}^\diamond_S)\otimes\Z_{(p)}$. For this, it is enough to verify that the inclusion of abelian groups
\[
V(\mathcal{A}_S)\subset V(\mathcal{A}^\diamond_S)
\]
maps onto a $\Z$-direct summand. However, this follows from the fact that $\bm{V}\subset \bm{V}^\diamond\vert_{\mathcal{M}}$ is a local direct summand; see Proposition~\ref{prop:functorial}(ii).
\end{proof}

\begin{remark}
\label{rem:Vmu simpler}
When $S$ is a scheme over $\mathcal{M}[p^{-1}]$, there is a more intrinsic description of the subset above.
Note that the natural action of the compact open subgroup $K\subset G(\A_f)$ on $\widehat{L}^\diamond$ preserves $\widehat{L}$, and induces the \emph{trivial} action on all the quotients involved in~\eqref{eqn:disc triv}. Therefore, over $\mathcal{M}[p^{-1}]$, we have a canonical isometry
\[
\alpha_p:\underline{\Z}_p\otimes(L^\vee/L)\xrightarrow{\simeq} \bm{V}_p^\vee/\bm{V}_p
\]
of $p$-torsion \'etale sheaves. In particular, by taking the pre-image of $\alpha_p(1\otimes\mu_p)$ in $\bm{V}_p^\vee$, we obtain a subsheaf of sets 
$
\bm{V}_{\mu,p}\subset \bm{V}_p^\vee.
$
The subset  $V_\mu(\mathcal{A}_S)_{(p)}$ consists precisely of those elements of $V(\mathcal{A}_S)_\Q$, whose $p$-adic realizations land in $\bm{V}_{\mu,p}$.
\end{remark}

For a scheme $S\to \mathcal{M}$, we now define a natural subset $V_\mu(\mathcal{A}_S)\subset V(\mathcal{A}_S)_\Q$ by setting
\[
V_\mu(\mathcal{A}_S) = \bigcap_pV_\mu(\mathcal{A}_S)_{(p)}.
\]

We record some properties of these spaces.

\begin{proposition}\label{Prop:positive_definite}\
\begin{enumerate}
\item If $\mu = 0$, then $V_0(\mathcal{A}_S) = V(\mathcal{A}_S)$.
	\item  Each $x\in V_\mu(\mathcal{A}_S)$, viewed as an element of $\End\bigl(\mathcal{A}_S\bigr)_\Q$, shifts the grading on $\mathcal{A}_S$ and  commutes with the right action of $C(L)$. 
	\item Moreover,  $x\circ x=Q(x)\cdot {\rm Id}$
for some non-negative  $Q(x)\in \Q$ satisfying
\begin{equation}\label{Q cong}
Q(x) \equiv Q(\mu) \pmod{\Z},
\end{equation}
and  $Q(x)=0$ if and only if $x=0$.
\end{enumerate}
\end{proposition}
\begin{proof}
The first statement was shown in the course of the proof of Lemma~\ref{lem:mu def self-dual ind}, where we found that $V_{\mu}(\mathcal{A}_S)$ is a subset of $V(\mathcal{A}_S)\otimes\Z_{(p)}$ whenever $\mu_p = 0$.

The second follows from the construction of the space $V(\mathcal{A}_S)$, and the fact that the endomorphisms $\bm{V} \subset \underline{\End}(\bm{H})$ of  (\ref{eqn:hom special})  are grade shifting and commute with $C(L)$.

For the third, we first check that the quadratic form $Q(x)$ is positive definite. But this follows from the existence of a polarization on $\mathcal{A}_S$ such that every element of $V(\mathcal{A}_S)\otimes\Q$ is fixed under the corresponding Rosati involution; see the proof of \cite[(5.12)]{mp:reg}. We also have to show that, for each prime $p$, we have
\[
Q(x)\equiv Q(\mu)\pmod{\Z_{(p)}}.
\]
We can assume that we have chosen an embedding $L\hookrightarrow L^\diamond$ with $L^\diamond$ self-dual at $p$. In this case, with $\Lambda$ and $\tilde{\mu}_p$ as in Lemma~\ref{lem:mu def self-dual ind}, we find that, in terms of the quadratic form $Q^\diamond$ on $V(\mathcal{A}^\diamond_S)$, we have
\[
Q^\diamond(x+\tilde{\mu}_p)\equiv 0\pmod{\Z_{(p)}}.
\] 
It can be checked using Proposition~\ref{prop:functorial}(ii) that $V(\mathcal{A}_S)$ is orthogonal to $\tilde{\mu}_p$ as a subspace of $V(\mathcal{A}_S^\diamond)_\Q$. Therefore, we have
\begin{equation}\label{eqn:x tilde mu}
Q(x) = Q^\diamond(x)\equiv -Q^\diamond(\tilde{\mu}_p)\pmod{\Z_{(p)}}.
\end{equation}

On the other hand, recall that $\tilde{\mu}_p\subset \Lambda^\vee$ was defined to be a lift of the image of $\mu_p$ under the canonical isomorphism
\[
(L^\vee/L)\otimes\Z_{(p)}\xrightarrow{\simeq}\Lambda^\vee/\Lambda\otimes\Z_{(p)}
\]
associated with the embedding $L\hookrightarrow L^\diamond$. In other words, if $\mu'_p\in L^\vee$ is any lift of $\mu_p$, we have $\mu'_p+\tilde{\mu}_p\in L^\diamond$. So arguing just as in the previous paragraph, we find
\begin{equation}
\label{eqn:tilde mu mu p}
Q(\mu'_p)\equiv - Q^\diamond(\tilde{\mu}_p)\pmod{\Z_{(p)}}.
\end{equation}
Comparing~\eqref{eqn:x tilde mu} and~\eqref{eqn:tilde mu mu p} finishes the proof.
\end{proof}

As before, suppose that we are given an embedding of maximal quadratic spaces $L_0\hookrightarrow L$ with $\Lambda=L_0^{\perp}\subset L$.  In particular the inclusion $L^\vee \hookrightarrow L_0^\vee\oplus \Lambda^\vee$
induces   an injection
\[
L^\vee / (L_0\oplus \Lambda ) \hookrightarrow  (L_0^\vee/L_0)\oplus (\Lambda^\vee/ \Lambda).
\]

Fix an $\mathcal{M}_0$-scheme $S\to \mathcal{M}_0$. Then, by construction, the map $\Lambda\hookrightarrow\End( \mathcal{A}_S)$ from Proposition~\ref{prop:functorial} factors through an isometric embedding $\Lambda\hookrightarrow V( \mathcal{A}_S )$.
\begin{proposition}\label{prop:decomistionVmu}
Fix an $\mathcal{M}_0$-scheme $S\to \mathcal{M}_0$.
\begin{enumerate}

\item There is a canonical isometry
\begin{equation}\label{eqn:tensor_inject}
 V(\mathcal{A}_{0,S})\iso \Lambda^{\perp}\subset V( \mathcal{A}_S ),
\end{equation}
where here $\Lambda^\perp$ is calculated inside $V( \mathcal{A}_S )$.
\item
For every $\mu\in L^\vee/L$ and every
$(\mu_1,\mu_2)\in \bigl(\mu+L\bigr)/\bigl(L_0\oplus \Lambda\bigr)$
the map (\ref{eqn:tensor_inject}), tensored with $\Q$, restricts to an injection
\[
V_{\mu_1}(\mathcal{A}_{0,S}) \times ({\mu_2}+\Lambda) \hookrightarrow V_{\mu}(\mathcal{A}_S).
\]

\item
The above injections determine  a decomposition
\[
V_{\mu}(\mathcal{A}_S)
=\bigsqcup_{(\mu_1,\mu_2)\in (\mu+ L)/(L_0\oplus \Lambda)}  V_{\mu_1}(\mathcal{A}_{0,S}) \times
\bigl({\mu_2}+\Lambda\bigr).
\]
\end{enumerate}
\end{proposition}

\begin{proof}
Claim (i) follows from the definitions and Proposition \ref{prop:functorial};  see also~\cite[(7.15)]{mp:reg}.
In particular, any element $x\in V(\mathcal{A}_S)\otimes\Q$ admits a decomposition
\[
 x=(x_0,\nu)\in (V(\mathcal{A}_{0,S})\otimes\Q)\times (\Lambda\otimes\Q).
\]
Using this decomposition, it is an easy exercise to deduce claims (ii) and (iii) from the definitions.
The main input is Lemma~\ref{lem:mu def self-dual ind}, which permits us to study all spaces in question by embedding $L$ into a maximal quadratic lattice that is self-dual at any given prime $p$.
\end{proof}


\subsection{Special divisors}


For $m\in \Q_{>0}$ and $\mu\in L^\vee/L$, define the \emph{special cycle} $\mathcal{Z}(m, \mu) \to \mathcal{M}$
as the  stack over $\mathcal{M}$ with functor of points
\begin{equation}\label{special divisor}
\mathcal{Z}(m, \mu) (S) = \left\{ x \in V_\mu( \mathcal{A}_S) : Q(x) = m \right\}
\end{equation}
for any scheme $S\to \mathcal{M}$.
Note that, by (\ref{Q cong}), the stack (\ref{special divisor}) is  empty unless the image of $m$ in $\Q/\Z$ agrees with $Q(\mu)$.

For later purposes we also define the stacks $\mathcal{Z}(0, \mu)$ in exactly the same way. As the only special endomorphism $x$ with  $Q(x)=0$  is the zero map, we have
\[
\mathcal{Z}(0, \mu) =
\begin{cases}
 \emptyset & \hbox{{\rm if }} \mu\neq 0 \cr \mathcal{M} & \hbox{{\rm if }} \mu=0.\cr
\end{cases}
\]

The special cycles behave nicely under pullback by the morphism of
Proposition \ref{prop:functorial}.   Indeed, the following is an immediate consequence of Proposition \ref{prop:decomistionVmu}.

\begin{proposition}\label{prop:decomistionZmu}
Suppose that we are given an embedding of maximal quadratic spaces
$L_0\hookrightarrow L$, and let $\Lambda\subset L$ be the submodule
of vectors orthogonal to $L_0$.   Let $\mathcal{M}_0 \to \mathcal{M}$ be the corresponding morphism of Shimura varieties,  as in Proposition \ref{prop:functorial}.  Thus, for any $m\in\Q_{\ge 0}$ and $\mu\in L^\vee/L$ there is a special cycle $\mathcal{Z}(m,\mu)\to \mathcal{M}$, and
for any $m_1\in \Q_{\ge 0}$ and $\mu_1\in L_0^\vee/L_0$ there is a  special
cycle $\mathcal{Z}_0(m_1,\mu_1) \to \mathcal{M}_0$.

Then, there is an isomorphism of $\mathcal{M}_0$-stacks
\[
\mathcal{Z}(m, \mu )\times_{\mathcal{M}}\mathcal{M}_0\\
\simeq \bigsqcup_{ \substack{ m_1+m_2=m \\ (\mu_1,\mu_2)\in ( \mu+L)/(L_0\oplus \Lambda)}}
 \mathcal{Z}_0(m_1,\mu_1)  \times \Lambda_{m_2,\mu_2},
\]
where
\[
\Lambda_{m_2,\mu_2} = \{x\in {\mu_2}+\Lambda : Q(x)=m_2\},
\]
and $ \mathcal{Z}_0(m_1,\mu_1)  \times \Lambda_{m_2,\mu_2}$ denotes the disjoint union of
$\# \Lambda_{m_2,\mu_2}$ copies of $\mathcal{Z}_0(m_1,\mu_1)$.
\end{proposition}

\begin{proposition}
For every $m\in \Q_{\geq 0}$ and $\mu\in L^\vee/L$, the morphism $\mathcal{Z}(m,\mu)\to \mathcal{M}$ is relatively representable, finite, and unramified. 
\end{proposition}

\begin{proof}
Suppose that we have an embedding $L\hookrightarrow L^\diamond$ of maximal quadratic spaces with $\Lambda\subset L^\diamond$ the submodule of vectors orthogonal to $L$. Suppose also that $m^\diamond\in \Q_{\geq 0}$, $\mu^\diamond\in L^{\diamond,\vee}/L^{\diamond}$, and $\mu'\in \Lambda^\vee/\Lambda$ are such that 
\[
(\mu,\mu')\in (\mu^\diamond + L^\diamond)/(L\oplus \Lambda),
\]
and
\begin{align*}
Q(\mu^\diamond)  &\equiv m^\diamond\pmod{\Z} \\  Q(\mu')  &\equiv m^\diamond - m \pmod{\Z}.
\end{align*}

Let $\mathcal{Z}(m^\diamond,\mu^\diamond)\to \mathcal{M}^\diamond$ be the morphism of stacks associated with the triple $(L^\diamond, m^\diamond,\mu^\diamond)$. Suppose that we know that $\mathcal{Z}(m^\diamond,\mu^\diamond)\to \mathcal{M}^\diamond$ is relatively representable, finite, and unramified.  Then Proposition~\ref{prop:decomistionZmu} implies that $\mathcal{Z}(m,\mu)$ can be viewed as a closed and open substack of $\mathcal{Z}(m^\diamond,\mu^\diamond)\times_{\mathcal{M}^\diamond}\mathcal{M}$, and is thus finite and unramified over $\mathcal{M}$.

Therefore, by replacing $L$ with an appropriate  choice of $L^\diamond$, one reduces the problem to studying the morphism $\mathcal{Z}(m,\mu) \to \mathcal{M}$  under the additional assumption that $L$ is self-dual.  Of course this implies that $\mu=0$.

Suppose we have a Noetherian scheme $S$ and a morphism $S\to \mathcal{M}$.  Recall that the \'etale sheaf  $\underline{\End}(\mathcal{A}_S)$ on $S$ is represented by a formally unramified scheme over $S$, whose connected components are projective.  Formal unramifiedness follows from the rigidity of morphisms of abelian schemes \cite[Corollary 6.2]{MFK}.  Representability and projectiveness of the components follows from the theory of Hilbert schemes and the valuative criterion of properness, as in \cite[\S 6.1.5, \S 6.1.6]{Hida}.

As in~\cite[Prop. 6.13]{mp:reg}, $\mathcal{Z}(m,\mu)_S$  is represented by a formally unramified  $S$-scheme whose connected components are projective. More precisely, it is isomorphic to a union of connected components of $\underline{\End}(\mathcal{A}_S)$. This is a consequence of the fact that, for any $f\in\End(\mathcal{A}_S)$, the property of being special can be checked at any geometric point of a connected component of $S$, and similarly for the property $f\circ f=m$.

The only thing left to prove is that $\mathcal{Z}(m,\mu)_S$ is of finite type over $S$. For this, one first notes that, for any projective scheme $X$ over $S$, the Hilbert scheme parameterizing closed subschemes of $X$ with fixed Hilbert polynomial is projective, and hence of finite type,  over $S$. See~\cite[\S 5.6]{FGA}.  We will realize  $\mathcal{Z}(m,\mu)_S$ as a closed subscheme of such   a Hilbert scheme  with $X=\mathcal{A}_S \times_S \mathcal{A}_S$.  

Of course, the Hilbert polynomial depends on the choice of an ample line bundle on $X$.  We will choose our line bundle as follows: Let $\mathcal{A}^\vee_S$ be the dual abelian scheme, and let $\mathcal{P}$ be the Poincar\'e line bundle on $\mathcal{A}_S\times_S \mathcal{A}^\vee_S$. As in Proposition~\ref{Prop:positive_definite}, one can choose a polarization $\psi:\mathcal{A}_S\to \mathcal{A}^\vee_S$ such that the associated Rosati involution fixes all special endomorphisms of $\mathcal{A}_S$. Define an ample line bundle on $\mathcal{A}_S$ by
\[
\mathcal{O}(1) = (1\times \psi)^*\mathcal{P}.
\]
As usual, for any $n\in\Z$, set $\co(n) = \co(1)^{\otimes n}$.   Define an ample line bundle on $X$ by $p_1^*\co(1)\otimes p_2^*\co(1)$, where $p_i:X\to \mathcal{A}_S$ are the two projections.

Suppose that $s\to \mathcal{Z}(m,\mu)_S$ is a geometric point.  It  determines a geometric point  $s\to S$, along with a special endomorphism $x\in V(\mathcal{A}_s)$  with $Q(x)=m$.    The Hilbert polynomial of the graph $\Gamma_x \subset X_s$  is associated with the function
\[
n\mapsto h^0(\mathcal{A}_s,\co(n)\otimes x^*\co(n)).
\]
Since $x$ is fixed by the Rosati involution, one easily checks that $x^*\co(n) = \co(mn)$. Therefore the Hilbert polynomial in question is the one attached to
\[
n\mapsto h^0(\mathcal{A}_s,\co((m+1)n)),
\]
which clearly depends only on $m$.

Sending a point  of $\mathcal{Z}(m,\mu)_S$ to the graph of the corresponding special endomorphism defines a closed immersion of $\mathcal{Z}(m,\mu)_S$ into the Hilbert scheme of closed subschemes of $X$ with fixed Hilbert polynomial.  This proves that $\mathcal{Z}(m,\mu)_S$ is of finite type over $S$, and completes the proof.
\end{proof}

\begin{lemma}
For any positive $m$ and any $\mu$,  the complex orbifold  $\mathcal{Z}(m, \mu)(\C)$ just defined agrees with (\ref{eqn:Zmmu}).
\end{lemma}

\begin{proof}
Consider the uniformization $M(\C) = G(\Q) \backslash \mathcal{D} \times G(\A_f) / K$. By construction, for every
\[
z=(\bm{h},g)\in \mathcal{D} \times G(\A_f),
\]
 the fiber of $\bm{V}_{\Be,\Q}$ at $z$ is the $\Q$-vector space $V$ with Hodge structure determined by  the negative plane $\bm{h} \subset V_\R$. The $(0, 0)$ part  of $\bm{V}_{\Be, \Q,z}$ is $V \cap {\bm{h}}^\perp$. The fibers of the subsheaves $\bm{V}_{\rm Be}$ and $\bm{V}_{\rm Be}^\vee$ of $\bm{V}_{\Be,\Q}$ at $z$ are $\widehat{L}_g \cap V=L_g$ and $\widehat{L}_g^\vee \cap V=L_g^\vee$, respectively. Here, using the notation of \S\ref{s:the complex Shimura variety}, $\widehat{L}_g = g\bullet\widehat{L}$ and $\widehat{L}^\vee_g = g\bullet\widehat{L}^\vee$.

 By definition, we now have
\[
V(\mathcal{A}_{z}) = \mathrm{End}(\mathcal{A}_{z}) \cap \bm{V}_{\rm{Be},z} =
(\bm{V}_{\Be,z})^{(0,0)} = L_g \cap  \bm{h}^\perp.
\]
This proves the lemma for the trivial class $\mu$.

Now, since $K$ acts trivially on $L^\vee/L$, there is a canonical isometry $\bm{V}_{\Be}^\vee/\bm{V}_{\Be} \simeq L^\vee/L\otimes\bm{1}$ of torsion sheaves on $M(\C)$. This determines, for every $\mu \in L^\vee/L$,
a locally constant intermediate sheaf of sets  $\bm{V}_{\mu,\Be} \subset \bm{V}_{\Be}^\vee$
whose fiber  $\bm{V}_{\mu,\Be, z}$  is  $\mu_g + L_g$. It is not hard to deduce from Remark~\ref{rem:Vmu simpler} that $V_\mu(\mathcal{A}_z)$ consists precisely of those elements of $V(\mathcal{A}_z)_{\Q}$ whose Betti realization lands in $\bm{V}_{\mu,\Be}$. It follows that
\[
V_\mu(\mathcal{A}_{z}) = (\mu_g+L_g)\cap\bm{h}^{\perp},
\]
which proves the lemma for non-trivial $\mu$.
\end{proof}

We now show that the stacks $\mathcal{Z}(m,\mu)$ with $m>0$
define (\'etale local) Cartier divisors on $\mathcal{M}$.

\begin{proposition}\label{prop:localstrZmmu}
Suppose $m>0$.  Then \'etale locally on the source, $\mathcal{Z}(m,\mu)$ defines a Cartier divisor on $\mathcal{M}$.  More precisely:
around every geometric point of $\mathcal{M}$ there is an \'etale neighborhood $U$ such that for every connected component
$Z\subset \mathcal{Z}(m,\mu)_U$, the map $Z\to U$  is a  closed immersion  defined by a single non-zero equation.
\end{proposition}

\begin{proof}
Suppose that we are given a geometric point $z$ of $\mathcal{Z}(m,\mu)$,  corresponding to an element $x\in V_\mu(\mathcal{A}_z)$. 
Let $R=\co_{\mathcal{M},z}$ be the \'etale local ring of $\mathcal{M}$ at $z$. Since $\mathcal{Z}(m, \mu) \to \mathcal{M}$ is unramified, it follows that
\[
\co_{\mathcal{Z}(m,\mu) , z} \iso R/J_x
\] 
for some ideal  $J_x\subset R$.  Set  $S=\Spec (R/J_x )$, so that  $x \in V_\mu(\mathcal{A}_z)$
has a universal  deformation to $x_S \in V_\mu(\mathcal{A}_S)$.

Fix an embedding $L\hookrightarrow L^\diamond$ so that $ L^\diamond$ is maximal and self-dual at $p$ of signature $( n^\diamond,2)$, and let $\Lambda\subset
 L^\diamond$ be the subset of vectors orthogonal to $L$. Let ${\mathcal{A}^\diamond}\to \mathcal{M}^\diamond$ be the Kuga-Satake abelian scheme over the
integral model of the  Shimura variety determined by $ L^\diamond$.
Proposition \ref{prop:decomistionVmu} (with $\mathcal{A}_0$ and $\mathcal{A}$ replaced by $\mathcal{A}$ and  ${\mathcal{A}^\diamond}$) gives an inclusion
\[
V(\mathcal{A}_S) \oplus \Lambda \hookrightarrow V(\mathcal{A}^\diamond_S),
\]
which becomes an isomorphism after tensoring with $\Q$.  There is a
 $\nu\in \Lambda^\vee$ such that the vector
\[
x^\diamond_S := x_S + \nu\in  (V(\mathcal{A}_S)\otimes\Q) \oplus (\Lambda\otimes\Q)
\]
lies in the $\Z$-lattice $V(\mathcal{A}^\diamond_S)$. The crystalline realization $t_{\crys}(x^\diamond)$ allows us to canonically lift the de Rham realization
\[
t_{\dR}(x^\diamond_S ) \in \bm{V}^\diamond_{ \dR, R /J_x}
\]
to an element
\[
t_{\rm crys}(x^\diamond_S ) \in \bm{V}^\diamond_{\mathrm{ dR}, R /J_x^2},
\]
through the divided power thickening $S\hookrightarrow \Spec\big(R /J_x^2\big)$.

It follows from \cite[(5.15)]{mp:reg} that  $t_{\dR}(x^\diamond_S)$ is orthogonal to $\bm{V}^{\diamond,1}_{{\dR},S}$.   Thus, if
\[
\omega\in \bm{V}_{{\dR},R /J_x^2}^1 \iso  \bm{V}^{\diamond,1}_{{\dR},R  / J_x^2}
\]
is  an $R$-module generator,  we obtain an element $ \big[\omega, t_{\rm crys}(x^\diamond_S)\big] \in J_x/J_x^2 $ using the pairing on $\bm{V}^\diamond_{\dR}$. In
fact, Grothendieck-Messing theory~\cite[(5.17)]{mp:reg} shows that $\big[\omega, t_{\rm crys}( x^\diamond_S )\big]$ generates $J_x/J_x^2$. The element
\[
t_{\rm crys}(x_S ):=t_{\rm crys}( x^\diamond_S  )- \nu
\]
lies in   $\bm{V}^\vee_{{\dR},R /J_x^2}$, and the orthogonality relation   $[\omega, \nu]=0$ implies
\[
 [\omega, t_{\rm crys}(x_S)]  =  [\omega, t_{\rm crys}(x^\diamond_S) ].
 \]

\begin{lemma}\label{lem:localstrZmmualpha}
The element $ [\omega, t_{\rm crys}( x_S ) ]$ generates $J_x/J_x^2$ as an $R$-module. In particular, $J_x\subset R$ is a principal ideal.
\end{lemma}
\begin{proof}
The discussion above proves the first assertion. The second follows by Nakayama's lemma.
\end{proof}

Proposition \ref{prop:localstrZmmu} follows immediately from the lemma.
\end{proof}

\begin{remark}
It follows from Proposition \ref{prop:localstrZmmu} that
$\mathcal{Z}(m,\mu)$ is of pure dimension $\dim (\mathcal{M})-1$.  The
morphism to $\mathcal{M}$ is not itself a closed immersion,
but nevertheless    $\mathcal{Z}(m,\mu)$  defines a Cartier divisor on $\mathcal{M}$,
 as in \cite[\S 3.1]{BHY}.
\end{remark}


\section{Harmonic modular forms and special divisors}
\label{s:analysis}


Fix a maximal quadratic space $L$ over $\Z$ of signature $(n,2)$ with $n\ge 1$, set $V=L_\Q$,  and let $\mathcal{M}\to \Spec(\Z)$ be the corresponding integral
model of the Shimura variety (\ref{uniformization}) defined in  \S \ref{s:shimura varieties}.

This section is a rapid review of some results and  constructions of
Bruinier \cite{Bruinier}, Bruinier-Funke  \cite{BruinierFunke},  and Bruinier-Yang \cite{BY}.
In particular, we recall the construction of a divisor $\mathcal{Z}(f)$ and a Green function
$\Phi(f)$ on $\mathcal{M}$ from a  harmonic weak Maass form $f$.


\subsection{Vector valued modular forms}


Let $\mathfrak{S}_L$ be the (finite dimensional) space of $\C$-valued functions on $L^\vee/L$.
 As in \cite{BY} there is a Weil representation
\[
\omega_L: \widetilde{\SL}_2(\Z) \to \Aut_\C(\mathfrak{S}_L),
\]
where $\widetilde{\SL}_2(\Z)$ is the metaplectic double cover of $\SL_2(\Z)$.
Define  the conjugate action $\overline{\omega}_L$  by
$
\overline{\omega}_L (\gamma) \varphi =\overline{ \omega_L(\gamma)\overline{\varphi}}.
$
We denote by $\omega_L^\vee$ the contragredient action of $\widetilde{\SL}_2(\Z)$ on the complex linear  dual
$\mathfrak{S}_L^\vee$.

\begin{remark}
Our $\overline{\omega}_L$ is the representation denoted $\rho_L$ in \cite{Borcherds}, \cite{Bruinier}, \cite{BruinierFunke}, and \cite{BY}.  If we denote by
$-L$ the quadratic space over $\Z$ whose underlying $\Z$-module is $L$, but endowed with the quadratic form
$-Q$, then $\mathfrak{S}_L=\mathfrak{S}_{-L}$ as vector spaces, but $\omega_L=\rho_{-L}$.
\end{remark}

For a half-integer $k$, denote by $H_k(\omega_L)$ the $\C$-vector space of $\mathfrak{S}_L$-valued harmonic weak
Maass forms of weight $k$ and representation $\omega_L$, in the sense\footnote{ These satisfy a more restrictive
growth condition at the cusp than the \emph{weak Maass forms} of \cite[\S3]{BruinierFunke}.}  of \cite[\S 3.1]{BY}.
As in \cite[\S3]{BY} there are subspaces
\[
S_k(\omega_L)\subset M^!_k(\omega_L) \subset H_k(\omega_L)
\]
of cusp forms and  weakly modular forms.  Define similar spaces of modular forms for the representation $\overline{\omega}_L$. Bruinier and Funke \cite{BruinierFunke}  have defined a differential operator
\begin{equation}\label{xi op}
\xi : H_{k}(\omega_L) \to S_{2-k}( \overline{\omega}_L )
\end{equation}
by $\xi(f) = 2 i v^k \overline{ (\partial f / \partial \overline{\tau})}$, where $\tau=u+iv$ is the variable on the upper half-plane,
and have proved the exactness of the sequence
\[
0 \to M^!_{k}(\omega_L) \to H_{k}(\omega_L) \map{\xi } S_{2-k}( \overline{\omega}_L ) \to 0.
\]

Every  $f(\tau)\in H_k(\omega_L)$ has an associated formal $q$-expansion
\[
f^+(\tau) = \sum_{ \substack{ m\in \Q \\  -\infty \ll m } } c_f^+(m) \cdot q^m,
\]
called the \emph{holomorphic part of $f$}, with $c_f^+(m) \in \mathfrak{S}_L$.
The delta functions $\varphi_\mu$ at elements $\mu \in L^\vee/L$ form a basis for $\mathfrak{S}_L$,
and so each coefficient $c_f^+(m)$ can be decomposed  as a sum
\[
c_f^+(m) = \sum_{\mu\in L^\vee/L} c_f^+(m, \mu) \cdot \varphi_\mu
\]
with  $c_f^+(m, \mu)\in \C$.
The transformation laws satisfied by $f$ imply that
\[
c_f^+(m,\mu) = c_f^+(m,-\mu),
\]
and also that  $c_f^+(m,\mu)=0$, unless  the image of $m$ in $\Q/\Z$ is equal to $-Q(\mu)$.


\subsection{Divisors and Green functions}
\label{ss:green functions}


Suppose we are given a harmonic weak Maass form
\[
f(\tau) \in H_{1-\frac{n}{2}} ( \omega_L)
\]
having  \emph{integral principal part},  in the sense that   $ c_f^+(m,\mu)\in \Z$ for all $\mu\in L^\vee/L$ and  all $m<0$.
 Define  a divisor on $\mathcal{M}$ by
\begin{equation*}
\mathcal{Z}(f) = \sum_{ m\in \Q_{>0} } \sum_{ \mu\in L^\vee/L}  c_f^+(-m,\mu) \mathcal{Z}(m, \mu).
\end{equation*}

Bruinier \cite{Bruinier}, following ideas of Borcherds \cite{Borcherds} and Harvey-Moore, constructed a Green function $\Phi(f)$ for
$\mathcal{Z}(f)$  as follows (see also~\cite[\S4]{BY}).  Let $G=\GSpin(V)$ as in \S \ref{s:shimura varieties}.
For each coset $g\in G(\A_f) /K$ there is a Siegel theta function
\[
\theta_L(\tau , z,g) : \mathcal{H} \times \mathcal{D} \to \mathfrak{S}_L^\vee
\]
as in \cite[(2.4)]{BY}, which is $\Gamma_g$-invariant in the variable $z\in\mathcal{D}$,
and transforms in the variable $\tau=u+iv\in\mathcal{H}$ like a modular form
of representation $\omega_L^\vee$.

The regularized theta integral
\[
\Phi(f,z,g) =  \int_\mathcal{F}^\mathrm{reg} \big\{ f(\tau) ,\theta_L(\tau,z,g) \big\} \frac{du\, dv}{v^2}
\]
of \cite[(4.7)]{BY} defines a $\Gamma_g$-invariant function on $\mathcal{D}$, where
\[
\{ \cdot,\cdot\} :\mathfrak{S}_L \times \mathfrak{S}_L^\vee \to \C
\]
is the tautological pairing.  Letting $g$ vary and using the uniformization
(\ref{uniformization}) yields a function $\Phi(f)$ on the orbifold $\mathcal{M}(\C)$.  This function is
smooth on the complement of $\mathcal{Z}(f)$, and
has a  logarithmic singularity along $\mathcal{Z}(f)$  in the sense that
for any  local equation $\Psi(z)=0$ defining $\mathcal{Z}(f)(\C)$,  the function
\[
\Phi(f,z) +\log|\Psi(z) |^2
\]
on  $\mathcal{M}(\C) \smallsetminus \mathcal{Z}(f)(\C)$ has a smooth extension across $\mathcal{Z}(f)(\C)$.

\begin{remark}\label{rem:miracle value}
Amazingly, the regularized integral defining $\Phi(f,z,g)$ converges at every point of $\mathcal{D}$,
and so the function $\Phi(f)$ has a well-defined value \emph{even at points of the divisor} $\mathcal{Z}(f)(\C)$.
The value of the discontinuous function $\Phi(f)$ at the points of $\mathcal{Z}(f)(\C)$ will play a key role in our later
calculations of improper intersection.
\end{remark}

Among the harmonic weak Maass forms of weight $1-n/2$ are the  {\em Hejhal-Poincar\`e series}
\[
F_{m,\mu} (\tau) \in H_{1-\frac{n}{2}} (\omega_L),
\]
 defined for all  $\mu\in L^\vee/L$ and positive $m \in \Z + Q(\mu)$,
which may be rescaled to have  holomorphic parts
 \[
 F_{m,\mu}^+(\tau) = \frac{1}{2}( q^{-m} \varphi_\mu+q^{-m} \varphi_{-\mu} ) + O(1).
 \]
In particular  $\mathcal{Z}( F_{m,\mu} ) =  \mathcal{Z}(m,\mu).$
See   \cite[Def. 1.8 \& Prop. 1.10]{Bruinier} and \cite[Remark 3.10]{BruinierFunke}
for details.


\subsection{Eisenstein series and the Rankin-Selberg integral}
\label{ss:analytic}


Fix a rational negative $2$-plane   $V_0\subset V$, and let $L_0=V_0\cap L$.
We assume that $C^+(L_0)$ is the maximal order in the quadratic imaginary field  $C^+(V_0)$.
This implies, in particular, that  $L_0$ is a maximal quadratic space. Let
\[
\Lambda = \{ x\in L :  x \perp L_0  \},
\]
so that  $L_0\oplus \Lambda \subset L$ with finite index,
and  $\Lambda$ is also a maximal quadratic space.

Let $\mathfrak{S}_\Lambda$ be the space of complex-valued functions on $\Lambda^\vee/\Lambda$,
and let $\omega_\Lambda:\widetilde{\SL}_2(\Z) \to \Aut(\mathfrak{S}_\Lambda)$ be the Weil representation.
Let $\mathfrak{S}_\Lambda^\vee$ be the complex linear dual of $\mathfrak{S}_\Lambda$, and
for each $m\in \Q$  define the representation number
$R_\Lambda(m)  \in \mathfrak{S}_\Lambda^\vee$ by
\[
R_\Lambda(m,\varphi) = \sum_{ \substack{ x\in \Lambda^\vee \\ Q(x)=m } } \varphi(x)
\]
for any $\varphi\in \mathfrak{S}_\Lambda$.
As in  \cite[\S2]{BY}, these  representation numbers are the Fourier coefficients of a
 vector-valued modular form
 \[
\Theta_\Lambda(\tau) =\sum_{m\in \Q} R_\Lambda(m) q^m \in M_{\frac{n}{2}}( \omega_\Lambda^\vee ).
\]
Abbreviate
$
R_\Lambda(m,\mu) = R_\Lambda(m,\varphi_\mu),
$
where $\varphi_\mu \in\mathfrak{S}_\Lambda$  is the characteristic function of  $\mu\in \Lambda^\vee/\Lambda$, so that
\[
R_\Lambda(m ,\mu ) = \#  \{x\in {\mu}+\Lambda : Q(x)=m\}.
\]

Let $F(\tau)$ be a cuspidal modular form
\[
F(\tau) = \sum_{m\in \Q_{>0}} b_F(m) q^m \in S_{1+\frac{n}{2}} ( \overline{\omega}_L).
\]
 Using the natural inclusion
$L^\vee \subset  L_0^\vee\oplus \Lambda^\vee$, every  function on $L^\vee$ extends to a function on
$L_0^\vee\oplus \Lambda^\vee$ identically zero off of $L^\vee$.  Thus, there is a natural ``extension by zero" map
\begin{equation}\label{transfer}
  \mathfrak{S}_L \to \mathfrak{S}_{L_0 \oplus \Lambda},
\end{equation}
which allows us to view $\overline{b_F(m)} \in  \mathfrak{S}_{L_0 \oplus \Lambda}$.
The inclusion $\Lambda^\vee\to L_0^\vee\oplus\Lambda^\vee$ defines a  \emph{restriction of functions} homomorphism
$\mathfrak{S}_{L_0 \oplus \Lambda} \to  \mathfrak{S}_\Lambda$. Using this, we can view $R_\Lambda(m) \in \mathfrak{S}_{L_0\oplus\Lambda}^\vee$.
This allows us to define  the \emph{Rankin-Selberg convolution $L$-function}
\[
L(F,\Theta_\Lambda,s) =
\Gamma\left( \frac{s+n}{2} \right)  \sum_{m\in \Q^+}
\frac{\big\{ \overline{b_F(m)} , R_\Lambda(m) \big\} }{ (4\pi m)^{( s+n)/2 }},
\]
in which the pairing on the right is  the tautological pairing
\begin{equation}\label{pairing}
\{ \cdot,\cdot\} :\mathfrak{S}_{L_0\oplus \Lambda} \times \mathfrak{S}_{L_0\oplus\Lambda}^\vee \to \C.
\end{equation}

Define an $\mathfrak{S}_{L_0}^\vee$-valued Eisenstein series $E_{L_0}(\tau,s)$ by
\[
E_{L_0}(\tau,s,\varphi) = \sum_{ \gamma \in B\backslash \SL_2(\Z)}
 ( \omega_{L_0}(\gamma)\varphi)(0)
\cdot  \frac{\mathrm{Im}(\gamma  \tau)^{s/2} }{c\tau + d}
\]
for all $\varphi\in \mathfrak{S}_{L_0}$, where $B\subset \SL_2(\Z)$ is the subset of upper
triangular matrices, and $\omega_{L_0}$ is the Weil representation of $\SL_2(\Z)$ on
$\mathfrak{S}_{L_0}$.  The Eisenstein series $E_{L_0}(\tau,s)$, which is precisely the \emph{incoherent}
Eisenstein series denoted $E_{L_0}(\tau,s;1)$ in  \cite[\S2.2]{BY},
transforms in the variable $\tau$ like a weight $1$ modular
form of representation $\omega_{L_0}^\vee$.  It is initially defined for $\mathrm{Re}(s) \gg 0$, but has meromorphic
continuation to all $s$, and vanishes at $s=0$.
The usual Rankin-Selberg unfolding method shows that
\[
L(F,\Theta_\Lambda,s) =
\int_{ \SL_2(\Z)\backslash \mathcal{H} } \big\{
\overline{F}, E_{L_0}(s) \otimes \Theta_\Lambda
\big\}  \, \frac{du\, dv}{v^{1-\frac{n}{2} }},
\]
where the pairing on the right is (\ref{pairing}), and we are using the isomorphism
\[
\mathfrak{S}_{L_0}^\vee \otimes \mathfrak{S}_\Lambda^\vee \iso
\mathfrak{S}_{L_0 \oplus\Lambda}^\vee
\]
to view $E_{L_0}(s) \otimes \Theta_\Lambda$
as a smooth  modular form valued in $\mathfrak{S}_{L_0 \oplus \Lambda}^\vee$.

The central derivative $E_{L_0}'(\tau,0)$ is a harmonic $\mathfrak{S}_{L_0}^\vee$-valued
function transforming like a weight $1$ modular form of representation $\omega^\vee_{L_0}$.
Denote the  holomorphic part of the central derivative by
\[
\mathcal{E}_{L_0}(\tau) = \sum_{ m\in \Q } a_{L_0}^+(m) \cdot q^m.
\]
We will sometimes  abbreviate
$
a_{L_0}^+(m,\mu)=a_{L_0}^+(m,\varphi_\mu),
$
where $\varphi_\mu \in\mathfrak{S}_{L_0}$ is the characteristic function of $\mu\in L_0^\vee/L_0$.
Note that our $a_{L_0}^+(m,\mu)$ is exactly the $\kappa(m,\mu)$ of \cite[(2.25)]{BY}.

Abbreviate  $\kk=C^+(V_0)$, and  recall  that $\kk$ is a quadratic imaginary field.
For any positive rational number $m$, define a finite set of primes
\begin{equation}\label{diff def}
\mathrm{Diff}_{L_0}(m) = \{ \mbox{primes }p <\infty :  L_0 \otimes_\Z\Q_p \mbox{ does not represent }m\}.
\end{equation}
Note that all primes in this set are non-split in $\kk$.
Because $L_0$ is negative definite, the set $\mathrm{Diff}_{L_0}(m)$ has odd
cardinality;  in particular, $\mathrm{Diff}_{L_0}(m)\not=\emptyset$.   For  $m\in\Z_{>0}$, set
\[
\rho(m) = \#\{ \mbox{nonzero ideals } \mathfrak{m} \subset \co_\kk : \mathrm{N}(\mathfrak{m}) = m\}.
\]
We will extend $\rho$ to a function on all rational numbers by setting $\rho(m)=0$ for all $m\in\Q\smallsetminus\Z_{>0}$.

The coefficients $a_{L_0}^+(m)$ were computed by Schofer \cite{Sch}, but those formulas
contain a  minor misstatement.  This misstatement was corrected when Schofer's formulas
were reproduced in \cite[Theorem 2.6]{BY}, but the formula for the constant term $a_{L_0}^+(0)$ in \cite{BY}
 is misstated\footnote{The formula for $a_{L_0}^+(0,0)=\kappa(0,0)$ in  \cite{BY} appears to agree with the
formula in  \cite{Sch}.  In fact it does not, as the two papers use different normalizations for the completed
$L$-function $\Lambda(\chi,s)$.  It is Schofer's formula that is correct.}.
We record now the corrected formulas.

 \begin{proposition}[Schofer]  \label{prop:eisenstein coeff}
Let $d_\kk$, $h_\kk$, and $w_\kk$
be the discriminant of $\kk$, the class number of $\kk$, and number of roots of unity in $\kk$, respectively, and assume that $d_\kk$ is odd.
 The coefficients $a^+_{L_0}(m)$ are as follows:
 \begin{enumerate}
 \item
The constant term  is
\[
 a_{L_0}^+(0,\varphi)   =
\varphi(0) \cdot \left(   \gamma +  \log\left|\frac{ 4 \pi }{d_\kk } \right| -  2 \frac{ L ' (\chi_\kk,0)}{L(\chi_\kk,0)} \right).
\]
 Here $\gamma=-\Gamma'(1)$ is the Euler-Mascheroni constant, and $\chi_\kk$ is the
quadratic Dirichlet character associated with $\kk/\Q$.
\item
If  $m>0$ and   $\mathrm{Diff}_{L_0}(m)=\{p\}$ for a single prime $p$, then
\[
a^+_{L_0}(m, \varphi )
=
- \frac{   w_\kk}{ 2 h_\kk}\cdot  \rho \left( \frac{m  |d_\kk| }{p^\epsilon}  \right) \cdot  \ord_p(pm)\cdot \log(p)
\sum_{
\substack{  \mu \in L_0^\vee / L_0
\\ Q(\mu) =m } } 2^{s(\mu) } \varphi(\mu) .
\]
On the right hand side, $s(\mu)$ is the number of  primes  $\ell \mid \mathrm{disc}(L_0)$ such that  $\mu \in L_{0,\ell}$,
\[
\epsilon  = \begin{cases}
1 & \mbox{if $p$ is inert in $\kk$} \\
0 &\mbox{if $p$ is ramified in $\kk$,}
\end{cases}
\]
and $Q(\mu)=m$ is understood as an equality  in $\Q/\Z$.
\item
If $m>0$ and $|\mathrm{Diff}_{L_0}(m) | >1$, or if $m<0$, then $a_{L_0}^+(m)=0$.
\end{enumerate}
\end{proposition}


\section{Complex multiplication  cycles}
\label{s:CM section}


Keep $L_0 \oplus \Lambda \subset L$ as in \S  \ref{ss:analytic},
so that $L$ is a maximal lattice of signature $(n,2)$,  $L_0$ is a sublattice of signature $(0,2)$, and
\[
\Lambda = L_0^\perp = \{ x\in L : x\perp L_0 \}.
\]
Abbreviate $V_0=L_0\otimes_\Z\Q$ and $V=L\otimes_\Z \Q$.  The even part $C^+(V_0)$ of the Clifford algebra is
isomorphic to a quadratic imaginary field $\kk$, and we continue to assume that $C^+(L_0)$  is its maximal order.

In this section we construct  an $\co_\kk$-stack $\mathcal{Y}$
from the quadratic space $L_0$. Assuming that the discriminant of
$\kk$ is odd, the inclusion $L_0 \subset L$ induces a morphism
$
\mathcal{Y} \to \mathcal{M}_{\co_\kk},
$
where $\mathcal{M}$ is the stack over $\Z$ determined by the quadratic space $L$.  We define special divisors on $\mathcal{Y}$, compute their degrees explicitly,
and compare them with the special divisors on $\mathcal{M}$.


\subsection{Preliminaries}


Fix once and for all an isomorphism of $C^+(L_0)$ with the maximal
order $\co_\kk$ in a quadratic imaginary subfield $\kk \subset
\C$.  Let $d_\kk$, $h_\kk$, and $w_\kk$ be as in Proposition
\ref{prop:eisenstein coeff}. The $\Z/2\Z$-grading on $C(L_0)$  takes the form
\begin{equation}\label{C_0 decomp}
C(L_0) = \co_\kk \oplus L_0,
\end{equation}
and $L_0$ is both a left and right $\co_\kk$-module; the two actions are related by
$x \alpha = \overline{\alpha} x$ for every $x \in L_0$ and  $\alpha \in \co_\kk$.

There is a fractional $\co_\kk$-ideal $\mathfrak{a}$ and an  isomorphism $L_0 \iso \mathfrak{a}$
of left $\co_\kk$-modules identifying the quadratic form on $L_0$ with the form $-\rm{N}(\cdot)/\rm{N}(\mathfrak{a})$ on $\mathfrak{a}$, and
an easy exercise shows that
\begin{equation}\label{easy dual}
L_0^\vee =\mathfrak{d}_\kk^{-1} L_0,
\end{equation}
where $\mathfrak{d}_\kk$ is the different of $\kk/\Q$.


\subsection{The $0$-dimensional Shimura variety}
\label{s:CM Shimura variety}


Let $T$ be the group scheme over $\Z$ with functor of points
\[
 T(R)=(C^+(L_0)\otimes_{\Z}R)^\times
\]
for any $\Z$-algebra $R$.
Thus  $T_{\Q}=\GSpin(V_0)$ is a torus of rank $2$ and can be identified with the maximal subgroup of $G$
that acts trivially on the subspace $\Lambda_{\Q}\subset V$.
Set $K_0=T(\widehat{\Z})$.

Using our fixed  embedding $ \co_\kk \subset \C$,  left multiplication
makes $V_{0,\R}$ into a complex vector space.  This determines  an orientation on $V_{0,\R}$, and the  oriented  negative plane
$V_{0,\R}\subset V_{\R}$ determines a point $\bm{h}_0\in\mathcal{D}$.
The isomorphism of (\ref{grassman}) then associates to $\bm{h}_0$ an
isotropic line in $V_{0,\C}$;  explicitly, this is the line on which the
 \emph{right} action of $\co_\kk$  agrees with the inclusion  $\co_\kk\subset \C=\End_\C(V_{0,\C})$.

The pair $(T_\Q, \bm{h}_0)$ is a Shimura datum with reflex field $\kk\subset \C$, and the  $0$-dimensional complex orbifold
\[
Y(\C) = T(\Q) \backslash \{ \bm{h}_0\} \times T(\A_f) / K_0
\]
 is the complex fiber of  a $\kk$-stack $Y$.
As in \S \ref{s:the complex Shimura variety}, the representation  of $T_\Q$ on $V_0$ and the lattice
\[
 L_0 \otimes \widehat{\Z}  \subset V_0 \otimes \A_f
\]
  determine a variation of polarized $\Z$-Hodge structures $(\bm{V}_{0,\rm Be},\Fil^\bullet\bm{V}_{0,\dR,Y(\C)})$  of type
$(-1,1)$, $(1,-1)$  over $Y(\C)$.

Similarly, the representation of $T_\Q$ on $C(V_0)$  by left multiplication,  together with the lattice
\[
C(L_0)_{ \widehat{\Z} } \subset C(V_0)_{ \A_f },
\]
determine  a variation of $\Z$-Hodge structures $(\bm{H}_{0,\rm Be},\Fil^\bullet\bm{H}_{0,\dR,\mathcal{Y}(\C)})$ of type $(-1,0)$, $(0,-1)$.
As in \S \ref{sec:canonical}, this is the homology of an abelian scheme
\[
\mathcal{A}_{0,Y}\to Y
\]
of relative dimension $2$, equipped with a right $C(L_0)$-action and a compatible $\Z/2\Z$-grading
\[
\mathcal{A}_{0,Y}=\mathcal{A}_{0,Y}^+\times\mathcal{A}_{0,Y}^-.
\]
As before, we will refer to this abelian scheme as the \emph{Kuga-Satake abelian scheme} over $Y$.


\subsection{The integral model}
\label{s:CM Shimura variety integral}


We now define  integral models of  $Y$ and $\mathcal{A}_{0,Y}$ over $\co_\kk$.

\begin{definition}
If $S$ is an $\co_\kk$-scheme, an \emph{$\co_\kk$-elliptic curve} over $S$ is an elliptic curve $\mathcal{E}\to S$
endowed with an action of $\co_\kk$ in  such a way that the induced action  on the $\co_S$-module $\Lie(\mathcal{E})$
is through the structure morphism $\co_\kk\to \co_S$.
\end{definition}

\begin{remark}
Our convention is that the  $\co_\kk$-action on  any $\co_\kk$-elliptic curve is written on the \emph{right}.
\end{remark}

Denote by  $\mathcal{Y}$ the $\co_\kk$-stack classifying $\co_\kk$-elliptic curves over $\co_\kk$-schemes.
Every point $y\in \mathcal{Y}(\C)$ has automorphism group $\co_\kk^\times$, and
\[
 \sum_{ y \in \mathcal{Y}(\C) } \frac{1}{\#\Aut(y) } =\frac{h_\kk}{w_\kk}.
\]
As in \cite[Proposition~5.1]{KRY1}, the stack $\mathcal{Y}$ is finite \'etale over $\co_\kk$.  In particular
$\mathcal{Y}$ is regular of dimension $1$, and is flat over $\co_\kk$.

The positive graded part $\mathcal{A}^+_{0,Y}$ of the Kuga-Satake abelian scheme, with its right action of
$C^+(L_0)= \co_\kk$,   is an $\co_\kk$-elliptic curve over $Y$  (for the obvious extension of this notion over algebraic stacks.)
Therefore, it defines a morphism  $Y \to \mathcal{Y}$. By checking on complex points, we see that this map defines an
 isomorphism
 $
 Y \iso \mathcal{Y}_\kk.
 $
Again by checking on complex points, one can show that the natural map
\[
\mathcal{A}^+_{0,Y} \otimes_{\co_\kk} C(L_0) \to \mathcal{A}_{0,Y}
\]
is an isomorphism, from which it follows that
\[
\mathcal{A}^-_{0,Y} \iso \mathcal{A}^+_{0,Y} \otimes_{\co_\kk} L_0.
\]

From these observations, we can use the universal $\co_\kk$-elliptic curve $\mathcal{E}\to\mathcal{Y}$
 to extend the Kuga-Satake abelian scheme  $\mathcal{A}_{0,Y} \to Y$ to an abelian scheme
\[
\mathcal{A}_0 =\mathcal{E}\otimes_{\co_\kk} C(L_0)
\]
over $\mathcal{Y}$ of relative dimension $2$.  By construction,  $\mathcal{A}_0$ carries a right action of $C(L_0)$  and a $\Z/2\Z$-grading
$
\mathcal{A}_0=\mathcal{A}_0^+  \times   \mathcal{A}_0^-
$
induced by  (\ref{C_0 decomp}),  where $\mathcal{A}^+_0=\mathcal{E}$ and  $\mathcal{A}^-_0=\mathcal{E}\otimes_{\co_\kk} L_0$.

\begin{remark}
Note that  the tensor product $\mathcal{E}\otimes_{\co_\kk} L_0$
is with respect to the \emph{left} action of $\co_\kk$ on $L_0$, and
 $\co_\kk\subset C(L_0)$
acts on $\mathcal{E}\otimes_{\co_\kk} L_0$  through \emph{right} multiplication
on $L_0$.  From the relation $x \alpha  =  \overline{\alpha} x$ for $x\in L_0$
and $\alpha\in \co_\kk$, it follows that $\co_\kk$ acts on $\Lie(\mathcal{A}_0^-)$
through the \emph{conjugate} of the structure morphism $\co_\kk \to \co_\mathcal{Y}$.
\end{remark}

There are obvious analogues on $\mathcal{Y}$ of the sheaves  $\bm{H}$ defined in \S \ref{sec:integral}.
Denote by $\bm{H}_{0,\dR}$  the first  de Rham homology of
$\mathcal{A}_0$ relative to $\mathcal{Y}$.  Similarly, denote by
$\bm{H}_{0,p, \rm et}$  the first  $p$-adic \'etale homology of $\mathcal{A}_0$
relative to $\mathcal{Y}[1/p]$, and  by $\bm{H}_{0,\crys}$ the first  crystalline homology of
the special fiber of $\mathcal{A}_0$ at a prime~$p$ relative to $\mathcal{Y}_{\F_p}$.  These sheaves
come with right actions of $C(L_0)$, and  $\Z/2\Z$-gradings  $\bm{H}_0=\bm{H}^+_0\times \bm{H}^-_0$
in such a way that $\bm{H}^-_0=\bm{H}^+_0\otimes_{\co_\kk} L_0$.
As in the higher dimensional case, over $\mathcal{Y}(\C)$ the de Rham and  \'etale sheaves   arise from the Betti realization $\bm{H}_{0,\rm Be}$.

Define
\[
\bm{V}_0\subset \underline{\mathrm{End}}_{C(L_0)} (\bm{H}_0)
\]
as the submodule of grade shifting endomorphisms. As can be checked from the complex uniformization, over $\mathcal{Y}(\C)$  this definition recovers the local system $\bm{V}_{0,\rm Be}$ (for the Betti realization) and the vector bundle $\bm{V}_{0,\dR, Y(\C)}$ (for the de Rham realization.) Indeed, this amounts to the fact that the embedding of $L_0$ into $\End(C(L_0))$ via  left multiplication  identifies it with the space of grade shifting endomorphisms of $C(L_0)$ that commute with right multiplication by $C(L_0)$.

Restriction to $\bm{H}^+_0$ defines an isomorphism
\begin{equation}\label{eq:bmVn=0}
\bm{V}_0\iso  \underline{\mathrm{Hom}}_{\co_\kk}(\bm{H}^+_0, \bm{H}^-_0).
\end{equation}
As  $\bm{H}^\pm_0$ is a locally free  $\co_\kk\otimes \bm{1}$-module of rank one,
the same is true of $\bm{V}_0$.  In particular,  $\bm{V}_0$ is a locally free $\bm{1}$-module of rank two. In the de Rham case, the Hodge filtration on $\bm{H}_{0,\dR}$ determines a submodule
  $\bm{V}_{0,\dR}^1\subset \bm{V}_{0,\dR}$ by
\[
\bm{V}_{0,\dR}^1=\bm{V}_{0,\dR} \cap
{\rm Fil}^1\underline{\mathrm{End}}(\bm{H}_{0,\dR}).
\]

\begin{definition}
As in Definition \ref{def:cotautological},   the line bundle $\bm{V}_{0,\dR}^1\subset \bm{V}_{0,\dR}$
is the \emph{tautological bundle} on $\mathcal{Y}$, and its dual is the \emph{cotautological bundle}.
\end{definition}

The following proposition makes the cotautological bundle more explicit.

\begin{proposition}
There are  canonical isomorphisms of $\co_\mathcal{Y}$-modules
\begin{equation}\label{cotautonCM}
( \bm{V}_{0,\dR}^1)^\vee \iso  \Lie(\mathcal{A}_0^+) \otimes \Lie(\mathcal{A}_0^-) \iso   \Lie(\mathcal{E})^{\otimes 2} \otimes  \bm{L}_0 ,
\end{equation}
where $\bm{L}_0=\co_\mathcal{Y} \otimes_{\co_\kk} L_0$.
\end{proposition}

\begin{proof}
We have canonical isomorphisms of line bundles over $\mathcal{Y}$
\begin{align*}
\bm{H}_{0,\dR}^+/\Fil^0\bm{H}_{0,\dR}^+ & \iso \Lie(\mathcal{A}_0^+)\\
\bm{H}_{0,\dR}^-/\Fil^0\bm{H}_{0,\dR}^- & \iso \Lie(\mathcal{A}_0^-),
\end{align*}
and the (de Rham realization of the) unique principal polarization on $\mathcal{A}_0^-$ induces an isomorphism of $\co_\mathcal{Y}$-modules
\[
  \Fil^0\bm{H}_{0,\dR}^-  \iso ( \bm{H}_{0,\dR}^-/\Fil^0\bm{H}_{0,\dR}^- )^\vee = \Lie(\mathcal{A}_0^-)^\vee.
\]
The isomorphism~(\ref{eq:bmVn=0}) restricts to
\begin{align}\label{taut simplify}
\bm{V}_{0,\dR}^1 & \iso
\mathrm{Fil}^1\underline{\Hom}_{\co_\kk}(\bm{H}_{0,\dR}^+, \bm{H}_{0,\dR}^-)  \nonumber \\
&\iso \underline{\Hom}_{\co_\kk} (  \bm{H}_{0,\dR}^+/\Fil^0\bm{H}_{0,\dR}^+  ,
  \Fil^0\bm{H}_{0,\dR}^-  )\\
&\iso \underline{\Hom}( \Lie(\mathcal{A}_0^+) ,  \Lie(\mathcal{A}_0^-)^\vee ), \nonumber
\end{align}
and dualizing yields   the first isomorphism in (\ref{cotautonCM}).
The second isomorphism  is clear from the definition  of $\mathcal{A}_0^\pm$ and the relation $\mathcal{A}^-_0=\mathcal{E}\otimes_{\co_\kk} L_0$.
\end{proof}

\begin{remark}\label{rem:taut ok action}
It is clear from (\ref{taut simplify}) that the isotropic line $\bm{V}_{0,\dR}^1  \subset \bm{V}_{0,\dR}$ is stable under the action of $\co_\kk$ induced by (\ref{eq:bmVn=0}), and that $\co_\kk$
acts on this line through the structure morphism $\co_\kk \to \co_\mathcal{Y}$.
\end{remark}


\subsection{Special endomorphisms}


If $\mathcal{E}_1$ and $\mathcal{E}_2$ are any  elliptic curves with right actions of $\co_\kk$, we make
$\Hom(\mathcal{E}_1,\mathcal{E}_2)$ into a right $\co_\kk$-module via
\[
(f \cdot \alpha) ( - ) = f( - ) \cdot \alpha
\]
for all $f\in \Hom(\mathcal{E}_1,\mathcal{E}_2)$ and $\alpha\in \co_\kk$.
With this convention, for any $\mathcal{Y}$-scheme $S$  there is an isomorphism of right $\co_\kk$-modules
\[
\Hom(\mathcal{E}_{S} , \mathcal{E} _S   \otimes_{\co_\kk} L_0 )  \iso \End (\mathcal{E}_S) \otimes_{\co_\kk} L_0.
\]
This  restricts to an isomorphism
\[
\Hom_{\co_\kk}(\mathcal{E}_S , \mathcal{E}_S \otimes_{\co_\kk} L_0 )
 \iso \End_{\overline{\co}_\kk} (\mathcal{E}_S) \otimes_{\co_\kk} L_0,
\]
 which we rewrite as
 \begin{equation} \label{herm twisting}
\Hom_{\co_\kk}(\mathcal{A}^+_{0,S} , \mathcal{A}^-_{0,S} ) \iso \End_{\overline{\co}_\kk} (\mathcal{E}_S) \otimes_{\co_\kk} L_0.
 \end{equation}
Here the subscripts $\co_\kk$ and $\overline{\co}_\kk$ indicate $\co_\kk$-linear and $\co_\kk$-conjugate-linear maps,
respectively.   The quadratic form $\deg$ on the left hand side of  (\ref{herm twisting}) corresponds to the quadratic form
\begin{equation}\label{quadratic degree}
 x\otimes v  \mapsto   - \deg(x)  \cdot Q(v)
\end{equation}
on the right hand side (which is positive definite, as $\deg$ and $-Q$ are positive definite.)

Elements of  $\End_{\overline{\co}_\kk} (\mathcal{E}_S) \otimes_{\co_\kk} V_0$
are \emph{special quasi-endomorphisms} of $\mathcal{E}_S$. For every $\mu\in L_0^\vee$ define  a collection
of special quasi-endomorphisms  by
\[
V_\mu(\mathcal{E}_S)=
\left\{  x  \in \End_{\overline{\co}_\kk} (\mathcal{E} _S)  \otimes_{\co_\kk} L_0^\vee
:  x - \mathrm{id}\otimes  \mu  \in \End (\mathcal{E} _S)  \otimes_{\co_\kk} L_0
 \right\}.
\]
This collection depends only on the coset $\mu \in L_0^\vee / L_0$.
 We will sometimes abbreviate
\[
V(\mathcal{E}_S)=\End_{\overline{\co}_\kk}(\mathcal{E}_S)\otimes_{\co_\kk}L_0
\]
for the space defined by the coset $\mu=0$.

\begin{remark}
For any  $S\to\mathcal{Y}$ as above, let
$
V(\mathcal{A}_{0,S})\ \subset \End(\mathcal{A}_{0,S})
$
be the submodule of $C(L_0)$-linear grade-shifting endomorphisms of $\mathcal{A}_{0,S}$.
Equivalently,  $x\in \End(\mathcal{A}_{0,S})$  lies in $V(\mathcal{A}_{0,S})$ if and only if
all  of its homological realizations  lie in the subsheaf
$
\bm{V}_0  \subset \underline{\End}(\bm{H}_0).
$
Restriction to the positive graded part identifies
\begin{equation}\label{eqn:A0 E special}
V(\mathcal{A}_{0,S}) \iso \Hom_{\co_\kk}(\mathcal{A}^+_{0,S} , \mathcal{A}^-_{0,S} )
\iso  \End_{\overline{\co}_\kk} (\mathcal{E}_S) \otimes_{\co_\kk} L_0 \iso V(\mathcal{E}_S),
\end{equation}
where the middle isomorphism is (\ref{herm twisting}).   
\end{remark}


\subsection{Degrees of special divisors}


In this subsection we define the special divisors
$
\mathcal{Z}_0(m,\mu) \to \mathcal{Y}
$
and compute their degrees.

 For a non-negative rational number $m$ and a class $\mu\in L^\vee_0/L_0$,
 define  $\mathcal{Z}_0(m,\mu)$ to be the  $\co_\kk$-stack  that   assigns to an $\co_\kk$-scheme $S$ the
groupoid of pairs $(\mathcal{E}_S,x)$ in which $\mathcal{E}_S$ is an $\co_\kk$-elliptic curve over $S$, and
$x\in V_\mu(\mathcal{E}_S)$ satisfies $Q(x)=m$.   Note that for  $m=0$ we have
\[
\mathcal{Z}_0(0, \mu) =
\begin{cases}
 \emptyset & \hbox{{\rm if }} \mu\neq 0 \\
 \mathcal{Y} & \hbox{{\rm if }} \mu=0.
\end{cases}
\]

The following theorem is due to Kudla-Rapoport-Yang \cite{KRY1} in the special case where
$\kk$ has prime discriminant and $\mu=0$.  Subsequent generalizations and variants can be found in
\cite{KY}, \cite{BHY}, and \cite{Ho}.

\begin{theorem}\label{thm:X degree}
Assume that $d_\kk$ is odd, and recall the finite set of primes $\mathrm{Diff}_{L_0}(m)$ of  (\ref{diff def}).
For $m>0$ the stack $\mathcal{Z}_0(m,\mu)$ has the following properties:
\begin{enumerate}
\item If $| \mathrm{Diff}_{L_0}(m) | >1$, or if $Q(\mu)\not=m$ in $\Q/\Z$,  then $\mathcal{Z}_0(m,\mu)=\emptyset$.

\item
 If $\mathrm{Diff}_{L_0}(m) =\{ p\} $ and $Q(\mu) =m$ in $\Q/\Z$,
then $\mathcal{Z}_0(m,\mu)$ has dimension $0$ and is supported
in characteristic $p$.  Furthermore,
\[
\frac{\log(\mathrm{N}(\mathfrak{p}) )}{\log(p)} \sum_{ z \in \mathcal{Z}_0(m,\mu)(\overline{\F}_\mathfrak{p}) }
\frac{\mathrm{length} (  \co_{\mathcal{Z}_0(m,\mu),z} )   }{ \# \Aut(z)  }=
 2^{s(\mu)-1}  \cdot \ord_p(p m) \cdot  \rho\left(  \frac{m |d_\kk| }{p^\epsilon} \right),
\]
where $\mathfrak{p}$ is the unique prime of $\kk$ above $p$, and $\epsilon$, $\rho$, and $s(\mu)$
are as in Proposition \ref{prop:eisenstein coeff}.
\end{enumerate}
In either case
\[
\sum_{ \mathfrak{p} \subset \co_\kk } \log(\mathrm{N}(\mathfrak{p} ))
\sum_{ z \in \mathcal{Z}_0(m,\mu)(  \overline{\F}_\mathfrak{p}) }
\frac{ \mathrm{length} (  \co_{\mathcal{Z}_0(m,\mu),z} )  }{ \# \Aut(z)  }
=    - \frac{h_\kk}{w_\kk} \cdot a_{L_0}(m, \mu ) .
\]
\end{theorem}

\begin{proof}
Let $\mathfrak{p}$ be a prime of $\co_\kk$, and let $p$ be the rational prime below it.
Suppose we have a point
\begin{equation}\label{deg point}
z=(\mathcal{E},x)\in \mathcal{Z}_0(m,\mu)(  \overline{\F}_\mathfrak{p} ).
\end{equation}
The existence of $x\in \End_{\overline{\co}_\kk} ( \mathcal{E}  )\otimes_{\co_\kk} L_0$  already implies that $\mathcal{E}$ is supersingular,
$p$ is non-split in $\kk$, and $\End(\mathcal{E})_\Q$ is the
quaternion division algebra ramified at $\infty$ and $p$.

As $\mathcal{O}_\kk\hookrightarrow \End(E)$ we can extend this to an inclusion $\kk \hookrightarrow \End(E)_\mathbb{Q}$. Since $\kk = \mathbb{Q}(\sqrt{d_\kk})$ there is an $\eta \in \mathbb{Q}^\times$ such that 
\[ \End(E)_\mathbb{Q} = \left( \frac{d_\kk, \eta}{\mathbb{Q}}\right).\] 
Consequently $\End_{\overline{\mathcal{O}}_\kk}(E)_\mathbb{Q} = \kk\cdot j$, where $j^2 = \eta$ and $j\alpha = \bar \alpha j, \forall \alpha \in \mathcal{O}_\kk$. Further, $\End_{\overline{\mathcal{O}}_\kk}(E)_\mathbb{Q} =  \kk\cdot j $ is a one dimensional $\kk$-vector space equipped with the quadratic form~$\deg$, which is just the reduced norm of the quaternion algebra.  As $j$ has reduced norm $-\eta$, we deduce that there is an isomorphism of quadratic spaces 
\[
\End_{\overline{\mathcal{O}}_\kk}(E)_\mathbb{Q} \cong W_{-\eta},
\] 
where for every $\gamma \in \mathbb{Q}^\times$ we let $W_\gamma= \kk$ equipped with the quadratic form $Q( w) =\gamma w \overline{w}.$

As $\End(E)_\mathbb{Q}$ is a quaternion algebra ramified precisely at~$p$ and~$\infty$, it follows that~$\eta$ is a local norm from~$\kk$ everywhere but at~$p$ and~$\infty$.

Recalling that $V_0=L_0\otimes_\Z\Q$, any choice of left $\kk$-module generator $v\in V_0$ identifies $V_0 \iso W_\gamma$, where $\gamma= Q(v)$.  Recalling (\ref{quadratic degree}), we deduce that there an isomorphism of quadratic spaces 
\begin{equation}\label{simple herms}
 \End_{ \overline{\co}_\kk}  ( \mathcal{E} ) \otimes_{\co_\kk} V_0  \iso W_{\eta \gamma}.
\end{equation}
The existence of $x$ shows that the quadratic space on the left represents $m$ globally, and hence $\eta \gamma/m$ is a norm from $\kk$.  This implies that $\gamma/m$, hence $\gamma m$, is a norm from $\kk$ everywhere locally except at $\infty$ and $p$, from which we
deduce  $\mathrm{Diff}_{L_0}(m) =\{ p\}$.   The equality $m=Q(\mu)$ in $\Q/\Z$ can be checked directly, but it also follows from (\ref{Q cong}) and Proposition~\ref{prop:0_decomistionVmu}.

We have now proved that $\mathcal{Z}_0(m,\mu)(  \overline{\F}_\mathfrak{p} )\not=\emptyset$
only when $\mathrm{Diff}_{L_0}(m) =\{ p\}$ and  $m=Q(\mu)$ in $\Q/\Z$, and hence have proved (i).
The proof of (ii) requires a lemma.

\begin{lemma}\label{lem:no mu}
The existence of a point (\ref{deg point}) implies $\mu \in L_{0,p}$.
\end{lemma}

\begin{proof}
By (\ref{easy dual}) we may assume that  $p$ is ramified in $\kk$.  Fix a generator $\delta_\kk\in \mathfrak{d}_\kk$,
so that $L_0^\vee = \delta_\kk^{-1} L_0$.   Let $E_p$ be the $p$-divisible group of $\mathcal{E}$ with its induced $\co_\kk$-action.

If $\mu\not\in L_{0,p}$ then $\mu \delta_\kk $ generates $L_{0,p}$ as an $\co_{\kk,p}$-module.
If we use this generator   to identify
\[
\End(E_p) \otimes_{\co_\kk} L_0\iso \End(E_p)
\]
and multiply  both sides of  the relation
$
x  -  \mathrm{id}\otimes \mu \in  \End(E_p) \otimes_{\co_\kk} L_0
$
on the right by $\delta_\kk$, we obtain
\begin{equation}\label{quaternion cong}
y - \mathrm{id} \in \End(E_p) \cdot\delta_\kk,
\end{equation}
where we have set  $y=x\delta_\kk\in \End_{\overline{\co}_\kk}(E_p)$.

The ring   $\End(E_p)$ is the maximal order in the division quaternion algebra over $\Q_p$, and this maximal
order is a (noncommutative) discrete valuation ring with  $\End(E_p)\cdot \delta_\kk $ as its  unique maximal ideal.
Using   $yv=\overline{v} y$ for all $v\in \co_{\kk,p}$, an exercise with quaternion algebras shows that $y$ has reduced trace $0$.  Taking the reduced trace of (\ref{quaternion cong}) now
implies that $-2\cdot \mathrm{id}$ lies in the maximal ideal of $\End(E_p)$, which is absurd, as
 our assumption that $d_\kk$ is odd implies $p>2$.
\end{proof}

It follows from Lemma \ref{lem:no mu} and the Serre-Tate theorem that we may completely ignore $\mu$ in the
deformation problem that computes the length of the \'etale local ring at $z$.  Gross's results on canonical liftings
\cite{Gr} therefore imply
\[
\frac{\log(\mathrm{N}(\mathfrak{p}) )}{\log(p)} \cdot
\mathrm{length} (  \co_{\mathcal{Z}_0(m,\mu),z} )  =  \ord_p(pm)
\]
for every point (\ref{deg point}).

If $\ord_p(m)<0$ then the relation $m = Q(\mu)$ in $\Q/\Z$ implies $\mu \not\in L_{0,p}$, and hence, using Lemma \ref{lem:no mu},  both sides of the desired equality in (ii) are equal to $0$.  Thus it only remains to prove
\begin{equation}\label{point count}
 \sum_{ z \in \mathcal{Z}_0(m,\mu)(  \overline{\F}_\mathfrak{p} ) }
\frac{1  }{ \# \Aut(z)  }=  2^{s(\mu)-1}  \cdot  \rho\left(  \frac{m |d_\kk| }{p^\epsilon} \right)
\end{equation}
under the assumption $\ord_p(m)\ge 0$.
The  argument is very similar to that of \cite[Proposition 3.1]{KY},  \cite[Proposition 2.18]{HY},
\cite[Theorem 5.5]{BHY},  and  \cite[Theorem 3.5.3]{Ho}, and hence we omit some details.

Fix one elliptic curve $\mathcal{E}$ over $ \overline{\F}_\mathfrak{p} $ with complex multiplication  by $\co_\kk$.
As the ideal class group of $\kk$ acts simply transitively on the elliptic curves with
complex multiplication by $\co_\kk$, we find
\[
\sum_{ z \in \mathcal{Z}_0(m,\mu)(  \overline{\F}_\mathfrak{p}  ) }  \frac{1  }{ \# \Aut(z)  }
 =  \sum_{\mathfrak{a} \in \mathrm{Pic}(\co_\kk) }
\sum_{ \substack{  x\in V_\mu(  \mathcal{E} \otimes_{\co_\kk} \mathfrak{a}  ) \\ \deg(x)=m }}
\frac{1}{w_\kk}  .
\]

As in the argument surrounding (\ref{simple herms}), the assumption $\mathrm{Diff}_{L_0}(m)=\{p\}$
implies that there is some vector
\[
x\in   \End_{ \overline{\co}_\kk}(\mathcal{E}) \otimes_{\co_\kk} V_0
\]
with $Q(x)=m$, and the group $\kk^1=\{ s\in \kk^\times : s\overline{s}=1\}$  acts transitively (by right multiplication on $V_0$) on the set of all such $x$.
Using $t\mapsto t /\overline{t}$ to identify $\Q^\times \backslash \kk^\times \iso \kk^1$ and arguing as in the proof of \cite[Proposition 2.18]{HY},  we obtain a factorization into \emph{orbital integrals}
\begin{align}
\sum_{ z \in \mathcal{Z}_0(m,\mu)(  \overline{\F}_\mathfrak{p}  ) }
\frac{1  }{ \# \Aut(z)  }
& =
\frac{1}{w_\kk} \sum_{\mathfrak{a} \in \mathrm{Pic}(\co_\kk) }
\sum_{ t \in \Q^\times \backslash \kk^\times } \bm{1}_{  V_\mu(  \mathcal{E}  \otimes_{\co_\kk} \mathfrak{a})  } (  x t \overline{t}^{-1} )   \nonumber
\\
 & =
 \frac{1}{2}  \prod_{ \ell } \sum_{t  \in \Q_\ell^\times \backslash \kk_\ell^\times / \co_{\kk,\ell}^\times }
 \bm{1}_{   V_\mu(  E _\ell )   } (    x t \overline{t}^{-1}  )  .
\label{orbital factorization}
\end{align}
 Here   $\bm{1}$ indicates  characteristic function, $E_\ell$ is the $\ell$-divisible group of $\mathcal{E}$, and
\[
V_{\mu} (E_\ell) = \{ y\in \End_{\overline{\co}_\kk}(E_\ell) \otimes_{\co_\kk} L_0^\vee
: y- \mathrm{id}\otimes \mu \in \End(E_\ell) \otimes_{\co_\kk} L_0  \}.
\]
On the right hand side of the first equality of (\ref{orbital factorization}) we are viewing
\[
x t \overline{t}^{-1} \in
 \End_{ \overline{\co}_\kk}(\mathcal{E}) \otimes_{\co_\kk} V_0
\iso
 \End_{ \overline{\co}_\kk}(\mathcal{E} \otimes_{\co_\kk} \mathfrak{a}) \otimes_{\co_\kk} V_0
\]
using the quasi-isogeny from $\mathcal{E}$ to $\mathcal{E} \otimes_{\co_\kk} \mathfrak{a}$
defined by $P \mapsto P \otimes 1$.

A choice of  left $\co_{\kk,\ell}$-module isomorphism  $L_{0,\ell}\iso \co_{\kk,\ell}$ identifies
\[
V_{\mu} (E_\ell)= \{ y\in \End_{\overline{\co}_\kk}(E_\ell )  \cdot  \mathfrak{d}_\kk^{-1} : y- \mu \in \End(E_\ell) \},
\]
where  $\mu \in L_{0,\ell}^\vee = \mathfrak{d}_{\kk,\ell}^{-1}$ is viewed as a quasi-endomorphism of $E_\ell$ using the complex multiplication.
From now on we use this identification  to view
$x\in \End_{\overline{\co}_\kk}(E_\ell ) \otimes_{\Z_\ell} \Q_\ell$
and $\mu \in \kk_\ell$.
Using (\ref{quadratic degree}), the condition $m=Q(\mu)$ in $\Q/\Z$ translates to $\mathrm{Nrd}(x) = - \mathrm{Nrd}( \mu )$ in $\Q_\ell/\Z_\ell$, where $\mathrm{Nrd}$ is the reduced norm on the quaternion algebra
\[
\End(E_\ell) \otimes_{\Z_\ell} \Q_\ell  = \kk_\ell \oplus  \big( \End_{ \overline{\co}_\kk} (E_\ell) \otimes_{\Z_\ell} \Q_\ell \big) .
\]
An explicit description (in terms of this decomposition)
of the maximal order $\End(E_\ell)$ can be found in \cite[Lemma 3.3]{KY}.

Using the explicit description of $\End(E_\ell)$, the local factors on the right hand side of (\ref{orbital factorization}) can  be computed directly.  We demonstrate the method for a prime $\ell$ ramified in $\kk$.
For such an $\ell$ we have
\[
\Q_\ell^\times \backslash \kk_\ell^\times / \co_{\kk,\ell}^\times  =\{ 1 ,t\}
\]
where $t\in \co_{\kk,\ell}$ is any uniformizer.  As $\ell>2$, we may choose $t$
so that $\overline{t}=-t$, and so
\[
\sum_{t  \in \Q_\ell^\times \backslash \kk_\ell^\times / \co_{\kk,\ell}^\times }
 \bm{1}_{   V_\mu(  E _\ell )   } (    x t \overline{t}^{-1}  )
 =
 \bm{1}_{   V_\mu(  E _\ell )   } \left(    x   \right)
+ \bm{1}_{   V_\mu(  E _\ell )   } \left(  -  x   \right) .
\]
We now consider the two sub-cases $\ell=p$ and $\ell\not=p$.

If $\ell=p$, our assumption  $\ord_p(m)\ge 0$ and the relation
$m=Q(\mu)$ in $\Q/\Z$ imply that $\mu \in  \co_{\kk,p}$, and that
$\mathrm{Nrd}(x) \in \Z_p$.   As $\End(E_p)$ is the maximal order in the ramified
quaternion algebra over $\Q_p$, this implies that $x\in \End(E_p)$.
From this it follows that both $x$ and $-x$ lie in $V_\mu(E_p)$, and so
\[
\sum_{t  \in \Q_\ell^\times \backslash \kk_\ell^\times / \co_{\kk,\ell}^\times }
 \bm{1}_{   V_\mu(  E _\ell )   }  (    x t \overline{t}^{-1}  ) =2.
 \]

Now suppose $\ell\not=p$. If $\mu \in \co_{\kk,\ell}$  then
$\mathrm{Nrd}(\mu) \in \Z_\ell$, and hence  $\mathrm{Nrd}(x)\in \Z_\ell$.
The description of $\End(E_\ell)$ in \cite[Lemma 3.3]{KY} now implies that
$x\in \End(E_\ell)$, and again  both $x$ and $-x$ lie in $V_\mu(E_p)$.
If $\mu \not \in \co_{\kk,\ell}$ then $\ell \cdot \mathrm{Nrd}(\mu) \in \Z_\ell^\times$, and hence
the relation $\mathrm{Nrd}(x) + \mathrm{Nrd}(\mu) \in \Z_\ell$
noted above implies  $ \ell\cdot \mathrm{Nrd}(x)  \in \Z_\ell^\times$.
Using the same relation, and the description of $\End(E_\ell)$ in \cite[Lemma 3.3]{KY},
one can show that exactly one of  $x\pm \mu$  lies in $\End(E_\ell)$.
Thus  exactly one of $x$ and $-x$
lies in $V_\mu(E_\ell)$, and we have proved
\[
\sum_{t  \in \Q_\ell^\times \backslash \kk_\ell^\times / \co_{\kk,\ell}^\times }
 \bm{1}_{   V_\mu(  E _\ell )   }  (    x t \overline{t}^{-1}  )
 =\begin{cases}
 1 & \mbox{if $\mu  \not\in L_{0,\ell}$} \\
 2 & \mbox{if $\mu  \in L_{0,\ell}$}.
 \end{cases}
\]

From similar  calculations at the unramified primes, one can deduce the relation
\[
\prod_\ell \sum_{t  \in  \Q_\ell^\times \backslash \kk_\ell^\times / \co_{ \kk, \ell} ^\times }
 \bm{1}_{  V_\mu(  E _\ell )    } \big(   x  t \overline{t}^{-1} \big)  =
2^{s(\mu)}  \cdot  \rho\left(  \frac{m |d_\kk| }{p^\epsilon} \right)
\]
(the right hand side has an obvious product factorization, and one verifies equality prime-by-prime.)
The equality (\ref{point count})  follows immediately.

The final claim of the theorem follows by comparing the above formulas with Proposition \ref{prop:eisenstein coeff}.
\end{proof}


\subsection{Mapping to the orthogonal Shimura variety}
\label{ss:CM to GSpin}


Recall from \S \ref{sec:canonical} the $\GSpin$ Shimura variety $M$ over $\Q$ defined by the quadratic space $L$.  In
 \S \ref{sec:integral} we defined the integral model  $\mathcal{M}$ over $\Z$,  and the Kuga-Satake abelian scheme
 $\mathcal{A} \to \mathcal{M}$.  The integral model $\mathcal{M}$ carries over it the collection of sheaves
 $
 \bm{V} \subset \underline{\End}( \bm{H})
 $
 of (\ref{eq:bmV}).

 The inclusion $L_0 \to L$ extends to a homomorphism $C(V_0) \to C(V)$ of Clifford algebras,
 which then induces a morphism $\GSpin(V_0) \to \GSpin(V)$.
 The theory of canonical models shows that the induced map of complex orbifolds $Y(\C)\to M(\C)$
 descends to a finite and unramified morphism of $\kk$-stacks
 \begin{equation}\label{CM cycle generic}
 Y\to M_{\kk}
 \end{equation}
In this section we assume that  the discriminant of $\kk$ is odd.
We will extend this morphism to a morphism $\mathcal{Y}\to\mathcal{M}_{\co_\kk}$ of integral models, and explain the
how the various structures (special cycles, sheaves, and
Kuga-Satake abelian schemes)  on the source and target are
related.

As noted earlier,  $\mathcal{Y}$ is finite \'etale over $\co_\kk$.  The proof of Proposition \ref{prop:0_functorial} below  will require
a slightly different perspective on this, couched in the language of local models.

Let $\mathrm{M}^{\rm loc}(L_0)$ be the quadric over $\Z$ that parameterizes isotropic lines in $L_0$.
There is an isomorphism of $\Z$-schemes
\begin{equation}\label{eqn:ML0_easy}
\Spec( \co_\kk )   \iso    \mathrm{M}^{\rm loc}(L_0)
\end{equation}
determined as follows. Consider the map
$
L_0\otimes_\Z\co_\kk \rightarrow L_0
$
defined by  $v\otimes\alpha \mapsto v\overline{\alpha}.$
The kernel of this map is an isotropic rank one $\co_\kk$-module corresponding to a map $\Spec( \co_\kk) \to\mathrm{M}^{\rm loc}(L_0)$,
which is the desired isomorphism~\eqref{eqn:ML0_easy}.

Now, consider the functor $\mathcal{P}_0$ on $\mathcal{Y}$-schemes parameterizing $\co_\kk$-equivariant isomorphisms of
rank two vector bundles
\[
\co_\kk \otimes_\Z \co_{\mathcal{Y}} \iso \bm{H}_{0,\dR}^+.
\]
This is represented over $\mathcal{Y}$ by a $T$-torsor $\mathcal{P}_0\to\mathcal{Y}$. Indeed, we only need to observe that, \'etale locally on $\mathcal{Y}$, $\bm{H}_{0,\dR}^+$ is a free $\co_\kk\otimes_\Z \co_{\mathcal{Y}}$-module of rank one.

Suppose that we are given a scheme $U\to\mathcal{Y}$ over $\mathcal{Y}$ and a section $U\to \mathcal{P}_0$, corresponding to an
$\co_\kk$-equivariant isomorphism
\[
\xi_0: \co_\kk \otimes_\Z \co_U \iso \bm{H}_{0,\dR}^+\vert_U.
\]
This determines an isomorphism
\[
 L_0 \otimes_\Z \co_U \iso \bm{H}_{0,\dR}^+\vert_U \otimes_{\co_\kk} L_0 = \bm{H}_{0,\dR}^-\vert_U.
\]
Combining the above isomorphisms with the canonical isomorphism
\[
L_0\otimes_\Z\co_U \iso \underline{\Hom}_{\co_\kk} ( \co_\kk \otimes_\Z\co_U , L_0 \otimes_\Z \co_U)
\]
yields an  $\co_\kk$-linear isometry
\[
\xi_0:L_0\otimes_\Z\co_U \iso \underline{\Hom}_{\co_\kk}\bigl(\bm{H}^+_{0,\dR}\vert_U,\bm{H}_{0,\dR}^-\vert_U \bigr) =
\bm{V}_{0,\dR}.
\]
The preimage $\xi_0^{-1}\bigl(\bm{V}^1_{0,\dR}\bigr)\subset L_0\otimes\co_U$ is an isotropic line, and so corresponds to a point in $\mathrm{M}^{\rm loc}(L_0)(U)$. We therefore obtain a map $U\to\mathrm{M}^{\rm loc}(L_0)$.
It is not hard to see that the composition of this map with the isomorphism~\eqref{eqn:ML0_easy} is just the structure map $U\to\Spec( \co_\kk)$.
The main point is that, as in Remark \ref{rem:taut ok action},  $\co_\kk$ acts on $\bm{V}^1_{0,\dR}$
through the structure map $\co_\kk \to \co_\mathcal{Y}$,
and on the quotient $\bm{V}_{0,\dR}/\bm{V}^1_{0,\dR}$ through the conjugate of the structure map.  Hence  the quotient $\bm{V}_{0,\dR}/\bm{V}^1_{0,\dR}$ is obtained from $\bm{V}_{0,\dR}$ via base-change along the augmentation map
\[
\co_\kk\otimes\co_{\mathcal{Y}} \to\co_{\mathcal{Y}}
\]
defined by $\alpha\otimes f \mapsto \overline{\alpha}f.$

In particular, since $\mathcal{Y}$ is \'etale over $\co_\kk$, if $U$ is \'etale over $\mathcal{Y}$, then the map $U\to\mathrm{M}^{\rm loc}(L_0)$ will also be \'etale. This can also be shown directly using Grothendieck-Messing theory.

The following result is analogous to Proposition \ref{prop:functorial}.

\begin{proposition}\label{prop:0_functorial}
The morphism (\ref{CM cycle generic}) extends uniquely to a finite and unramified morphism
\begin{equation}\label{CM integral cycle}
\mathcal{Y}\to\mathcal{M}_{\co_\kk},
\end{equation}
and  there is a canonical $C(L)$-equivariant isomorphism
\begin{equation}\label{0_abscheme_isom}
\mathcal{A}_0\vert_{\mathcal{Y}} \otimes_{C(L_0)} C(L) \iso
\mathcal{A}\vert_{\mathcal{Y}}
\end{equation}
of $\Z/2\Z$-graded abelian schemes  over $\mathcal{Y}$. In
particular,
\begin{equation}\label{0_realization tensor}
 \bm{H}_0\vert_{\mathcal{Y}}\otimes_{C(L_0)}C(L)\iso\bm{H}\vert_{\mathcal{Y}}.
\end{equation}

Furthermore, set $\bm{\Lambda} =\Lambda\otimes \bm{1}$, and recall the space of special endomorphisms
\[
V(\mathcal{A}\vert_{\mathcal{Y}})\subset
\End_{C(L)}(\mathcal{A}\vert_{\mathcal{Y}} )
\]
 of \S \ref{ss:specialend}.
There is a canonical embedding $\Lambda\subset V
(\mathcal{A}\vert_{\mathcal{Y}})$ with the following properties:
\begin{enumerate}
\item Its homological realization exhibits $\bm{\Lambda} \subset
\bm{V}\vert_{\mathcal{Y}}$ as a local direct summand. \item The
injection
\[
\underline{\mathrm{End}}_{C(L_0)}(\bm{H}_0\vert_{\mathcal{Y}}) \to
\underline{\mathrm{End}}_{C(L)}(\bm{H}\vert_{\mathcal{Y}})
\]
induced by (\ref{0_realization tensor}) identifies
$\bm{V}_0\vert_{\mathcal{Y}}$ with the   submodule of all elements
of $\bm{V}\vert_{\mathcal{Y}   }$ anticommuting with all elements
of  $\bm{\Lambda}$.  Furthermore
$
\bm{V}_0\vert_{\mathcal{Y}} \subset \bm{V}\vert_{\mathcal{Y} }
$
is locally a direct summand.

\item  In the de Rham case, the inclusion
$
\bm{V}_{0,\dR}\vert_{\mathcal{Y}}\subset\bm{V}_{\dR}\vert_{\mathcal{Y}}
$
identifies
$
\bm{V}_{0,\dR}^1\vert_{\mathcal{Y}} =
\bm{V}_{\dR}^1\vert_{\mathcal{Y}}.
$
\end{enumerate}
\end{proposition}

\begin{proof}
Just as in \eqref{eqn:rat tensor}, there is a  $C(L)$-equivariant isomorphism
\begin{align}\label{eqn:0_rat tensor}
\mathcal{A}_0 \vert_Y \otimes_{C(L_0)}C(L) \iso \mathcal{A}\vert_Y
\end{align}
of $\Z/2\Z$-graded abelian schemes over $Y$.

Extending the map of (\ref{CM cycle generic})  of $\kk$-stacks to
a map (\ref{CM integral cycle}) of $\co_\kk$-stacks is equivalent
to extending the map of $\Q$-stacks $Y\to M$ to a map of
$\Z$-stacks $\mathcal{Y}_{\Z}\to\mathcal{M}$. Here we write
$\mathcal{Y}_{\Z}$ for the stack $\mathcal{Y}$ viewed over
$\Spec(\Z)$.

To do this, fix a prime $p>0$, and assume first that $L_{\Z_{(p)}}$ is self-dual. As observed in \S \ref{sec:integral}, an appropriate choice of $\delta\in
C^+(L)\cap C(V)^\times$  determines a degree $d$ polarization of the Kuga-Satake abelian scheme $\mathcal{A}_M\to M$, inducing a finite unramified map
\[
i_{\delta}: M\to\mathcal{X}^\siegel_{2^{n+1},d,\Q}
\]
to a moduli stack of polarized abelian varieties.  The integral model  $\mathcal{M}_{\Z_{(p)}}$ is then simply
the normalization of $\mathcal{X}^\siegel_{2^{n+1},d,\Z_{(p)}}$ in $M$.

The composition
\[
Y\to M\to \mathcal{X}^\siegel_{2^{n+1},d,\Q}
\]
determines a polarized abelian scheme over $Y$, which is precisely (\ref{eqn:0_rat tensor}).   The polarized abelian scheme
$\mathcal{A}_0 \vert_{\mathcal{Y}_{\Z_{(p)}}}  \  \otimes_{C(L_0)}C(L)$ defines an extension of (\ref{eqn:0_rat tensor}) to
$\mathcal{Y}_{\Z_{(p)}}$, which shows that the above composition
extends to a  morphism
\[
\mathcal{Y}_{\Z_{(p)}}\to  \mathcal{X}^\siegel_{2^{n+1},d,\Z_{(p)}}.
\]
Since $\mathcal{Y}$ is regular, and hence normal, this extension must lift to a map
$
\mathcal{Y}_{\Z_{(p)}}\to\mathcal{M}_{\Z_{(p)}},
$
and by construction \eqref{eqn:0_rat tensor} extends to an isomorphism
\[
 (\mathcal{A}_0\otimes_{C(L_0)} C(L) )\vert_{\mathcal{Y}_{\Z_{(p)}}} \iso \mathcal{A}\vert_{\mathcal{Y}_{\Z_{(p)}}}.
\]

The construction of the embedding
$
 \Lambda\hookrightarrow \bigl(\mathcal{A}\vert_{\mathcal{Y}_{\Z_{(p)}}}\bigr),
$
as well as the proofs of assertions (i), (ii) and (iii) over $\mathcal{Y}_{\Z_{(p)}}$, now proceed exactly as in Proposition~\ref{prop:functorial}.

It remains to show (under our self-duality assumption on $L_{\Z_{(p)}}$) that the map $\mathcal{Y}_{\Z_{(p)}}\to\mathcal{M}_{\Z_{(p)}}$ is finite and unramified. For this, using Proposition~\ref{prop:Mcheck} and the above discussion of local models, we see that the map is \'etale locally on the source isomorphic to an \'etale neighborhood of the closed immersion $\mathrm{M}^{\rm loc}(L_0)\to\mathrm{M}^{\rm loc}(L)$.
This shows that $\mathcal{Y}_{\Z_{(p)}}$ is unramified over $\mathcal{M}_{\Z_{(p)}}$. That it is also finite is immediate from the fact that $\mathcal{Y}$ is finite over $\co_\kk$.

It remains to deal with the primes $p$ where $L_{\Z_{(p)}}$ is not self-dual. For this, embed $L$ as an isometric direct summand of a maximal lattice $ L^\diamond$
of signature $( n^\diamond,2)$ for which  $ L^\diamond_{\Z_{(p)}}$ is self-dual. Let $\mathcal{M}^\diamond$ be the regular integral model over $\Z$ for the Shimura
variety $ M^\diamond$ associated with $ L^\diamond$. As in (\ref{sec:integral}), let $\check{\mathcal{M}}_{\Z_{(p)}}$ be the normalization of
$\mathcal{M}^\diamond_{\Z_{(p)}}$ in $M$.

From what we have shown above, the composition
$
Y\to M\to M^\diamond
$
extends to a map  $\mathcal{Y}_{\Z_{(p)}}\to\mathcal{M}^\diamond_{\Z_{(p)}},$ which, since $\mathcal{Y}$ is a normal stack over $\Z$, must lift to a finite map
$\mathcal{Y}_{\Z_{(p)}}\to\check{\mathcal{M}}_{\Z_{(p)}}$.  As in the proof of Proposition~\ref{prop:functorial}, we find that that this lift must map
$\mathcal{Y}_{\Z_{(p)}}$ to $\mathcal{M}_{\Z_{(p)}}$ and must also be unramified, being, \'etale locally on the source, isomorphic to an \'etale neighborhood of
 \begin{equation}\label{local CM cycle}
 \mathrm{M}^{\rm loc}(L_0)_{\Z}\to\mathrm{M}^{\rm loc}(L).
 \end{equation}
If $p=2$ the assumption that $C^+(L_0)$ is a maximal order of odd
discriminant implies that $L_{0,\F_2}$ is self-dual so that
(\ref{local CM cycle}) must map into the smooth locus of the
target.

For $p>2$ the key point  is that the radical of $L_{0,\F_p}$, if
non-zero,  defines an $\F_p$-valued point of the quadric
$\mathrm{M}^{\rm loc}(L_0)$ associated with $L_0$. Therefore
(\ref{local CM cycle}) must map  into the regular locus (see the
proof of Proposition~\ref{prop:functorial}).

Finally, the extension~\eqref{0_abscheme_isom} of the
isomorphism~\eqref{eqn:0_rat tensor} once again follows from the
property of N\'eron models, and the remaining assertions follow
from Proposition~\ref{prop:functorial} and from what we have
already shown in the self-dual case.
\end{proof}

\begin{remark}
The proof of Proposition \ref{prop:0_functorial} actually proves the stronger claim that the morphism
$
\mathcal{Y} \to \mathcal{M}
$
is finite and unramified.
\end{remark}

Now suppose that $S$ is a scheme over $\mathcal{Y}$. Then
Proposition~\ref{prop:0_functorial} gives us an isometric
embedding $\Lambda\hookrightarrow V( \mathcal{A}_S )$. We obtain
the following analogue of Proposition~\ref{prop:decomistionVmu}.

\begin{proposition}\label{prop:0_decomistionVmu}\mbox{}
\begin{enumerate}
\item There is a canonical isometry
\[
 V(\mathcal{E}_{S})\iso \Lambda^{\perp}\subset V( \mathcal{A}_S ).
\]
\item
The induced map $V(\mathcal{E}_S) \oplus \Lambda \hookrightarrow V(\mathcal{A}_S)$, tensored with $\Q$, restricts to an injection
\[
V_{\mu_1}(\mathcal{E}_{S}) \times ({\mu_2}+\Lambda) \hookrightarrow V_{\mu}(\mathcal{A}_S)
\]
for every $\mu\in L^\vee/L$ and every
\[
(\mu_1,\mu_2)\in \bigl(\mu+L\bigr)/\bigl(L_0\oplus \Lambda\bigr) \subset (L_0^\vee /L_0) \oplus (\Lambda^\vee /\Lambda).
\]
\item
The above injections determine  a decomposition
\[
V_{\mu}(\mathcal{A}_S)
=\bigsqcup_{(\mu_1,\mu_2)\in (\mu+ L)/(L_0\oplus \Lambda)}  V_{\mu_1}(\mathcal{E}_{S}) \times
\bigl({\mu_2}+\Lambda\bigr).
\]
\end{enumerate}
\end{proposition}

\begin{proof}
Exactly as in the proof  Proposition \ref{prop:decomistionVmu},  Proposition~\ref{prop:0_functorial} implies that there is a canonical  isometry
\[
V(\mathcal{A}_{0,S})\iso \Lambda^{\perp}\subset V( \mathcal{A}_S ).
\]
Thus assertion (i) is clear from \eqref{eqn:A0 E special}.

In particular, there is a canonical inclusion
$
V(\mathcal{E}_S )  \times  \Lambda  \subset  V( \mathcal{A}_S ),
$
which we can make explicit using the $C(L)$-equivariant isomorphism
$
\mathcal{E}_S \otimes_{\co_\kk} C(L) \iso \mathcal{A}_S
$
of (\ref{0_abscheme_isom}).  Indeed, each  $\lambda \in \Lambda$ induces a special  endomorphism
\[
 \mathcal{A}_S \iso \mathcal{E}_S \otimes_{\co_\kk} C(L) \map{ e \otimes c \mapsto e\otimes \lambda c} \mathcal{E}_S \otimes_{\co_\kk} C(L)\iso  \mathcal{A}_S,
\]
while each 
$
 \phi \otimes \ell \in  \End_{\overline{\co}_\kk }(\mathcal{E}_S) \otimes_{\co_\kk} L_0 =V(\mathcal{E}_S ) 
$ 
induces a special  endomorphism
\[
 \mathcal{A}_S \iso \mathcal{E}_S \otimes_{\co_\kk} C(L) \map{ e \otimes c \mapsto \phi(e)\otimes  \ell c} \mathcal{E}_S \otimes_{\co_\kk} C(L)\iso  \mathcal{A}_S.
\]

With this in mind, assertions (ii) and (iii) amount to unwinding the definitions. By embedding $L$ into a unimodular lattice and using Proposition~\ref{prop:decomistionVmu}, we may reduce to the case where $L$ is itself unimodular. In this case, we have canonical isometries
\[
L_0^\vee/L_0 \xleftarrow{\simeq}L/(L_0\oplus\Lambda)\xrightarrow{\simeq}\Lambda^\vee/\Lambda,
\]
and $\mu_1\in L_0^\vee/L_0$ corresponds to  $\mu_2\in \Lambda^\vee/\Lambda$  if and only if  $(\mu_1,\mu_2) \in L/(L_0 \oplus \Lambda)$.

If $(\mu_1,\mu_2) \in L/(L_0 \oplus \Lambda)$ and 
\[
(x,\lambda) \in V_{\mu_1}(\mathcal{E}_{S}) \times ({\mu_2}+\Lambda) 
\]
then we may write $x=\phi \otimes \ell  + \mathrm{id} \otimes \mu_1$ for some $\phi\in \End_{\overline{\co}_\kk }(\mathcal{E}_S)$ and $\ell \in L_0$.  As $\mu_1+\mu_2\in L$, the quasi-endomorphism  $(x,\lambda) \in V(\mathcal{A}_S)_\Q$ is  equal to 
\begin{equation}\label{integral special sum}
 \mathcal{A}_S \iso \mathcal{E}_S \otimes_{\co_\kk} C(L) \map{ e \otimes c \mapsto \phi(e)\otimes  \ell c  + e\otimes(\mu_1+\mu_2)c} \mathcal{E}_S \otimes_{\co_\kk} C(L)\iso  \mathcal{A}_S,
\end{equation}
and so lies in $V(\mathcal{A}_S)$.

Conversely, any element of $V(\mathcal{A}_S)$ is given by a pair \[(x,\lambda) \in V(\mathcal{E}_S)_\Q \times \Lambda_\Q,\] and it is easy to see that we must have $x\in V(\mathcal{E}_S)^\vee$ and $\lambda \in \Lambda^\vee$.
Let $\mu_2\in \Lambda^\vee /\Lambda$ be the coset of $\lambda$, and let $\mu_1\in L_0^\vee /L_0$ be the unique element for which $(\mu_1,\mu_2)\in L/(L_0 \oplus \Lambda)$.  If we write $x=\phi \otimes \ell$ with $\ell\in V_0$ then, by assumption, the quasi-endomorphism of $\mathcal{A}_S$ determined by $(x,\lambda)$ is the (integral) endomorphism (\ref{integral special sum}), while  $x-\mathrm{id}\otimes \mu_1$ is the quasi-endomorphism
\[
 \mathcal{A}_S \iso \mathcal{E}_S \otimes_{\co_\kk} C(L) \map{ e \otimes c \mapsto \phi(e)\otimes  \ell c  - e\otimes \mu_1c} \mathcal{E}_S \otimes_{\co_\kk} C(L)\iso  \mathcal{A}_S.
\]
These differ by the endomorphism $e\otimes c\mapsto e\otimes(\mu_1+\mu_2)c$, and so $x-\mathrm{id}\otimes \mu_1$ is an (integral) endomorphism of $\mathcal{E}_S \otimes_{\co_\kk} C(L)$.  As $L_0 \subset C(L)$ as an $\co_\kk$-module direct summand, it follows that $x-\mathrm{id}\otimes \mu_1$  restricts to a homomorphism $\mathcal{E}_S \to \mathcal{E}_S \otimes_{\co_\kk}L_0$.   This proves that  $x\in V_{\mu_1}(\mathcal{E}_S)$, and all parts of the claim follow. 
\end{proof}

We immediately obtain from Proposition~\ref{prop:0_decomistionVmu} the following analogue of Proposition~\ref{prop:decomistionZmu}.

\begin{proposition}\label{prop:0_decomistionZmu}
There is an isomorphism of $\mathcal{Y}$-stacks
\begin{equation}\label{eq:0_decompdivisors}
\mathcal{Z}(m,\mu)\times_{\mathcal{M}}\mathcal{Y}\\
\iso \bigsqcup_{ \substack{ m_1+m_2=m \\ (\mu_1,\mu_2)\in (\mu+L)/(L_0\oplus \Lambda)}}
 \mathcal{Z}_0(m_1,\mu_1)_{\co_\kk}  \times \Lambda_{m_2,\mu_2},
\end{equation}
where
\[
\Lambda_{m_2,\mu_2} = \{x\in {\mu_2}+\Lambda : Q(x)=m_2\},
\]
and $ \mathcal{Z}_0(m_1,\mu_1)_{\co_\kk}  \times \Lambda_{m_2,\mu_2}$ denotes the disjoint union of $\# \Lambda_{m_2,\mu_2}$ copies of
$\mathcal{Z}_0(m_1,\mu_1)_{\co_\kk}$.
\end{proposition}

\begin{remark}\label{rmk:properandimproper}
The terms  $\mathcal{Z}_0(m_1,\mu_1)_{\co_\kk}$ in the
decomposition (\ref{eq:0_decompdivisors}) define $0$-cycles  on
$\mathcal{Y}$ when $m_1\neq 0$. The remaining terms
\[
\bigsqcup_{{(0,\mu_2)\in (\mu+ L)/(L_0\oplus \Lambda)}}
 \mathcal{Z}_0(0,0)_{\co_\kk} \times \Lambda_{m,\mu_2}
=  \bigsqcup_{{(0,\mu_2)\in (\mu+ L)/(L_0\oplus \Lambda)}}
\mathcal{Y} \times \Lambda_{m,\mu_2}
\]
are those with $(m_1,\mu_1)=(0,0)$, and account for the improper
intersection between  the cycles $\mathcal{Z}(m,\mu)_{\co_\kk}$
and $\mathcal{Y}$.
\end{remark}


\section{Degrees of metrized line bundles}
\label{s:main section}


As in previous sections, fix a maximal quadratic space $L$ over
$\Z$ of signature $(n,2)$, and a $\Z$-module direct summand
$L_0\subset L$ of signature $(0,2)$.  Assume  the even
Clifford algebra $C^+(L_0)$ is isomorphic to the maximal order in
a quadratic imaginary field $\kk\subset \C$ of odd
discriminant. Set
\[
\Lambda = \{x\in L : x\perp L_0\}.
\]

Recall from \S \ref{sec:integral} that we have associated with $L$ an algebraic stack
\begin{equation}\label{base change model}
\mathcal{M} \to \Spec(\co_\kk).
\end{equation}
Similarly, in \S \ref{s:CM Shimura variety integral} we constructed from $L_0$ an algebraic stack
$
\mathcal{Y} \to \Spec(\co_\kk).
$
 The functoriality results of Proposition \ref{prop:0_functorial} provide us with a finite and unramified morphism
\[
i : \mathcal{Y} \to \mathcal{M}.
\]
Of course the  stack $\mathcal{M}$ was originally constructed over $\Z$.  Its restriction to $\co_\kk$ is all that we will use in this section.


\subsection{Metrized line bundles}
\label{ss:metrizedbundles}


We will need some basic notions of arithmetic intersection theory on the stack $\mathcal{M}$.   For more
details see  \cite{GS}, \cite{KRY2}, \cite{KRY3}, or \cite{Vistoli}

Although the integral model over $\Z$ described in \S \ref{sec:integral} is regular, its base change (\ref{base change model}) to $\co_\kk$ need not be. This is not
a serious problem: the first part of the following proposition implies that $\mathcal{M}$  has a perfectly good theory of Cartier divisors. Whenever we speak of
``divisors" on $\mathcal{M}$ we always mean ``Cartier divisors."

\begin{proposition}
The stack $\mathcal{M}$ is locally integral. If $D$ denotes  the greatest common divisor  of $\mathrm{disc}(L_0)$ and  $\mathrm{disc}(L)$, then the restriction of
$\mathcal{M}$ to  $\co_\kk[1/D]$ is regular.
\end{proposition}

\begin{proof}
Using Proposition~\ref{prop:Mcheck} and the definition of $\mathcal{M}$, to show that $\mathcal{M}$ is locally integral, it suffices to show that the regular locus of the quadric $\mathrm{M}^{\rm loc}(L)$ associated with the quadratic space $L$ is integral after base change to $\co_\kk$. In explicit terms, the complete local rings of this locus at its $\overline{\F}_p$-points are either formally smooth over $\Z_p$, or are isomorphic to one of the two following $W(\overline{\F}_p)$-algebras~\cite[(2.16)]{mp:reg}:
\[
 W(\overline{\F}_p)\frac{[[u_1,\ldots,u_{n+1}]]}{(\sum_iu_i^2+p)}, \;\;\text{or }\; \;W(\overline{\F}_p)\frac{[[u_1,\ldots,u_n,w]]}{(\sum_iu_i^2+p(w+1))}.
\]
Since $n\geq 1$, it is easily checked that, in all three cases, tensoring with $\co_\kk$ over $\Z$ gives us something locally integral.

The claim about regularity over $\co_\kk[1/D]$ follows from the smoothness of the integral model of \S \ref{sec:integral} over $\Z[1/\mathrm{disc}(L)]$ and of $\co_\kk[1/\mathrm{disc}(L_0)]$ over~$\Z$.
\end{proof}

An \emph{arithmetic divisor} on $\mathcal{M}$ is a pair $\widehat{\mathcal{Z}} = (\mathcal{Z} , \Phi)$ consisting of a Cartier divisor $\mathcal{Z}$ on $\mathcal{M}$
and a Green function $\Phi$ for $\mathcal{Z}(\C)$ on $\mathcal{M}(\C)$.  Thus if $\Psi=0$ is a local equation for $\mathcal{Z}(\C)$,
the function $\Phi+ \log |\Psi|^2$, initially defined on the complement of  $\mathcal{Z}(\C)$, is required to extend smoothly across
the singularity $\mathcal{Z}(\C)$.  A \emph{principal arithmetic divisor} is an arithmetic divisor of the form
\[
\widehat{\mathrm{div}}(\Psi) = (\mathrm{div}(\Psi) , - \log|\Psi|^2 )
\]
for a rational function $\Psi$ on $\mathcal{M}$.  The group of all arithmetic divisors is denoted $\widehat{\mathrm{Div}}(\mathcal{M})$,
and its quotient by the subgroup of principal arithmetic divisors is the \emph{arithmetic Chow group} $\widehat{\mathrm{CH}}^1(\mathcal{M})$ of Gillet-Soul\'e.

A \emph{metrized line bundle} on $\mathcal{M}$ is  a line bundle
endowed with  a smoothly varying  Hermitian metric on its complex points.
The isomorphism classes of metrized line bundles form a group $\widehat{\mathrm{Pic}}( \mathcal{M} )$ under tensor product.

There is an isomorphism
\[
\widehat{\mathrm{CH}}^1(\mathcal{M}) \iso \widehat{\mathrm{Pic}}( \mathcal{M} )
\]
defined as follows.  Given an arithmetic divisor $(\mathcal{Z},\Phi)$, the line bundle $\mathcal{L}=\co(\mathcal{Z})$
is endowed with a canonical rational section $s$ with divisor $\mathcal{Z}$.  This section is nothing more than the constant
function $1$ on $\mathcal{M}$, viewed as a section of $\mathcal{L}$.  We endow $\mathcal{L}$ with the metric defined by
$-\log || s||^2 = \Phi.$   For the inverse construction, start with a metrized line bundle $\widehat{\mathcal{L}}$ on $\mathcal{M}$
and let $s$ be any nonzero rational section.  The associated arithmetic divisor, well-defined modulo principal arithmetic divisors,  is
\[
\widehat{\mathrm{div}}(s) = (\mathrm{div}(s) , - \log||s||^2).
\]

Of course there is a similar discussion with $\mathcal{M}$ replaced by $\mathcal{Y}$.  As $\mathcal{Y}$ is smooth of relative dimension $0$ over $\co_\kk$, all
divisors on $\mathcal{Y}$ are supported in nonzero characteristics;  thus, a Green function for a divisor on $\mathcal{Y}$ is \emph{any} complex-valued function on
the $0$-dimensional orbifold $\mathcal{Y}(\C)$.  In particular any arithmetic divisor $(\mathcal{Z},\Phi)$ decomposes as  $(\mathcal{Z},  0 ) + ( 0,\Phi)$, and
$\mathcal{Z}$ can be further decomposed as the difference of two effective Cartier divisors.

To define the  \emph{arithmetic degree} (as in  \cite{KRY2} or \cite{KRY3})
of  an arithmetic divisor $\widehat{\mathcal{Z}}$ on $\mathcal{Y}$ we first assume that $\widehat{\mathcal{Z}} = (\mathcal{Z},0)$
with $\mathcal{Z}$ effective.  Then
\[
\widehat{\deg}(\widehat{\mathcal{Z}}) = \sum_{\mathfrak{p} \subset \co_\kk} \log \mathrm{N}(\mathfrak{p}) \sum_{ z\in \mathcal{Z}(  \overline{\F}_\mathfrak{p} ) }
\frac{  \mathrm{length}( \co_{\mathcal{Z},z} )    }{ \# \Aut(z) }
\]
where  $\co_{\mathcal{Z},z}$ is the \'etale local ring of $\mathcal{Z}$ at $z$.
If $\widehat{\mathcal{Z}} = ( 0 , \Phi)$ is purely archimedean, then
\[
\widehat{\deg}(\widehat{\mathcal{Z}}) = \sum_{ y \in \mathcal{Y}(\C)}  \frac{  \Phi(y)  }{ \# \Aut(y) }.
\]
The arithmetic degree extends linearly to all arithmetic divisors, and defines a homomorphism
\[
\widehat{\deg}  :\widehat{\mathrm{Pic}}( \mathcal{Y} )  \to \R.
\]

We now define a homomorphism
\[
[ \cdot: \mathcal{Y} ] :  \widehat{\mathrm{Pic}}( \mathcal{M} ) \to \R,
\]
the \emph{arithmetic degree along $\mathcal{Y}$},   as the composition
\[
\widehat{\mathrm{Pic}}( \mathcal{M} ) \map{i^*} \widehat{\mathrm{Pic}}( \mathcal{Y} )  \map{\widehat{\deg} } \R .
\]


\subsection{Specialization to the normal bundle}


Fix a positive $m\in \Q$ and a $\mu\in L^\vee /L$, and denote by
\begin{equation}\label{our guy}
\mathcal{Z}(m,\mu) \to \mathcal{M}
\end{equation}
the stack obtained from (\ref{special divisor}) by base change from $\Z$ to $\co_\kk$.
 The fact that (\ref{our guy}) and  $i:\mathcal{Y}\to \mathcal{M}$
are not  closed immersions is a minor nuisance.  The following definition is introduced to address  technical
difficulties arising from this defect.

\begin{definition}
A \emph{sufficiently small \'etale open chart} of $\mathcal{M}$ is a scheme $U$
together with an \'etale morphism $U \to \mathcal{M}$ such that
\begin{enumerate}
\item
 the natural map $\mathcal{Z}(m,\mu)_U \to U$ restricts to  a closed immersion on every
 connected component $Z\subset \mathcal{Z}(m,\mu)_U$,
\item
the natural map $\mathcal{Y}_U \to U$ restricts to  a closed immersion
on every connected component $Y \subset \mathcal{Y}_U$ .
\end{enumerate}
\end{definition}

\begin{remark}
The stack  $\mathcal{M}$ admits a covering by sufficiently small \'etale open charts.
This is a consequence of \cite[Lemma 1.19]{Vistoli} and the fact that  $\mathcal{Z}(m,\mu) \to \mathcal{M}$
and $\mathcal{Y} \to \mathcal{M}$ are finite and unramified.  Note that these two morphisms are relatively representable, and
so for any $U \to \mathcal{M}$ as above  the pull-backs $\mathcal{Z}(m,\mu)_U$ and $\mathcal{Y}_U$   are actually schemes.
\end{remark}

\begin{remark}
When we think of $\mathcal{Z}(m,\mu)$ as a Cartier divisor  on $\mathcal{M}$, its pull-back  to $U$ is
simply  the sum of the connected components of $\mathcal{Z}(m,\mu)_U$,
each viewed as closed subscheme of $U$.
\end{remark}

\begin{remark}
In \S \ref{s:CM section} the notation $Y$ was used for the generic fiber of $\mathcal{Y}$.
Throughout \S \ref{s:main section} the notation $Y$ is only used for a connected component of $\mathcal{Y}_U$, and so no confusion should arise.
\end{remark}

Fix a sufficiently small $U\to \mathcal{M}$ and a connected component $Y\subset \mathcal{Y}_U$.
The smoothness of $\mathcal{Y}$ over $\co_\kk$ implies that $Y$ is a regular integral scheme of dimension one.
Let $I\subset \co_U$ be the ideal sheaf defining the closed immersion $Y\to U$, and recall that the
\emph{normal bundle} of $Y \subset U$ is the $\co_Y$-module
\[
N_Y U = \underline{\Hom}(I/I^2 , \co_Y).
\]
As $Y\to U$ may not be a regular immersion, the $\co_Y$-module $I/I^2$ may not be locally free.  However,
we will see below in Proposition \ref{prop:normal calc} that $N_Y U$ is locally free of rank $n$, and so defines a vector bundle $N_Y U \to Y$.
By taking the disjoint union over all connected components $Y\subset \mathcal{Y}_U$ we obtain a vector
bundle on $\mathcal{Y}_U$.  Letting $U$ vary over a cover of $\mathcal{M}$ by sufficiently small
\'etale open charts and glueing defines the \emph{normal bundle}
\begin{equation}\label{pi bundle}
\pi: N_\mathcal{Y} \mathcal{M} \to \mathcal{Y}.
\end{equation}

Recall the Hejhal-Poincar\'e series
\[
F_{m,\mu}(\tau) \in H_{1-\frac{n}{2}}(\omega_L)
\]
of  \S\ref{ss:green functions}.   The discussion of  \S \ref{ss:green functions}   endows the divisor
$
\mathcal{Z}(m,\mu) =   \mathcal{Z}(F_{m,\mu})
$
with a Green function
\[
\Phi_{m,\mu}=\Phi( F_{m,\mu} ,  \cdot ),
\]
which is defined at every point of $\mathcal{M}(\C)$, but is discontinuous at points of the divisor $\mathcal{Z}(m,\mu)$.
The corresponding arithmetic divisor is denoted
\begin{equation}\label{arithmetic divisor}
\widehat{\mathcal{Z}}(m,\mu) = (\mathcal{Z}(m,\mu), \Phi_{m,\mu}) \in \widehat{\mathrm{Div}}(\mathcal{M}).
\end{equation}

From the arithmetic divisor (\ref{arithmetic divisor}) we will construct a new arithmetic divisor
\begin{equation}\label{arithmetic specialization}
 \sigma(  \widehat{\mathcal{Z}}(m,\mu) ) \in \widehat{\mathrm{Div}}( N_\mathcal{Y} \mathcal{M} )
\end{equation}
called the \emph{specialization to the normal bundle}.  Fix a cover of $\mathcal{M}$
by sufficiently small \'etale open charts.  Given a   chart $U \to \mathcal{M}$ in the cover,
fix a connected component $Y\subset \mathcal{Y}_U$ and decompose $\mathcal{Z}(m,\mu)_U = \bigsqcup Z$
as the union of its connected components.  By refining our cover, we may assume that each
closed subscheme $Z \to U$ is defined by  a single equation $f_Z=0$.

\begin{lemma}\label{lem:first order vanishing}
For every connected component $Z\subset \mathcal{Z}(m,\mu)_U$ the intersection
\[
Y\cap Z := Y\times_U Z
\]
satisfies one of the following (mutually exclusive) properties:
\begin{enumerate}
\item
$Y\cap Z$ has dimension $0$, and $f_Z$ restricts to a nonzero section of  $\co_U /I$,
\item
$Y\cap Z=Y$, and  $f_Z$ defines a section of $I \subset \co_U$  with nonzero image under the natural map
$
I/I^2 \to \underline{\Hom} (N_U Y , \co_Y).
$
\end{enumerate}
\end{lemma}

\begin{proof}
As $Y$ is an integral scheme of dimension one, its closed subscheme $Y\cap Z$ is either all of $Y$ or of dimension $0$.
Given this, the only nontrivial thing to check is that when $Y\cap Z=Y$, the image of $f_Z$ under
$I/I^2 \to \underline{\Hom} (N_U Y , \co_Y)$ is nonzero.  This can be checked after restricting to the complex fiber,
where it follows from the smoothness of the divisor  $Z(\C)$ on $U(\C)$.  This smoothness can be checked
using the complex uniformization (\ref{eqn:Zmmu}).
\end{proof}

 If $Y\cap Z$ has dimension $0$,  then restricting $f_Z$ to $Y$ and pulling back
via  $N_YU \to Y$ results in a  function $\sigma(f_Z)$ on $N_YU$, which is  homogeneous of degree zero.
On the other hand, if $Y\cap Z=Y$ then the image of $f_Z$ in $\underline{\Hom} (N_U Y , \co_Y)$ defines a
function $\sigma(f_Z)$   on $N_UY$, which is homogeneous of degree one.
In either case, define  an effective Cartier divisor on $N_Y U$ by
\begin{equation}\label{divisor specialization}
\sigma( \mathcal{Z}(m,\mu) ) = \sum_{Z \subset \mathcal{Z}(m,\mu)_U}  \mathrm{div}\big( \sigma(f_Z) \big).
\end{equation}

The function
\[
 \phi_{m,\mu} =  \Phi_{ m ,\mu} + \sum_{Z \subset \mathcal{Z}(m,\mu)_U}  \log|f_Z|^2,
\]
initially defined on the complex fiber of $U \smallsetminus \mathcal{Z}(m,\mu)_U$, extends smoothly to all of $U(\C)$.
By restricting to $\mathcal{Y}(\C)$ and then pulling back via (\ref{pi bundle}) we obtain a smooth  function
$\pi^* i^* \phi_{m,\mu}  $ on  $(N_YU)(\C)$.  Define $\sigma(\Phi_{ m ,\mu})$ by the relation
\begin{equation}\label{green specialization}
 \pi^* i^* \phi_{m,\mu}    = \sigma( \Phi_{ m ,\mu} ) + \sum_{Z \subset \mathcal{Z}(m,\mu)_U}  \log|  \sigma( f_Z ) |^2.
\end{equation}

The resulting arithmetic divisor
\[
 \big( \sigma(  \mathcal{Z}(m,\mu) )   ,\sigma( \Phi_{m ,\mu} )    \big)
\in \widehat{\mathrm{Div}} (N_YU)
\]
 does not depend on the  choice of $f_Z$'s used in its construction.
Using  \'etale descent these arithmetic divisors glue together to form the desired
arithmetic divisor (\ref{arithmetic specialization}).

\begin{proposition}\label{prop:poor mans Hu}
  The composition
\[
 \widehat{\mathrm{Pic}}(\mathcal{M} ) \map{i^*} \widehat{\mathrm{Pic}}(\mathcal{Y} )
 \map{\pi^*}  \widehat{\mathrm{Pic}}( N_\mathcal{Y} \mathcal{M} )
\]
sends the  metrized line bundle  defined by $\widehat{\mathcal{Z}}(m,\mu)$ to the metrized
line bundle defined by  $\sigma( \widehat{\mathcal{Z}}(m,\mu))$
\end{proposition}

\begin{proof}
Let $U\to \mathcal{M}$ be a sufficiently small \'etale open chart, and write
\[\mathcal{Z}(m,\mu)_U = \bigsqcup Z\] as the union of its connected components.
Each connected component determines a line bundle $\co(Z)=f_Z^{-1} \co_U$, and
\[
\mathcal{Z}(m,\mu)_U\iso \bigotimes \co(Z)
\]
as line bundles on $U$.

Let $\pi^*i^*\co(Z)$ be the pullback of $\co(Z)$ via
$
N_YU \map{\pi} Y \map{i} U.
$
If we denote by  $s_Z$ the constant function  $1$ on $U$ viewed as a section of $\co(Z)$,
then
\[
\sigma(s_Z) =  \sigma(f_Z)  \pi^* i^*(f_Z^{-1} s_Z)
\]
is a nonzero section of  $\pi^*i^*\co(Z)$, and does not depend on the choice of $f_Z$.
The tensor product $\sigma(s)=\otimes \sigma(s_Z)$ over all $Z$ is a global section of the metrized  line bundle
$\pi^* i^* \widehat{\mathcal{Z}}(m,\mu)_U$,  and these section glue together over an \'etale cover of
$\mathcal{M}$ to  define a global section $\sigma(s)$ of $\pi^* i^* \widehat{\mathcal{Z}}(m,\mu)$.

Tracing through the definitions  shows that
$
\widehat{\mathrm{div}}(\sigma(s)) =  \sigma( \widehat{\mathcal{Z}}(m,\mu) ),
$
as arithmetic divisors, and so
\[
\pi^* i^* \widehat{\mathcal{Z}}(m,\mu) \iso  \sigma( \widehat{\mathcal{Z}}(m,\mu) )
\]
as metrized line bundles.
\end{proof}

\begin{remark}
The construction (\ref{arithmetic specialization}) and Proposition \ref{prop:poor mans Hu} will be
our main tools for computing improper intersections.  The technique  is  based on the thesis  of  J.~Hu \cite{Hu}, which
  reconstructs the arithmetic intersection theory of Gillet-Soul\'e \cite{GS} using Fulton's method of
deformation to the normal cone.

Letting $C_\mathcal{Y}\mathcal{M}$ denote the normal cone of
$\mathcal{Y}$ in $\mathcal{M}$, Hu\footnote{Hu works in greater generality, and allows for the cycles $\mathcal{Z}$ and $\mathcal{Y}$ to have arbitrary codimension.
} constructs a \emph{specialization to the normal cone} map
\[
\sigma: \widehat{\mathrm{Div}}(\mathcal{M}) \to  \widehat{\mathrm{Div}}(C_\mathcal{Y}\mathcal{M}),
\]
and shows that
\[
[\widehat{\mathcal{Z}} : \mathcal{Y} ] = [ \sigma(\widehat{\mathcal{Z}}) : \mathcal{Y} ]
\]
for any arithmetic divisor $\widehat{\mathcal{Z}}$.  Here the intersection pairing on the right is
defined using the canonical closed immersion $\mathcal{Y} \to C_\mathcal{Y}\mathcal{M}$.

The normal bundle $N_\mathcal{Y}\mathcal{M}$ is essentially a first order approximation to
$C_\mathcal{Y}\mathcal{M}$, and our construction (\ref{arithmetic specialization}) and Proposition \ref{prop:poor mans Hu}
amount to truncating Hu's theory to first order.  In order to use this to compute arithmetic intersections,
it is essential to know that our divisors satisfy Lemma \ref{lem:first order vanishing}, which says, loosely speaking,
that the functions $f_Z$ vanish along $Y$ to at most order $1$.  This guarantees that the specialization (\ref{arithmetic specialization})
to the normal bundle does not lose too much information about the divisor $\mathcal{Z}(m,\mu)$.

To be clear: we are exploiting here a special property of the divisors $\mathcal{Z}(m,\mu)$ and their relative
positions with respect to $\mathcal{Y}$.  It would not be very useful
to define the specialization  (\ref{arithmetic specialization}) for an arbitrary divisor on $\mathcal{M}$,
as truncating Hu's specialization map to first order would lose essential higher order information.
\end{remark}


\subsection{The cotautological bundle}
\label{ss:cotaut}


Recall from (\ref{eq:bmV}) the vector bundle $\bm{V}_{\dR}$ on $\mathcal{M}$ with its isotropic line $\bm{V}^1_{\dR}$,
and from  (\ref{eq:bmVn=0}) the line bundle  $\bm{V}_{0,\dR}$ on $\mathcal{Y}$.   Proposition \ref{prop:0_functorial} implies that  $ \bm{V}_{0, {\dR}}   \subset    i^*\bm{V}_{\dR} $
in such a way that     $ \bm{V}_{0, {\dR}}^1   =  i^*\bm{V}^1_{\dR}$.

Denote by $\bm{T}=( \bm{V}^1_{\dR})^\vee$ and  $\bm{T}_0=( \bm{V}^1_{0,\dR})^\vee$ the
cotautological bundles on $\mathcal{M}$ and $\mathcal{Y}$, respectively, so that $\bm{T}_0=i^*\bm{T}$.
Under the complex uniformization
 $
 \Gamma_g \backslash \mathcal{D} \to \mathcal{M}(\C)
 $
 of (\ref{uniformization})   the fiber of $\bm{V}^1_\dR$  at  $z\in \mathcal{D}$ is identified with the tautological line
$\C z \subset V_\C$, which we endow  with the metric
\begin{equation}\label{taut metric}
|| z ||^2 =\frac{-1}{4\pi e^\gamma} [z , \overline{z}].
\end{equation}
Here, as before,   $\gamma = -\Gamma'(1)$ is the Euler-Mascheroni constant.  Dualizing, we obtain a  metric on the
cotautological bundle, and the resulting metrized line bundle is denoted
\[
\widehat{\bm{T}}   \in \widehat{\mathrm{Pic}}(\mathcal{M}) .
\]
We endow $\bm{T}_0$ with the unique metric for which
\[
\widehat{\bm{T}}_0 \iso i^*\widehat{\bm{T}}.
\]

\begin{proposition} The metrized cotautological bundle satisfies
\[
\frac{w_\kk}{h_\kk} \cdot [ \widehat{\bm{T}} : \mathcal{Y} ]
= 2 \frac{ L'(\chi_\kk,0) }{ L(\chi_\kk , 0)  } +
 \log \left|  \frac{   d_\kk  }{ 4 \pi } \right|  - \gamma .
\]
\end{proposition}

\begin{proof}
This is virtually identical to the proof of \cite[Theorem 6.4]{BHY}, and so we only give an outline.
Let $\mathcal{E}\to \mathcal{Y}$ be the universal $\co_\kk$-elliptic curve.  If we endow the $\co_{\mathcal{Y}}$-module
$\Lie(\mathcal{E})$ with the Faltings metric, the Chowla-Selberg formula implies
\[
\frac{w_\kk}{h_\kk} \cdot \widehat{\deg} \ \widehat{\Lie}(\mathcal{E})
=    \frac{ L'(\chi_\kk,0) }{ L(\chi_\kk , 0)  }  +   \log(2\pi)+  \frac{ 1}{2 } \log |d_\kk|  .
\]

Now view $L_0$ as an $\co_\kk$-module through left multiplication,
and endow the corresponding line bundle $\bm{L}_0$ on $\Spec(\co_\kk)$ with the metric 
\[
|| x ||^2 = -16\pi^3 e^\gamma \cdot Q(x).
\]
 Pulling back $\bm{L}_0$ to $\mathcal{Y}$ via the
structure morphism   yields a metrized line bundle $\widehat{\bm{L}}_0$ on  $\mathcal{Y}$ satisfying
\[
\frac{w_\kk}{h_\kk} \cdot \widehat{\deg}\ \widehat{\bm{L}}_0 = - \log(16 \pi^3e^\gamma) .
\]

By keeping track of the metrics, the isomorphism (\ref{cotautonCM}) defines an isomorphism of metrized line bundles
\[
\widehat{\bm{T}}_0 \iso \widehat{\Lie}( \mathcal{E} )^{\otimes 2}
\otimes  \widehat{\bm{L}}_0
\]
on $\mathcal{Y}$, and the claim follows.
\end{proof}

Comparing with  Proposition \ref{prop:eisenstein coeff} proves:

\begin{corollary}\label{cor:cotaut degree}
The coefficient $a^+_{L_0}(0) \in \mathfrak{S}_{L_0}^\vee$ satisfies
\[
a^+_{L_0}(0,\mu) = -  \frac{w_\kk} {h_\kk}  \cdot
\begin{cases}
 [ \widehat{\bm{T}} : \mathcal{Y} ]  & \text{if $\mu=0$} \\
 0 & \text{otherwise}
 \end{cases}
\]
for every $\mu\in L_0^\vee/L_0$.
\end{corollary}


\subsection{The normal bundle at  CM points}\label{ss:normal bundle}


Consider the orthogonal complement $(\bm{V}^1_{0,\dR})^{\perp}\subset i^*\bm{V}_{\dR}$.  Since $\mathcal{Y}$ is regular of dimension $1$, this complement is torsion-free and is therefore a vector sub-bundle of $i^*\bm{V}_{\dR}$ of co-rank $1$. (Note that the corresponding assertion over $\mathcal{M}$ may not be true; the orthogonal complement of $\bm{V}^1_{\dR}$ in $\bm{V}_{\dR}$ will in general not be a vector bundle unless $L$ is unimodular.)

Set $\Fil^0i^*\bm{V}_{\dR}=(\bm{V}^1_{0,\dR})^{\perp}$ and $\mathrm{gr}^0_{\Fil}i^*\bm{V}_{\dR}=\Fil^0i^*\bm{V}_{\dR}/\bm{V}^1_{0,\dR}$.

\begin{proposition}\label{prop:normal calc}
There is a canonical isomorphism of $\co_\mathcal{Y}$-modules
\[
N_\mathcal{Y}\mathcal{M}  \iso  \bm{T}_0 \otimes  \mathrm{gr}^0_{\Fil}i^*\bm{V}_{\dR}  .
\]
In particular, the normal bundle  is locally free of rank $n$, and is therefore relatively representable over $\mathcal{Y}$.
\end{proposition}

\begin{proof}
Fix a sufficiently small \'etale open chart $U\to \mathcal{M}$, and let $Y\subset \mathcal{Y}_U$
be a connected component.
Let $Y[\varepsilon]=Y\times_\Z \Spec (\Z[\varepsilon] )$ be the scheme of dual numbers relative to $Y$,
and let  ${\rm Def}_i$ be the Zariski sheaf of infinitesimal deformations of $i:Y\to U$.
By this we mean  the sheaf associating to an open subscheme $S\subset Y$ the set
${\rm Def}_i(S)$ of morphisms $j_S : S[\varepsilon] \to U$ such that $j_S \vert_S= i_S$.
Here we are writing $i_S$ for the inclusion $S\hookrightarrow Y\hookrightarrow U$.

 \begin{lemma}\label{lem:normal 1}
There is a canonical isomorphism of functors
\begin{equation}\label{eq:tau}
 {\rm Def}_i \iso N_{Y} U.
 \end{equation}
 \end{lemma}

 \begin{proof}
The morphism (\ref{eq:tau}) is defined in the usual way: every $j_S \in {\rm Def}_i(S)$    defines a map
 $ \co_{U}/I^2\to \co_{S}[\varepsilon]$, whose restriction
 \[
 I/I^2 \to \co_{S}\cdot \varepsilon\iso \co_S
 \]
 defines an element of $(N_Y U)(S)$.

To construct an inverse to (\ref{eq:tau}), first  recall that $\mathcal{Y}$ is formally \'etale over $\co_\kk$.  This implies that the surjection   $\co_U/I^2 \to
\co_Y$ admits a canonical section, and so
\[
\co_U / I^2 \iso \co_Y \oplus I/I^2
\]
as sheaves of rings on $\co_Y$.  Any $S$-point $s\in (N_YU) (S)$ therefore defines a homomorphism
\[
\co_U / I^2  \iso \co_Y \oplus I/I^2 \map{1\oplus s \cdot \epsilon} \co_S [\epsilon] ,
\]
which in turn determines a deformation $j_S:S[\epsilon] \to U$ of $i_S$.
\end{proof}

\begin{lemma}\label{lem:normal 2}
There is a canonical isomorphism 
\[
 {\rm Def}_i  \iso  ( \bm{V}^1_{0,\dR} )^\vee\otimes_{\co_{Y}}\mathrm{gr}^0_{\Fil}\left( i^* \bm{V}_{\dR}  \right)
\]
of sheaves of sets on $Y$.
\end{lemma}

\begin{proof}
Let $S\subset Y$ be an open subscheme, and let $j_S \in {\rm Def}_i(S)$.
The pull-back $ j_S^*\bm{V}_{\dR}$  is a locally free sheaf of $\co_{S[\varepsilon]}$-modules of rank $n+2$,
endowed with an $\co_{S[\varepsilon]}$-submodule  $j_S^* \bm{V}_{{\dR}}^{1}$ of rank one,
locally a  direct summand.    As $\bm{V}_{\dR}$ is a vector bundle with integrable connection,  the  retraction
  $S[\varepsilon]\to S$ induces  a canonical isomorphism
 \[
j_S^*  \bm{V}_{\dR}\iso  i_S^* \bm{V}_{\dR}  \otimes_{\co_{S}} \co_{S[\varepsilon]}.
 \]

 We claim that, under this isomorphism, $j_S^*\bm{V}^1_{\dR}$ maps into $\Fil^0 i_S^*\bm{V}_{\dR}\otimes_{\co_S}\co_{S[\varepsilon]}$;
 in other words, the image of $j_S^*\bm{V}^1_{\dR}$ is orthogonal to $i_S^*\bm{V}^1_{\dR}\otimes_{\co_S}\co_{S[\varepsilon]}$. This is easily deduced from the fact that both this image and $i_S^*\bm{V}^1_{\dR}\otimes_{\co_S}\co_{S[\varepsilon]}$ are isotropic lines lifting $i_S^*\bm{V}^1_{\dR}$.

Consider the composition
\[
 j_S^*\bm{V}^1_{\dR}\hookrightarrow \Fil^0i_S^* \bm{V}_{\dR}  \otimes_{\co_{S}} \co_{S[\varepsilon]}\to \mathrm{gr}^0_{\Fil}\bigl(i_S^* \bm{V}_{\dR} \bigr) \otimes_{\co_{S}}\co_{S[\varepsilon]}.
\]
The reduction of this map modulo $(\varepsilon)$ is $0$, and therefore  it must factor through a map
\[
 i_S^*\bm{V}^1_{\dR}\to\varepsilon\cdot\biggl(\mathrm{gr}^0_{\Fil}\bigl(i_S^* \bm{V}_{\dR} \bigr) \otimes_{\co_{S}}\co_{S[\varepsilon]}\biggr)\xrightarrow{\simeq}\mathrm{gr}^0_{\Fil}\bigl(i_S^* \bm{V}_{\dR} \bigr).
\]
Thus, we have produced a canonical morphism
\[
 {\rm Def}_i \rightarrow ( \bm{V}^1_{0,\dR} )^\vee\otimes_{\co_{Y}}\mathrm{gr}^0_{\Fil}\left( i^* \bm{V}_{\dR}  \right).
\]

To show that this is an isomorphism, we can, by~\eqref{prop:Mcheck} and the above discussion of local models,
assume that the immersion $Y\subset U$ is isomorphic to an \'etale neighborhood of the closed immersion
$\mathrm{M}^{\mathrm{loc}}(L_0)\hookrightarrow\mathrm{M}^{\mathrm{loc}}(L)$.
Moreover, we can choose this isomorphism in such a way that it identifies $\bm{V}_{0,\dR}$ and $\bm{V}_{\dR}$ with the \emph{trivial} vector bundles $\bm{1}\otimes L_0$ and $\bm{1}\otimes L$, respectively, and such that $\bm{V}^1_{0,\dR}$ and $\bm{V}^1_{\dR}$ get identified with the tautological isotropic lines in $\bm{1}\otimes L_0$ and $\bm{1}\otimes L$, respectively.

So, if $\bm{L}_0^1\subset\bm{1}\otimes L_0$ is the tautological isotropic line over $\mathrm{M}^{\mathrm{loc}}(L_0)$, and $\bm{L}^0\subset\bm{1}\otimes L$ is its orthogonal complement, we are reduced to showing the elementary fact that there is a canonical isomorphism
\[
 N_{\mathrm{M}^{\mathrm{loc}}(L_0)}(\mathrm{M}^{\mathrm{loc}}(L))\xrightarrow{\simeq}\underline{\Hom}(\bm{L}_0^1,\bm{L}^0/\bm{L}_0^1)
\]
of Zariski sheaves over $\mathrm{M}^{\mathrm{loc}}(L_0)$.
We leave this as an exercise to the reader.
\end{proof}

The proof of Proposition \ref{prop:normal calc} now follows by combining Lemmas (\ref{lem:normal 1}) and (\ref{lem:normal 2})
and glueing over an \'etale cover of $\mathcal{M}$.

\end{proof}

We will have use for the following composition:
\begin{equation}\label{eqn:AlmostNormal}
N_\mathcal{Y}\mathcal{M}  \xrightarrow[\eqref{prop:normal calc}]{\simeq}  \bm{T}_0 \otimes  \mathrm{gr}^0_{\Fil}i^*\bm{V}_{\dR} \to \bm{T}_0\otimes \bigl(i^*\bm{V}_{\dR}/\bm{V}_{0,\dR}\bigr).
\end{equation}
Here, the map on the right is induced by the natural map
\[
\mathrm{gr}^0_{\Fil}i^*\bm{V}_{\dR}\to i^*\bm{V}_{\dR}/\bm{V}^0_{\dR}.
\]

Note that this last map is an isomorphism over the generic fiber of $\mathcal{Y}$. Indeed, it amounts to knowing that the inclusion $\bm{V}_{0,\dR}+\Fil^0 i^*\bm{V}_{\dR}\subset i^*\bm{V}_{\dR}$ is an isomorphism over the generic fiber. This follows from the self-duality of the quadratic form on $\bm{V}_{\dR,M}$.

On the other hand, over the generic fiber, the map $\bm{\Lambda}_{\dR}\to i^*\bm{V}_{\dR}/\bm{V}^0_{\dR}$ is also an isomorphism. Therefore, for any point $y\in\mathcal{Y}(\C)$,~\eqref{eqn:AlmostNormal} induces a canonical isomorphism
\begin{equation}\label{eqn:CompNormalExplicit}
 (N_{\mathcal{Y}}\mathcal{M})_y\xrightarrow{\simeq}\bm{T}_{0,y}\otimes\Lambda_{\C}.
\end{equation}

We will now give an explicit description of this isomorphism. To begin, $y$ determines an isotropic line $\C y\subset L_{\C}$, whose dual can be naturally identified with the fiber $\bm{T}_{0,y}$ of $\bm{T}_0$ at $y$. The construction of the isomorphism~\eqref{eqn:CompNormalExplicit} proceeded by choosing an appropriate local trivialization of $\bm{V}_{\dR}$. In an analytic neighborhood of $y$, this simply means that we choose a local section $U\to\mathcal{D}$ of the complex analytic uniformization carrying $y$ to $\C y$.

Note that $\mathcal{D}$ is contained in the complex quadric $\mathrm{M}^{\mathrm{loc}}(L_{\C})$. We therefore see that it is enough to describe the induced isomorphism:
\begin{equation}\label{eqn:TangentSpace}
 T_{\C y}\mathcal{D}=T_{\C y}\mathrm{M}^{\mathrm{loc}}(L_{\C})\xrightarrow{\simeq}(\C y)^\vee\otimes\Lambda_{\C}.
\end{equation}
Here, $T_{\C y}\mathcal{D}$ is the tangent space of $\mathcal{D}$ at $\C y$.

We consider the local immersion
\begin{align*}
\Hom(\C y,\Lambda_{\C})=(\C y)^\vee\otimes\Lambda_{\C}\to\mathrm{M}^{\mathrm{loc}}(L_{\C}),
\end{align*}
which carries a section $\xi$ to the isotropic line spanned by
\[
y+\xi( y )-\frac{Q(\xi( y ))}{[ y ,\overline{ y }]}\overline{ y }.
\]
This carries $0$ to $\C y$, and the derivative at $0$ is exactly the inverse to~\eqref{eqn:TangentSpace}.


\subsection{Specialization of the Green function}


Suppose we have a complex point $y\in \mathcal{Y}(\C)$.  Under the uniformization
\[
\mathcal{Y}(\C) \iso T(\Q) \backslash \{ \bm{h}_0 \} \times T(\A_f) / \widehat{\co}_\kk^\times
\]
of \S \ref{s:CM Shimura variety}, the point $y$ is represented by a pair $(\bm{h}_0 , g)$ with
\[
g\in T(\A_f) \subset G(\A_f),
\]
and so its image in $\mathcal{M}(\C)$ lies on the component $\Gamma_g \backslash \mathcal{D}$ of
(\ref{uniformization}).  Here $G=\GSpin(L)$.  By mild abuse of notation
we also let  $\C y \in\mathcal{D}$ denote the isotropic line corresponding to the
oriented negative plane  $\bm{h}_0=L_{0\R}$, so that  $\C y$ is a lift of $y$ under
$
\mathcal{D} \to  \mathcal{M}(\C) ,
$
and $L_\C = \C y\oplus \C \overline{y} \oplus \Lambda_\C$.

From (\ref{eqn:Zmmu}), we find that a neighborhood of $y$ admits a complex analytic uniformization
\[
 \bigsqcup_{ \substack{ x\in \mu_g + L_g \\  Q(x)=m  }}\mathcal{D}(x)  \to \mathcal{Z}(m,\mu)(\C).
\]

The components $\mathcal{D}(x)$ passing through $\C y$ are precisely those for which $x$ is orthogonal to $y$. But, since $x$ is rational, this is equivalent to requiring that $x$ be orthogonal to both $y$ and $\overline{y}$; or, in other words, that $x$ lies in $\Lambda^\vee$.

As
the orthogonal transformation $g \in \SO(L \otimes \A_f)$ acts trivially on the direct summand $\Lambda \otimes \A_f$,
we see that  such $x$ in fact lie in $\mu+L$.
This shows that in a small enough analytic neighborhood  $U$ of $\C y$ there is an isomorphism of  divisors
\begin{equation}\label{small analytic divisor}
U \cap  \sum_{ \lambda \in \Lambda_{m,\mu} }\mathcal{D}(\lambda)   \iso  U \cap  \mathcal{Z}(m,\mu)(\C) ,
\end{equation}
  where
\[
\Lambda_{m,\mu} = \left\{ \lambda  \in \Lambda^\vee :
\begin{array}{c}
Q(\lambda) =m \\   \lambda \in \mu + L  \end{array}
 \right\}.
\]

Recall that each isotropic line $\C z \in \mathcal{D}$ corresponds to a negative plane $\bm{h}_z \subset V_\R$.
For each  $\lambda \in \Lambda_{m,\mu}$  let $\lambda_z \in V_\R$ be the orthogonal projection of $\lambda$ to
$\bm{h}_z$. Simple linear algebra gives
\[
Q(\lambda_{z}) =  \frac{[\lambda, \bar z]\cdot [\lambda, z]}{[z, \bar z]}.
\]
Define a function on $U\smallsetminus \mathcal{Z}(m,\mu)(\C)$ by
\[
\Phi^\mathrm{reg} _{ m,\mu } =
\Phi _{ m,\mu }  +  \sum_{\lambda \in \Lambda_{m,\mu}} \log \big| 4\pi e^\gamma \cdot Q(\lambda_z )\big|.
\]

Recall that the function $\Phi_{m,\mu}(z)$ is defined, but  discontinuous,  at $z=y$.  The following proposition
gives us a method of computing the value of this function at the discontinuity $z=y$.

\begin{proposition}
The function $\Phi^\mathrm{reg} _{ m,\mu } $ extends smoothly to all of $U$, and this extension satisfies
\[
 \Phi^{\mathrm{reg}} _{ m,\mu } (y) =  \Phi _{ m,\mu }  (y).
\]
\end{proposition}

\begin{proof}
Let $\tau=u+iv \in \mathcal{H}$ be the variable on the upper half-plane.
The function $\Phi_{m,\mu}( z )$ is defined, as in \cite[(4.8)]{BY}, as the constant term of the Laurent expansion  at $s=0$ of
\begin{equation}\label{reg theta}
\lim_{T\to \infty}  \int_{\mathcal{F}_T} \big\{ F_{m,\mu}(\tau), \theta_L(\tau, z )\big\} \, \frac{du\, dv}{v^{s+2}},
\end{equation}
where
\[
\mathcal{F}_T = \{ \tau=u+iv \in \mathcal{H}: \vert u \vert\leq 1/2, \ \vert \tau\vert  \geq 1 ,\   v\leq T \}.
\]

Now  substitute   the definition \cite[(2.4)]{BY} of $\theta_L(\tau,z)$ and the Fourier expansion of $F_{m,\mu}$ into (\ref{reg theta}), and
argue as in \cite[Theorem 6.2]{Borcherds}.   Using the decomposition $F_{m,\mu} = F^+_{m,\mu} + F^-_{m,\mu}$
of  \cite[(3.4a) and (3.4b)]{BY}  and the rapid decay as $v\to \infty$ of the Fourier coefficients of
$F^-_{m,\mu}$,  we find
\begin{equation}\label{magic integral}
\Phi_{ m ,\mu}(z) = \varphi_{m,\mu}(z)   + \sum_{  \lambda\in \Lambda_{m,\mu} }
 \mathrm{CT}_{s=0}
\left[
\int_1^\infty e^{4\pi v Q(\lambda_z) } \, \frac{dv}{v^{s+1}}
\right]
\end{equation}
for some smooth function $\varphi_{m,\mu}$ on $U$, where $\mathrm{CT}_{s=0}$ means take the  constant term at $s=0$.

Each $\lambda\in \Lambda_{m,\mu}$ is orthogonal to  $y$, and so $Q(\lambda_y)=0$.
Thus setting $z=y$ in (\ref{magic integral}) and computing the constant term at $s=0$ shows that
\[
\Phi_{m,\mu}(y) = \varphi_{m,\mu}(y).
\]

On the other hand, for any $z\in U \smallsetminus \mathcal{Z}(m,\mu)(\C)$ we have  $Q(\lambda_z)<0$ for every
$\lambda\in \Lambda_{m,\mu}$, and so
\begin{align}\label{CT}
\sum_{  \lambda\in \Lambda_{m,\mu} }
\mathrm{CT}_{s=0}
\left[
\int_1^\infty e^{4\pi v Q(\lambda_z) } \, \frac{dv}{v^{s+1}}
\right]
&  =
\sum_{  \lambda\in \Lambda_{m,\mu} }  \int_1^\infty e^{4\pi v Q(\lambda_z) } \, \frac{dv}{v}  \\
& =
\sum_{  \lambda\in \Lambda_{m,\mu} } \Gamma(0 , -4\pi Q(\lambda_z)), \nonumber
\end{align}
where
\begin{equation}\label{gamma}
\Gamma(0,x) = \int_x^\infty  e^{-t} \, \frac{dt}{t} =  -\gamma-\log(x) -  \sum_{k=1}^\infty  \frac{ (-x)^k}{k\cdot k!}.
\end{equation}

Comparing (\ref{CT}) with (\ref{gamma}) shows that
\[
 \sum_{  \lambda\in \Lambda_{m,\mu} } \log| 4\pi e^\gamma \cdot Q(\lambda_z) | =
- \sum_{  \lambda\in \Lambda_{m,\mu} }
\mathrm{CT}_{s=0}
\left[
\int_1^\infty e^{4\pi v Q(\lambda_z) } \, \frac{dv}{v^{s+1}}
\right]
\]
up  to a  smooth function on $U$, vanishing at $z=y$.  Adding this equality to (\ref{magic integral})  proves that
\[
\Phi^\mathrm{reg}_{m,\mu} (z) = \varphi_{m,\mu} (z)
\]
up to a smooth function vanishing at $z=y$, and completes the proof.
\end{proof}

For any $y\in\mathcal{Y}(\C)$, we have a canonical isomorphism (see ~\ref{eqn:CompNormalExplicit})
\begin{equation}\label{normal fiber}
( N_\mathcal{Y} \mathcal{M} )_y \iso (\C y)^\vee \otimes_\C \Lambda_\C=\Hom(\C y,\Lambda_\C).
\end{equation}
For each $\lambda\in \Lambda_{m,\mu}$, define a function $\Phi_\lambda$ on $\Hom(\C y,\Lambda_\C)$ by
\[
\Phi_\lambda ( \xi) =  -  \log  \left|  \frac{ 4\pi e^\gamma  [\lambda, \xi( y )] ^2}{  [ y  , \bar{ y } ]} \right|  .
\]
This function has a logarithmic singularity along the hyperplane $(\C y)^\vee \otimes \lambda^\perp$.
Letting $y$ vary over  $\mathcal{Y}(\C)$ yields  a function $\Phi_\lambda$ on the vector bundle
$(N_\mathcal{Y} \mathcal{M})(\C)$ having   a logarithmic singularity along a sub-bundle of hyperplanes.

\begin{proposition}\label{prop:green normalize}
The  Green function $\sigma(\Phi_{m,\mu})$ defined by (\ref{green specialization})  satisfies
\[
\sigma\big(\Phi_{m,\mu}\big) =
\pi^*i^*\Phi_{m,\mu}  + \sum_{\lambda\in \Lambda_{m,\mu}}\Phi_\lambda,
\]
where $\pi^*i^*\Phi_{m,\mu}$ denotes the pullback of  $\Phi_{m,\mu}$ via
\[
(N_\mathcal{Y} \mathcal{M})(\C)  \map{\pi} \mathcal{Y}(\C) \map{i} \mathcal{M}(\C).
\]
\end{proposition}

\begin{proof}
As above, fix a point $y\in \mathcal{Y}(\C)$ and let $U\subset \mathcal{D}$ be a neighborhood
of the isotropic line $\C y \subset L_\C$. As in~\eqref{ss:normal bundle}, we can define a holomorphic local immersion
\[
 \Hom(\C y,\Lambda_\C)=(\C y)^\vee \otimes  \Lambda_\C \to\mathrm{M}^{\mathrm{loc}}(L_{\C})
\]
by sending $\xi$ to the $\C$-span of
\[
 y +\xi( y  )- \frac{ Q(\xi( y )) \overline{ y  }} { [ y ,\overline{ y }] } .
 \]
As we saw at the end of~\eqref{ss:normal bundle}, the induced isomorphism on tangent spaces
\[
(\C y)^\vee \otimes \Lambda_\C  \to T_y\mathcal{D} \iso (N_\mathcal{Y}\mathcal{M})_y
\]
agrees with the isomorphism of (\ref{normal fiber}).

From now on we use the above map to identify $U$ with an open  neighborhood of the origin of
$ (\C y)^\vee \otimes \Lambda_\C$.  Each divisor $\mathcal{D}(\lambda)\cap U$ appearing in
(\ref{small analytic divisor})  is  identified  with the zero locus of
\[
f_\lambda( \xi )=  [\lambda ,\xi( y  ) ]  \cdot  \left| \frac{4\pi e^\gamma}{ [ y  ,\overline{ y  }] }  \right|^{1/2}   .
\]
This function is already linear, and so when we apply the construction $f_\lambda\mapsto \sigma(f_\lambda)$
of the discussion preceding (\ref{green specialization}) and identify $ (\C y)^\vee \otimes \Lambda_\C$
with its own tangent space at the origin, we obtain
\[
\sigma(f_\lambda)( \xi)=   [  \lambda ,  \xi(\bm{v})] \cdot  \left| \frac{4\pi e^\gamma}{ [\bm{v},\overline{\bm{v}}] }  \right|^{1/2}   .
\]
Thus  $-\log |\sigma(f_\lambda) |^2  = \Phi_\lambda.$

If we define a smooth function on $U$ by
\[
\phi_{m,\mu}= \Phi_{m,\mu} + \sum_{\lambda\in \Lambda_{m,\mu} } \log |f_\lambda|^2,
\]
then directly from the definition (\ref{green specialization}) and the paragraph above we see that
\[
\sigma(\Phi_{m,\mu}) = \pi^* i^* \phi_{m,\mu} + \sum_{\lambda\in \Lambda_{m,\mu} }  \Phi_\lambda.
\]
Now use  $\phi_{m,\mu} (y) = \Phi^\mathrm{reg}_{m,\mu} (y) =  \Phi_{m,\mu}(y)$
to complete the proof.
\end{proof}


\subsection{The calculation of the pullback}


Recall  that the Kuga-Satake abelian scheme $\mathcal{A} \to \mathcal{M}$ restricts to the abelian scheme
\[
\mathcal{A}_\mathcal{Y}=\mathcal{E} \otimes_{C^+(L_0)} C(L)
\]
over $\mathcal{Y}$, where $\mathcal{E} \to \mathcal{Y}$ is the universal elliptic curve.

By Proposition \ref{prop:0_decomistionVmu}, every
$
\lambda\in \Lambda_{m,\mu}
$
determines a canonical special endomorphism $\bm{\ell}_\lambda \in V_\mu(\mathcal{A}_\mathcal{Y})$, and the pair
 $(\mathcal{A}_\mathcal{Y} , \bm{\ell}_\lambda)$ is a $\mathcal{Y}$-valued point of $\mathcal{Z}(m,\mu)$.  The corresponding map
$\mathcal{Y} \to  \mathcal{Z}(m,\mu)$ is finite and unramified, and so    we may define an  $\co_\mathcal{Y}$-submodule
\begin{equation}\label{Z-lambda}
\mathcal{Z}_\lambda =N_\mathcal{Y} \mathcal{Z}(m,\mu)
\end{equation}
of $N_\mathcal{Y} \mathcal{M}$ in exactly the same way that  $N_\mathcal{Y} \mathcal{M}$ was defined.

Proposition \ref{prop:0_functorial} provides us with a canonical morphism
$\Lambda \to i^*\bm{V}_{\dR}$ such that $\Lambda^{\perp}=\bm{V}_{0,\dR}\subset i^*\bm{V}_{\dR}$. Thus, the bilinear pairing on $i^*\bm{V}_{\dR}$ gives us a map
\[
 i^*\bm{V}_{\dR}/\bm{V}_{0,\dR} \to \co_\mathcal{Y}\otimes\Lambda^{\vee}
\]
of vector bundles over $\mathcal{Y}$.
This map is injective, and using local trivializations, one sees that its image is identified precisely with the inclusion
\[
 \co_{\mathcal{Y}}\otimes(L/L_0)\hookrightarrow\co_\mathcal{Y}\otimes\Lambda^\vee.
\]
This latter inclusion is the natural one induced by the bilinear pairing on $L$.

Now, by definition, $\Lambda_{m,\mu}$ is contained in $(L/L_0)^{\vee} \subset L^\vee$. Therefore, we obtain a morphism:

\[
[\, \cdot\, , \lambda ] :  i^*\bm{V}_{\dR}  / \bm{V}_{0 , \dR}   \to \co_\mathcal{Y}\otimes(L/L_0) \xrightarrow{1\otimes\lambda}\co_{\mathcal{Y}}.
\]
of $\co_\mathcal{Y}$-modules.

Here is another way to describe this morphism. Recall from \S \ref{ss:CM to GSpin} that  we have a canonical $T$-torsor over $\mathcal{Y}$ parameterizing $\co_\kk$-equivariant trivializations of $\bm{H}^+_{0,\dR}$. Using this torsor of local trivializations, we get a canonical functor from representations of $T$ to vector bundles over $\mathcal{Y}$ with integrable connection.

The map $[\, \cdot\, ,\lambda]$ is now simply obtained by applying this functor to the $T$-invariant functional $L/L_0\to \Z$ given by pairing with $\lambda$. In particular, if $t_{\dR}(\bm{\ell}_{\lambda})$ is the de Rham realization of $\bm{\ell}_{\lambda}$ (it is a section of $i_S^*\bm{V}^\vee_{\dR}$), we find that  functionals
\[
i_S^*\bm{V}_{\dR} \to i_S^*\bm{V}_{\dR}/\bm{V}_{0,\dR}\xrightarrow{[\cdot,\lambda]}\co_S
\]
and
\[
[\, \cdot\, , t_{\dR}(\bm{\ell}_{\lambda})]:i_S^*\bm{V}_{\dR} \to\co_S
\]
are identical.

\begin{proposition}\label{prop:taut rep}
Fix  any  $\lambda\in \Lambda_{m,\mu}$.
\begin{enumerate}
\item
The divisor $\mathcal{Z}_\lambda$ is canonically identified with the kernel of
\[
N_\mathcal{Y}\mathcal{M}   \xrightarrow{\eqref{eqn:AlmostNormal}} \bm{T}_0 \otimes  ( i^*\bm{V}_{\dR}  / \bm{V}_{0 , \dR} )
  \map{  [\cdot, \lambda ]} \bm{T}_0.
\]
\item
The metrized line bundle
$
\pi^* \widehat{\bm{T}}_0
$
is represented by the arithmetic divisor
\[
 ( \mathcal{Z}_\lambda , \Phi_\lambda ) \in  \widehat{\mathrm{Div}} ( N_\mathcal{Y} \mathcal{M} ).
\]
\end{enumerate}
\end{proposition}

\begin{proof}
We work over a sufficiently small \'etale neighborhood $Y\subset U$ of $\mathcal{Y}\to\mathcal{M}$. Suppose that we are given an open subspace $S\subset Y$ and $j_S\in\mathrm{Def}_i(S)$. By the construction of the map~\eqref{eqn:AlmostNormal}, to prove (1), we have to show that $\bm{\ell}_{\lambda}$ lifts to an element in $V_{\mu}(j_S^*\mathcal{A})$ if and only if the composition
\[
 j_S^*\bm{V}^1_{\dR}\to i_S^*\bm{V}_{\dR} / \bm{V}_{0,\dR}\otimes_{\co_S}\co_{S[\varepsilon]}\xrightarrow{[\, \cdot\,  ,\lambda]}\co_{S[\varepsilon]}
\]
is identically $0$.  This follows from Lemma~\ref{lem:localstrZmmualpha} and the identification we made above of $[\, \cdot\, ,\lambda]$ with pairing against $t_{\dR}(\bm{\ell}_{\lambda})$.

The second claim follows from the first, as the map
$
[\, \cdot\,  ,\lambda] : N_\mathcal{Y}\mathcal{M}  \to  \bm{T}_0
$
defines a section $s_\lambda$ of $\pi^*\bm{T}_0$ with
$
\widehat{\mathrm{div}} (s_\lambda) =  ( \mathcal{Z}_\lambda , \Phi_\lambda ).
$
 Indeed, using the notation of (\ref{normal fiber}),
 the norm of this section  at a  complex point $\xi \in    (N_\mathcal{Y}\mathcal{M})_y$
in the fiber over $y\in\mathcal{Y}(\C)$ is
\[
-\log || s_\lambda    ||^2   = -\log \left|  \frac{ [ \lambda , \xi( y  )] }{ || y ||}  \right|^2   = \Phi_\lambda(\xi),
\]
where the second equality is by definition (\ref{taut metric}) of  the metric on $\widehat{\bm{T}}_0$.
\end{proof}

  As in \cite[\S 7.3]{BHY}, we define a new metrized line bundle
\[
\widehat{\mathcal{Z}}^\heartsuit(m,\mu)
= \widehat{\mathcal{Z}}(m,\mu) \otimes \widehat{\bm{T}}^{ \otimes - \# \Lambda_{m,\mu}}
\in \widehat{\mathrm{Pic}}(\mathcal{M}).
\]

\begin{proposition}\label{prop:cutimproper}
The image  of $\widehat{\mathcal{Z}}^\heartsuit(m,\mu)$  under the pullback
\[
 \widehat{\mathrm{Pic}}  ( \mathcal{M} )  \map{i^*} \widehat{\mathrm{Pic}}   ( \mathcal{Y} )
\]
 is represented by the  arithmetic divisor
\[
\Big(\sum_{
\substack{ m_1>0 \\ m_2\geq 0 \\ m_1+m_2=m }}
\sum_{\substack{ \mu_1\in L_0^\vee/L_0 \\ \mu_2\in \Lambda^\vee/\Lambda\\
 (\mu_1,\mu_2) \in   (\mu + L)/(L_0 \oplus \Lambda) }
}
 R_\Lambda(m_2,\mu_2)  \cdot \mathcal{Z}_0(m_1 , \mu_1 ) \ , \  i^* \Phi_{m,\mu}
  \Big) \in \widehat{\mathrm{Div}}(\mathcal{Y})
\]
\end{proposition}

\begin{proof}
By  Proposition \ref{prop:0_decomistionZmu}  there is a decomposition of $\mathcal{Y}$-stacks
\[
 \mathcal{Z}( m , \mu ) \times_\mathcal{M}   \mathcal{Y}
  = \bigsqcup_{
\substack{ m_1,m_2\in \Q_{\ge 0} \\ m_1+m_2=m }}
\bigsqcup_{\substack{ \mu_1\in L_0^\vee/L_0 \\ \mu_2\in \Lambda^\vee/\Lambda\\
 (\mu_1,\mu_2) \in   (\mu + L)/(L_0 \oplus \Lambda) }
}
 \mathcal{Z}_0(m_1 , \mu_1 ) \times \Lambda_{m_2,\mu_2}.
\]
  Using Remark \ref{rmk:properandimproper} we may therefore decompose
\[
\mathcal{Z}( m , \mu ) \times_\mathcal{M}   \mathcal{Y} = \mathcal{Z}_0^\mathrm{prop}\,
\sqcup  ( \mathcal{Y} \times \Lambda_{m,\mu} )  ,
\]
where
\[
\mathcal{Z}_0^\mathrm{prop} =
\bigsqcup_{
\substack{ m_1 > 0 \\ m_2 \ge 0 \\ m_1+m_2=m }}
\bigsqcup_{\substack{ \mu_1\in L_0^\vee/L_0 \\ \mu_2\in \Lambda^\vee/\Lambda\\
 (\mu_1,\mu_2) \in   (\mu + L)/(L_0 \oplus \Lambda) }
}
 \mathcal{Z}_0(m_1 , \mu_1 ) \times \Lambda_{m_2,\mu_2}
\]
is a $\mathcal{Y}$-stack of dimension $0$.   Moreover, the induced map
\[
 \mathcal{Y} \times \Lambda_{m,\mu}  \to \mathcal{Z}( m , \mu )
\]
determines, for each $\lambda\in \Lambda_{m,\mu}$, a map $\mathcal{Y} \to \mathcal{Z}(m,\mu)$,
which is none other than the map used in the definition (\ref{Z-lambda}) of $\mathcal{Z}_\lambda$.

Over a sufficiently small \'etale open chart $U\to \mathcal{M}$ we may fix a connected component
$Y\subset \mathcal{Y}_U$, and decompose
\[
\mathcal{Z}(m,\mu)_U = Z^\mathrm{prop}\, \sqcup (Y \times \Lambda_{m,\mu})
\]
where $Z^\mathrm{prop}$ is the union of all connected components $Z\subset  \mathcal{Z}(m,\mu)_U$ for which $Z\cap Y$ has dimension
$0$.  If we apply the specialization to the normal bundle construction of (\ref{divisor specialization})
separately to $Z^\mathrm{prop}$ and to each of the $\# \Lambda_{m,\mu}$ copies of $Y$,
and then glue over an \'etale cover, we find the  equality
\begin{equation}\label{specialization decomp}
\sigma(\mathcal{Z}(m,\mu) \big)
= \pi^* \mathcal{Z}_0^\mathrm{prop}  + \sum_{\lambda\in \Lambda_{m,\mu}} \mathcal{Z}_\lambda
\end{equation}
of  divisors on $N_\mathcal{Y}\mathcal{M}$.

By Propositions \ref{prop:poor mans Hu} and \ref{prop:taut rep}
the image of $\widehat{\mathcal{Z}}^\heartsuit(m,\mu)$ under
\[
 \widehat{\mathrm{Pic}}  ( \mathcal{M} )  \map{i^*} \widehat{\mathrm{Pic}}   ( \mathcal{Y} ) \map{\pi^*}
  \widehat{\mathrm{Pic}}   ( N_\mathcal{Y} \mathcal{M} )
\]
is
\[
\pi^* i^* \widehat{\mathcal{Z}}^\heartsuit(m,\mu)=
\sigma\big( \widehat{\mathcal{Z}} (m,\mu) \big) - \sum_{ \lambda \in \Lambda_{m,\mu} } ( \mathcal{Z}_\lambda, \Phi_\lambda) ,
\]
and Proposition \ref{prop:green normalize} and (\ref{specialization decomp}) allow us to rewrite this
as an equality of metrized line bundles
\[
\pi^* i^* \widehat{\mathcal{Z}}^\heartsuit(m,\mu)=
\pi^* \big(
 \mathcal{Z}_0^\mathrm{prop}    ,  i^* \Phi_{m,\mu}
\big)
\]
on $N_\mathcal{Y}\mathcal{M}$.  Pulling back by the zero section $\mathcal{Y} \to N_\mathcal{Y}\mathcal{M}$
yields  an isomorphism
\[
 i^* \widehat{\mathcal{Z}}^\heartsuit(m,\mu)= (  \mathcal{Z}_0^\mathrm{prop}   ,  i^* \Phi_{m,\mu} )
\]
of metrized line bundles on $\mathcal{Y}$.
\end{proof}


\subsection{The main result}


We are now ready to prove our main result.  For the reader's convenience, we  review the \emph{dramatis personae}.

The lattices $L_0\oplus \Lambda \subset L$ determine a finite and unramified morphism of stacks $\mathcal{Y} \to \mathcal{M}$ over $\co_\kk$, and hence a linear
functional
\[
[\cdot: \mathcal{Y}] : \widehat{\mathrm{CH}}^1(\mathcal{M}) \to \R
\]
defined in \S \ref{ss:metrizedbundles}.

Fix a  weak harmonic Maass form $f(\tau)\in H_{1-n/2}(\omega_L)$.  The  holomorphic part
\[
f^+(\tau) = \sum_{m\gg -\infty} c^+_f(m)  \cdot q^m
\]
is a formal
$q$-expansion valued in the finite-dimensional  vector space $\mathfrak{S}_L$ of complex functions on $L^\vee/L$.
We assume that $f^+$ has integral principal part, so that the constructions
of \S \ref{ss:green functions}  provide us with an arithmetic divisor
\[
\widehat{\mathcal{Z}}(f) = (\mathcal{Z}(f) , \Phi(f) ) \in \widehat{\mathrm{CH}}^1(\mathcal{M}).
\]
We also have, from  \S \ref{ss:cotaut},  the metrized cotautological bundle
\[
\widehat{\bm{T}} \in  \widehat{\mathrm{CH}}^1(\mathcal{M}).
\]

Recalling the Bruinier-Funke differential operator
\[
\xi: H_{1 - \frac{n}{2}}(\omega_L) \to  S_{1+\frac{n}{2}}(\overline{\omega}_L)
\]
from (\ref{xi op}), we may form the convolution $L$-function $L(  \xi(f) ,\Theta_\Lambda,s)$ of $\xi(f)$ with the theta series
$
\Theta_\Lambda(\tau) \in M_{\frac{n}{2}}(\omega_\Lambda^\vee)
$
of \S \ref{ss:analytic}.
The convolution $L$-function vanishes at $s=0$, the center of the functional equation.

Recall  the $\mathfrak{S}_{L_0\oplus \Lambda}$-valued formal $q$-expansion
$\mathcal{E}_{L_0} \otimes \Theta_\Lambda$ of \S \ref{ss:analytic}.
Using (\ref{transfer}) and (\ref{pairing}),  we obtain a  scalar-valued $q$-expansion
\[
 \{ f^+ , \mathcal{E}_{L_0} \otimes \Theta_\Lambda \}
 =\sum_{m_1,m_2,m_3 \in \Q} \big\{ c_f^+(m_1) , a_{L_0}^+(m_2) \otimes R_\Lambda(m_3) \big\} \cdot q^{m_1+m_2+m_3}
 \]
  with constant term
\begin{equation}\label{const term}
\mathrm{CT} \{ f^+ , \mathcal{E}_{L_0} \otimes \Theta_\Lambda \}
=
\sum_{m_1+m_2+m_3=0} \big\{ c_f^+(m_1) , a_{L_0}^+(m_2) \otimes R_\Lambda(m_3) \big\}.
\end{equation}

The following  \emph{CM value formula} was proved by Schofer \cite{Sch} for weakly holomorphic modular forms, and then generalized  to harmonic weak Maass forms
by  Bruinier-Yang \cite{BY}.  We repeat here the statement of \cite[Theorem 4.7]{BY}, with a sign error corrected.
See also \cite[Theorem 5.3.6]{BHY}.

\begin{theorem}[Schofer, Bruinier-Yang]
\label{thm:green calc}
The Green function $\Phi(f)$ satisfies
\[
\frac{w_\kk}{h_\kk} \sum_{ y\in \mathcal{Y}(\C)} \frac{\Phi(f,y) }{ \#\Aut(y) }
=
- L'( \xi (f)  ,  \Theta_\Lambda,0) +  \mathrm{CT} \{ f^+ , \mathcal{E}_{L_0} \otimes \Theta_\Lambda \} ,
\]
where \[\xi: H_{1 - \frac{n}{2}}(\omega_L) \to  S_{1+\frac{n}{2}}(\overline{\omega}_L)\] is the Bruinier-Funke
differential operator of (\ref{xi op}).
\end{theorem}

\begin{remark}
Recalling Remark \ref{rem:miracle value}, Theorem \ref{thm:green calc}
 holds even when some points of $\mathcal{Y}(\C)$ lie on $\mathcal{Z}(f)(\C)$,
the divisor along which  $\Phi(f)$ has its logarithmic singularities.
\end{remark}

Inspired by Theorem \ref{thm:green calc}, the following formula  was conjectured by Bruinier-Yang \cite{BY}.

\begin{theorem}\label{thm:mainthm}
Assume $d_\kk$ is odd.  Every  weak harmonic Maass form $f\in H_{1-n/2}(\omega_L)$ with integral principal part  satisfies
\[
[ \widehat{\mathcal{Z}} (f) : \mathcal{Y} ]  + c_f^+(0,0)  \cdot [ \widehat{\bm{T}}  : \mathcal{Y} ]
= - \frac{h_\kk}{w_\kk}  \cdot L'(  \xi(f) ,\Theta_\Lambda,0) ,
\]
where  $c_f^+(0,0)$ is the value of   $c_f^+(0) \in \mathfrak{S}_L$ at the trivial coset in $L^\vee/L$.
\end{theorem}

\begin{proof}
First  assume $f(\tau)=   F_{m,\mu}(\tau)$ is the Hejhal-Poincar\'e series of \S \ref{ss:green functions}, so that
\[
\widehat{\mathcal{Z}}(f)   = \big(\mathcal{Z}( m ,\mu ) , \Phi_{m,\mu} \big)
 = \widehat{\mathcal{Z}}^\heartsuit (m,\mu )\otimes \widehat{\bm{T}} ^{ \otimes R_\Lambda(m,\mu) }.
\]

By  Proposition   \ref{prop:cutimproper}  there is a decomposition
\begin{eqnarray*}
[ \widehat{\mathcal{Z}}^\heartsuit(m,\mu)  : \mathcal{Y} ]
&  = &
\sum_{  \substack{ m_1+m_2=m \\ m_1>0   \\    (\mu_1,\mu_2) \in   (\mu + L)/(L_0 \oplus \Lambda) }  }
 R_\Lambda(m_2,\mu_2) \cdot \widehat{\deg}\, \mathcal{Z}_0(m_1 , \mu_1 ) \\
& & +   \sum_{ y\in \mathcal{Y}(\C) } \frac{ \Phi_{m,\mu}(y) } { \#\Aut(y) } ,
\end{eqnarray*}
where we  view
$
 \mathcal{Z}_0(m_1 , \mu_1 ) \in \widehat{\mathrm{Div}}(\mathcal{Y})
$
as an arithmetic divisor with trivial Green function.
Theorem \ref{thm:X degree} shows that the first term is
\begin{eqnarray*}\lefteqn{
 \sum_{  \substack{ m_1+m_2=m \\ m_1>0   \\    (\mu_1,\mu_2) \in   (\mu + L)/(L_0 \oplus \Lambda) }  }
 R_\Lambda(m_2,\mu_2) \cdot \widehat{\deg}\, \mathcal{Z}_0(m_1 , \mu_1 )   } \\
& = &
-\frac{h_\kk}{w_\kk}  \sum_{  \substack{ m_1+m_2=m \\ m_1>0   \\    (\mu_1,\mu_2) \in   (\mu + L)/(L_0 \oplus \Lambda) }  }
  a^+_{L_0}(m_1,\mu_1) R_\Lambda(m_2,\mu_2) \\
  & = &
  -\frac{h_\kk}{w_\kk}  \sum_{ \substack{ m_1+m_2=m  \\ m_1>0 }}  \big\{ c_f^+(-m) , a^+_{L_0}(m_1) \otimes R_\Lambda(m_2) \big\}.
\end{eqnarray*}
For our special choice of $f(\tau)$  the constant term (\ref{const term}) simplifies to
\begin{eqnarray*}
\mathrm{CT} \big\{ f^+ , \mathcal{E}_{L_0} \otimes \Theta_\Lambda \big\}
& = &
\big\{ c_f^+(0) , a_{L_0}^+(0) \otimes R_\Lambda(0) \big\}  \\
& &   + \sum_{  m_1+m_2 = m    } \big\{ c_f^+(-m) ,  a_{L_0}^+(m_1) \otimes R_\Lambda(m_2)\big\}
\end{eqnarray*}
and so Theorem \ref{thm:green calc}  becomes
\begin{align*}
 \sum_{ y\in \mathcal{Y}(\C) } \frac{ \Phi_{m,\mu}(y) } { \#\Aut(y) }
 &  =
- \frac{h_\kk}{w_\kk}  \cdot L'(  \xi(f) , \Theta_\Lambda,0)
+  \frac{h_\kk}{w_\kk} \big\{ c_f^+(0) , a_{L_0}^+(0) \otimes R_\Lambda(0) \big\}  \\
& \quad   + \frac{h_\kk}{w_\kk}   \sum_{  m_1+m_2 = m    } \big\{ c_f^+(-m) ,  a_{L_0}^+(m_1) \otimes R_\Lambda(m_2)\big\} .
\end{align*}
We have now proved
\begin{align*}
[ \widehat{\mathcal{Z}}^\heartsuit(m,\mu)  : \mathcal{Y} ]
 & =  - \frac{h_\kk}{w_\kk}  \cdot L'(  \xi(f) , \Theta_\Lambda,0)
   +    \frac{h_\kk}{w_\kk} \cdot   \big\{ c_f^+(0) , a_{L_0}^+(0) \otimes R_\Lambda(0) \big\}   \\
&\quad +  \frac{h_\kk}{w_\kk}   \big\{ c_f^+(-m) , a^+_{L_0}( 0 ) \otimes R_\Lambda(m)  \big\}  .
\end{align*}

Corollary \ref{cor:cotaut degree}  implies both
\[
 c_f ^+(0,0) \cdot [ \widehat{\bm{T}} : \mathcal{Y} ]
 = - \frac{h_\kk}{w_\kk} \cdot  \big\{ c_f ^+(0) , a^+_{L_0}(0) \otimes R_\Lambda(0) \big\},
\]
and
\[
 R_\Lambda(m,  \mu) \cdot [ \widehat{\bm{T}} : \mathcal{Y} ]
 = - \frac{h_\kk}{w_\kk} \cdot   \big\{ c_f^+(-m) ,  a^+_{L_0}(0) \otimes R_\Lambda(m) \big\} ,
\]
leaving
\[
[\widehat{\mathcal{Z}}^\heartsuit (m,\mu) : \mathcal{Y} ]  +
  [ \widehat{\bm{T}}^{\otimes R_\Lambda(m,  \mu) } : \mathcal{Y} ]
+ c_f^+(0,0)  [ \widehat{\bm{T}} : \mathcal{Y} ]
= - \frac{h_\kk}{w_\kk} \cdot L'(  \xi(f) ,\Theta_\Lambda,0),
\]
and proving the desired formula in the special case of $f=F_{m,\mu}$.

Now assume that $f$ satisfies $f^+(\tau)=O(1)$.  In particular, $\mathcal{Z}(f) = \emptyset$
and the Green function $\Phi(f)$ is a smooth function on
$\mathcal{M}(\C)$.  The arithmetic intersection is purely  archimedean, and Theorem \ref{thm:green calc} shows that
\begin{align*}
[ \widehat{\mathcal{Z}}(f)  : \mathcal{Y} ]  &=
 \sum_{ y\in \mathcal{Y}(\C) } \frac{ \Phi(f,y) } { \#\Aut(y) } \\
& =
- \frac{h_\kk}{w_\kk}  \cdot L'( \xi( f ) ,\Theta_\Lambda,0)
+  \frac{h_\kk}{w_\kk} \cdot \mathrm{CT} \big\{ f^+ , \mathcal{E}_{L_0} \otimes \Theta_\Lambda \big\}  \\
& =
- \frac{h_\kk}{w_\kk}  \cdot L'( \xi( f ) ,\Theta_\Lambda,0)
+  \frac{h_\kk}{w_\kk} \cdot  \big\{ c_f^+(0) , a^+_{L_0}(0) \otimes R_\Lambda(0) \big\} .
\end{align*}
Corollary \ref{cor:cotaut degree} implies that
\[
\frac{h_\kk}{w_\kk} \cdot   \big\{  c_f^+(0) , a^+_{L_0}(0) \otimes R_\Lambda(0)  \big\}
 =  \frac{h_\kk}{w_\kk}  \cdot   c_f^+(0,0) \cdot a^+_{L_0} (0,0)
 = - c_f^+(0,0) \cdot [ \widehat{\bm{T}}  : \mathcal{Y}]  ,
\]
completing the proof in  this  case.

Finally, every weak harmonic Maass form can be written as a linear combination of the Hejhal-Poincar\'e series, and a form
with holomorphic part $f^+(\tau) = O(1)$. Thus the claim follows from the linearity in $f$ of both sides of the
desired equality.
\end{proof}

\bibliographystyle{amsalpha}
\providecommand{\bysame}{\leavevmode\hbox to3em{\hrulefill}\thinspace}

\end{document}